\documentclass[a4paper,reqno,11pt]{amsart}

\usepackage{amsmath, amsfonts, amssymb, amsthm, amscd}
\usepackage{graphicx}
\usepackage{psfrag}
\usepackage{perpage}
\usepackage{url}
\usepackage{color}
\usepackage{mathrsfs}
\usepackage{tikz}
\usepackage{mathabx}

\usepackage{dsfont} % nice indicator function symbol (blackboard 1)

\usepackage[utf8]{inputenc}
\usepackage[T1]{fontenc}
\usepackage{microtype}

\usepackage[a4paper,scale={0.72,0.74},marginratio={1:1},footskip=7mm,headsep=10mm]{geometry}

\usepackage{hyperref}

\makeatletter
\def\@secnumfont{\bfseries\scshape}

\def\section{\@startsection{section}{1}%
  \z@{.7\linespacing\@plus\linespacing}{.5\linespacing}%
  {\normalfont\large\bfseries\scshape\centering}}

\def\subsection{\@startsection{subsection}{2}%
  \z@{.5\linespacing\@plus.7\linespacing}{-.5em}%
  {\normalfont\bfseries\scshape}}

\def\subsubsection{\@startsection{subsubsection}{3}%
  \z@{.5\linespacing\@plus.7\linespacing}{-.5em}%
  {\normalfont\scshape}}

\def\specialsection{\@startsection{section}{1}%
  \z@{\linespacing\@plus\linespacing}{.5\linespacing}%
  {\normalfont\centering\large\bfseries\scshape}}
\makeatother

\makeatletter

\renewenvironment{proof}[1][\proofname]{\par
\pushQED{\qed}%
\normalfont \topsep4\p@\@plus4\p@\relax
\trivlist
\item[\hskip\labelsep
\bfseries
#1\@addpunct{.}]\ignorespaces
}{%
\popQED\endtrivlist\@endpefalse
}
\makeatother

\makeatletter
\newcommand \Dotfill {\leavevmode \leaders \hb@xt@ 6pt{\hss .\hss }\hfill \kern \z@}
\makeatother

\makeatletter
\def\@tocline#1#2#3#4#5#6#7{\relax
  \ifnum #1>\c@tocdepth % then omit
  \else
    \par \addpenalty\@secpenalty\addvspace{#2}%
    \begingroup \hyphenpenalty\@M
    \@ifempty{#4}{%
      \@tempdima\csname r@tocindent\number#1\endcsname\relax
    }{%
      \@tempdima#4\relax
    }%
    \parindent\z@ \leftskip#3\relax \advance\leftskip\@tempdima\relax
    \rightskip\@pnumwidth plus4em \parfillskip-\@pnumwidth
    #5\leavevmode\hskip-\@tempdima
      \ifcase #1
       \or\or \hskip 1.65em \or \hskip 3.3em \else \hskip 4.95em \fi%
      #6\nobreak\relax
    \Dotfill
    \hbox to\@pnumwidth{\@tocpagenum{#7}}\par
    \nobreak
    \endgroup
  \fi}
\makeatother

\makeatletter
\def\l@section{\@tocline{1}{0pt}{1pc}{}{\scshape}}
\renewcommand{\tocsection}[3]{%
\indentlabel{\@ifnotempty{#2}{\ignorespaces#1 #2.\hskip 0.7em}}#3}
\def\l@subsection{\@tocline{2}{0pt}{1pc}{5pc}{}}

\def\l@subsubsection{\@tocline{3}{0pt}{1pc}{7pc}{}}

\makeatother

\numberwithin{equation}{section}

\newtheoremstyle{mytheorem}{.7\linespacing\@plus.3\linespacing}{.7\linespacing\@plus.3\linespacing}%
     {\itshape}%         Body font
     {}%         Indent amount (empty = no indent, \parindent = para indent)
     {\bfseries}% Thm head font (e.g. \bfseries, \scshape, \sffamily)
     {. }%        Punctuation after thm head
     {0.3ex}%     Space after thm head (\newline = linebreak)
     {\thmname{{\bfseries #1}}\thmnumber{ {\bfseries #2}}\thmnote{ (#3)}}  % Thm head spec

\theoremstyle{mytheorem}

\newtheorem{theorem}{Theorem}[section]
\newtheorem{lemma}[theorem]{Lemma}
\newtheorem{proposition}[theorem]{Proposition}

\newtheorem{remark}[theorem]{Remark}
\newtheorem{definition}[theorem]{Definition}

\newtheorem{hypothesis}[theorem]{Hypothesis}

\newcommand{\bbE}{{\ensuremath{\mathbb E}} }

\newcommand{\bbN}{{\ensuremath{\mathbb N}} }

\newcommand{\bbP}{{\ensuremath{\mathbb P}} }

\newcommand{\bbR}{{\ensuremath{\mathbb R}} }

\newcommand{\bbT}{{\ensuremath{\mathbb T}} }

\newcommand{\bbZ}{{\ensuremath{\mathbb Z}} }

\newcommand{\cL}{{\ensuremath{\mathcal L}} }

\newcommand{\ga}{\alpha}
\newcommand{\gb}{\beta}
\newcommand{\gl}{\lambda}

\newcommand{\go}{\omega}

\renewcommand{\tilde}{\widetilde}          % wider `tilde'
\DeclareMathSymbol{\leqslant}{\mathalpha}{AMSa}{"36} % nicer `smaller or equal'
\DeclareMathSymbol{\geqslant}{\mathalpha}{AMSa}{"3E} % nicer `larger or equal'
\DeclareMathSymbol{\eset}{\mathalpha}{AMSb}{"3F}     % nicer `emptyset'
\newcommand{\sumtwo}[2]{\sum_{\substack{#1 \\ #2}}} % sum with 2 lines
\newcommand{\sumthree}[3]{\sum_{\substack{#1 \\ #2 \\ #3}}} % sum with 3 lines
\newcommand{\asto}[1]{\underset{{#1}\to\infty}{\longrightarrow}}
\newcommand{\astoo}[2]{\underset{{#1}\to{#2}}{\longrightarrow}}

\newcommand{\R}{\mathbb{R}}

\newcommand{\Z}{\mathbb{Z}}
\newcommand{\N}{\mathbb{N}}

\newcommand{\T}{\mathbb{T}}

\def\bs{\boldsymbol}

\newcommand{\PEfont}{\mathrm}

\newcommand{\p}{\ensuremath{\PEfont P}}
\DeclareMathOperator{\Ci}{\ensuremath{\PEfont C}}

\DeclareMathOperator{\e}{\ensuremath{\PEfont E}}

\newcommand\Ro{\ensuremath{\bs{\mathrm{R}}}}

\newcommand{\E}{\e}
\renewcommand{\P}{\p}

\DeclareMathOperator{\bbvar}{\ensuremath{\mathbb{V}ar}}
\DeclareMathOperator{\bbcov}{\ensuremath{\mathbb{C}ov}}

\newcommand{\ind}{\mathds{1}}

\newcommand{\eps}{\varepsilon}
\renewcommand{\epsilon}{\varepsilon}
\renewcommand{\theta}{\vartheta}
\renewcommand{\rho}{\varrho}

\newenvironment{myenumerate}{%
\renewcommand{\theenumi}{\arabic{enumi}}%
\renewcommand{\labelenumi}{{\rm(\theenumi)}}%
\begin{list}{\labelenumi}
	{%
	\setlength{\itemsep}{0.4em}%
	\setlength{\topsep}{0.5em}%
	\setlength\leftmargin{2.45em}%
	\setlength\labelwidth{2.05em}%
	\setlength{\labelsep}{0.4em}%
	\usecounter{enumi}%
	}%
	}%
{\end{list}
}

\newenvironment{aenumerate}{%
\renewcommand{\theenumi}{\alph{enumi}}%
\renewcommand{\labelenumi}{{\rm(\theenumi)}}%
\begin{list}{\labelenumi}
	{%
	\setlength{\itemsep}{0.4em}%
	\setlength{\topsep}{0.5em}%
	\setlength\leftmargin{2.45em}%
	\setlength\labelwidth{2.05em}%
	\setlength{\labelsep}{0.4em}%
	\usecounter{enumi}%
	}%
	}%
{\end{list}
}

{\end{list}
}

\renewenvironment{enumerate}{
\begin{myenumerate}}%
{\end{myenumerate}}

\newenvironment{myitemize}{%
\begin{list}{$\bullet$}%
 	{%
	\setlength{\itemsep}{0.4em}%
	\setlength{\topsep}{0.5em}%
	\setlength\leftmargin{2.65em}%
	\setlength\labelwidth{2.65em}%
	\setlength{\labelsep}{0.4em}%
	}%
	}%
{\end{list}}

\renewenvironment{itemize}{
\begin{myitemize}}%
{\end{myitemize}}

\MakePerPage[2]{footnote} % restarts footnote counter at every new page

\date{\today}

\newcommand\dd{\mathrm{d}}

\newcommand\hbeta{{\hat{\beta}}}

\newcommand\sfY{\mathsf Y}

\newcommand\sfX{\mathsf X}

\newcommand\sfp{\mathsf p}
\newcommand\sft{\mathsf t}

\newcommand\bi{\boldsymbol{i}}

\newcommand\ba{\boldsymbol{a}}
\newcommand\bb{\boldsymbol{b}}
\newcommand\bc{\boldsymbol{c}}
\newcommand\bd{\boldsymbol{d}}

\newcommand\bn{\boldsymbol{n}}
\newcommand\bx{\boldsymbol{x}}
\newcommand\by{\boldsymbol{y}}
\newcommand\bz{\boldsymbol{z}}
\newcommand\bw{\boldsymbol{w}}

\newcommand\fin{\mathrm{fin}}

\begin{document}

\title[Universality in marginally relevant disordered systems]
{Universality in marginally relevant\\ disordered systems}

\begin{abstract}
We consider disordered systems of directed polymer type, for which disorder is so-called
marginally relevant. These include the
usual (short-range) directed polymer model in dimension $(2+1)$, the long-range directed
polymer model with Cauchy tails in dimension $(1+1)$ and
the disordered pinning model with tail exponent $1/2$. We show that in a suitable weak disorder
and continuum limit, the partition functions of these different
models converge to a universal limit:
a log-normal random field with a multi-scale correlation structure, which undergoes a phase transition
as the disorder strength varies. As a by-product, we show that the solution of the
two-dimensional Stochastic Heat Equation, suitably regularized, converges to the same limit.
The proof, which uses the celebrated Fourth Moment Theorem, reveals an interesting
chaos structure shared by all models in the above class.
 \end{abstract}

\author[F.Caravenna]{Francesco Caravenna}
\address{Dipartimento di Matematica e Applicazioni\\
 Universit\`a degli Studi di Milano-Bicocca\\
 via Cozzi 55, 20125 Milano, Italy}
\email{francesco.caravenna@unimib.it}

\author[R.Sun]{Rongfeng Sun}
\address{Department of Mathematics\\
National University of Singapore\\
10 Lower Kent Ridge Road, 119076 Singapore
}
\email{matsr@nus.edu.sg}

\author[N.Zygouras]{Nikos Zygouras}
\address{Department of Statistics\\
University of Warwick\\
Coventry CV4 7AL, UK}
\email{N.Zygouras@warwick.ac.uk}


\maketitle

\tableofcontents

\addtocounter{section}{0}

\section{Introduction}
\label{sec:intro}

Many disordered systems arise as random perturbations of a pure (or homogeneous) model.
Examples include the random pinning model~\cite{G07}, where the pure system is a renewal process,
the directed polymer model~\cite{CSY04}, where the pure system is a directed random walk, the
random field Ising model~\cite{B06} and the stochastic heat equation~\cite{BC95}. A
fundamental question for such systems is:
{\em Does addition of disorder alter the
qualitative behavior of the pure model, such as its
large-scale properties and/or critical exponents?}
\smallskip

If the answer is yes, regardless of how small the disorder strength is, then the model
is called {\em disorder relevant}. If, on the other hand,
disorder has to be strong enough to cause a qualitative change,
then the model is called {\em disorder irrelevant}.
This difference
can be understood heuristically via renormalization transformations~\cite{B06, G10}:
if one rescales space (coarse graining) and looks at the resulting renormalized
disordered system on larger
and larger spatial scales, then one will observe that the ``effective'' strength of disorder will asymptotically diverge if disorder
is relevant, while it will vanish if disorder is irrelevant.
\smallskip

Whether a model is disorder relevant or irrelevant depends crucially on the spatial dimension $d$ and
its correlation length exponent $\nu$. A milestone in the study of disordered systems in
the physics literature is the Harris criterion~\cite{H74}, which asserts that if $d<2/\nu$,  then disorder
is relevant, while if $d>2/\nu$, then it is irrelevant.
In the critical case $d=2/\nu$, disorder is
marginal and the Harris criterion is inconclusive: disorder can be either {\em marginally relevant}
or {\em marginally irrelevant} depending on the finer details of the model.
\smallskip

Inspired by the study of an intermediate disorder regime for directed polymers \cite{AKQ14a},
we proposed in~\cite{CSZ13}
a new perspective on disorder relevance. The key observation is that, if
a model is disorder relevant, then it is possible to tune the strength of disorder down to zero
(weak disorder limit) at the same time as one rescales space
(continuum limit), so as to obtain a
one-parameter family of \emph{disordered continuum models},
indexed by a macroscopic
disorder strength parameter $\hat\beta \ge 0$.
In a sense, such continuum models interpolate between the
scaling limit of the pure model ($\hat\beta = 0$)
and the scaling limit of the original disordered model $(\hat\beta = \infty)$,
allowing one to study the onset of the effect of disorder.
\smallskip

The main step in the construction of such disordered continuum models is to identify their
\emph{partition functions}. In~\cite{CSZ13}, we formulated general conditions on the pure model
that are consistent with the Harris criterion $d < 2/\nu$ for disorder relevance,
which allowed us to construct explicitly the continuum partition functions. However, the {\em marginally relevant} case ($d=2/\nu$ in the Harris criterion) escapes the framework proposed in~\cite{CSZ13}.

In the present work, we develop a novel approach to study the continuum limit of marginally relevant systems of directed polymer type, which include
the usual short-range directed polymer model on $\Z^2$,
the long-range directed polymer model on $\Z$ with Cauchy tails, and
the pinning model with tail exponent $\alpha=1/2$.
We show that, surprisingly, there is a common underlying structure among all these marginally relevant models (see Section \ref{ss:heur} and Key Proposition \ref{thm:main1}), which leads to a number of universal phenomena. More precisely,

\begin{itemize}
\item A properly defined
{\em replica overlap} $\Ro_{N}$ for each model diverges as a slowly varying function
(usually a logarithm)
of the polymer length $N\to\infty$.

\item If the disorder strength is sent to $0$
as $\beta_N=\hbeta/\sqrt{\Ro_{N}}$ for fixed $\hbeta>0$, then
the partition function has a universal limit in distribution, {\em irrespective of the model}:
\begin{equation} \label{eq:limit}
	Z_{N,\beta_N}
	\xrightarrow[\,N\to\infty\,]{d}
	\bs{Z}_{\hat\beta}
	\overset{d}{=} \begin{cases}
	\text{log-normal} & \text{ if } \hat \beta < 1 \\
	0 & \text{ if } \hat\beta \ge 1
	\end{cases} \,.
\end{equation}
with the log-normal variable depending on the parameter $\hat\beta$.
\item A process-level version of \eqref{eq:limit} also holds:
for $\hat\beta < 1$, the family of log partition functions $\log Z_{N, \beta_N}(x)$,
indexed by the starting point $x$ of the polymer, converges to a limiting Gaussian
random field (depending on $\hat\beta$) with an explicit \emph{multi-scale covariance structure}.

\end{itemize}

The transition from a non-degenerate limit $\bs{Z}_\hbeta >0$ to a degenerate limit
$\bs{Z}_\hbeta=0$, as $\hbeta$ increases,
marks a transition from {\em weak disorder} to
{\em strong disorder}.
We emphasize that such a transition for marginally relevant models, in particular, the $(2+1)$-dimensional directed polymer,
is new and has not been anticipated. Previously, it was only known (see e.g.~\cite{CSY04}) that for the directed polymer in dimension $d+1$,
there is a transition from weak to strong disorder at a critical $\beta_c(d)$, with $\beta_c(d)>0$ when $d\geq 3$ (corresponding to disorder
irrelevance) and $\beta_c(d)=0$ when $d=1,2$ (corresponding to disorder relevance).
(For $d=2$ the polymer was shown in~\cite{F12}
to be diffusive if $\beta_N \ll 1/\sqrt{\Ro_{N}}$.)

Interestingly, our results show that in the marginal dimension $d=2$, there is still a transition on the finer scale of $\beta=\hat\beta/\sqrt{\Ro_N}$, with critical value $\hat\beta_c=1$. This appears to be a special
feature of marginality, since no such transition exists at any finer scale of disorder in dimension $d=1$~\cite{AKQ14a}.

Another point worth remarking is that the explicitly identified critical point $\hbeta=1$ is actually the point where the $L^2$ norm of the partition functions blow up in the limit. This is in contrast to the directed polymer in dimension $d+1$ with $d\geq 3$ (see e.g.\ \cite{BS10}),
or the log-correlated Gaussian Multiplicative Chaos \cite{RV13} which also undergoes a weak to strong disorder transition.
For these two models, their critical points are strictly larger than their respective $L^2$ critical points.

\smallskip

Our results unify different polymer models that are classified as marginally relevant.
However, beyond this universality, even more interesting is the method we develop, which reveals a multi-scale and Gaussian chaos structure that is common to all the models we consider. In particular, the partition functions can be approximated by a sum of stochastic integrals involving
white noises in all possible dimensions, which through resummation, can be seen as the exponential of a Gaussian (see Section~\ref{ss:proofsteps} for an outline of the main steps). The key technical ingredients include a non-trivial combinatorial argument (Proposition~\ref{thm:main1}), and the application of a version of the Fourth Moment Theorem~\cite{dJ90,NP05} for Gaussian approximation.
\medskip

An interesting corollary of our results is that they link marginal relevant models to a class of singular SPDEs at the {\em critical dimension}.
In particular, they bring new insights on how to define the solution of the
{\em two-dimensional} Stochastic Heat Equation (2d SHE), which is formally written as
\begin{equation}\label{she0}
\frac{\partial u(t,x)}{\partial t} =
\frac{1}{2}\Delta u(t,x) + \beta \, \dot{W}(t,x) \, u(t,x), \qquad u(0, \cdot) \equiv 1,
\end{equation}
for $(t,x) \in [0,\infty) \times \R^2$,
$\beta>0$ and $\dot{W}$ is the space-time white noise.
\smallskip

Rigorously defining the solution of \eqref{she0} remains a difficult open problem
due to ill-defined terms such as $\dot{W}u$. In special cases,
such as the one-dimensional SHE, it was shown in \cite{BC95} that a solution
can be defined by first mollifying $\dot{W}$
and then sending the mollification parameter to zero.
But there was no systematic approach
to make sense of singular SPDEs until recent breakthroughs by
Hairer~\cite{H13, H14}, through {\em Regularity Structures},
and by Gubinelli, Imkeller and Perkowski~\cite{GIP15},
through \emph{Paracontrolled Distributions} (see also Kupiainen~\cite{K14} for
field theoretic approach).
However, these approaches do not cover the {\em critical} dimension two for the SHE,
and the singular SPDEs that can be treated so far are all known as {\em sub-critical}
(or {\em super-renormalizable} in the physics literature~\cite{K14}).

\smallskip

It turns out that the notion of {\em sub-criticality} for singular SPDEs
matches with the notion of disorder relevance, while {\em criticality} corresponds to the
case where the effect of disorder is marginal.
To illustrate this fact for the SHE, consider the change of variables
\begin{equation*}
	(t,x) = T_\eps(\tilde t, \tilde x) := (\eps^{-2}\tilde t, \eps^{-1}\tilde x) \,,
\end{equation*}
which for small $\eps > 0$ corresponds to a space-time coarse graining
transformation.
Looking at \eqref{she0},
it is easily seen that $\tilde u(\tilde t, \tilde x):=u(T_\eps(\tilde t, \tilde x))$ formally
solves the SPDE
\begin{equation}\label{she1}
\frac{\partial \tilde u}{\partial \tilde t}
= \frac{1}{2}\Delta \tilde u + \beta \eps^{\frac{d}{2}-1}  \dot{\tilde W}  \tilde u, \qquad \tilde u(0,\cdot)\equiv 1,
\end{equation}
where $\dot{\tilde W}$ is a new space-time White noise obtained from $\dot W$ via scaling.
Therefore, coarse-graining space-time for the SHE has the effect of changing the strength of the
noise to $\eps^{\frac{d}{2}-1}\beta$ which, as $\eps\to 0$, diverges for $d=1$
(disorder relevance), vanishes for $d\geq 3$ (disorder irrelevance), and remains
unchanged for $d=2$ (marginality).

Since the difficulties in studying the regularity
properties of an SPDE are related to small scale
divergences, it is interesting to \emph{blow up}
space-time, i.e.\ consider the change of variables $(t,x) = T_{\epsilon^{-1}}(\tilde t, \tilde x)$.
This leads to a renormalized equation which is just \eqref{she1}
with $\epsilon$ replaced by $\epsilon^{-1}$, hence blowing-up space time
produces an effective noise strength which
behaves reciprocally with respect to coarse-graining,
i.e.\ vanishes as $\epsilon \to 0$ for $d=1$ and diverges for $d \ge 3$.
This explains why the SHE with $d=1$ can be analyzed by~\cite{H14,GIP15,K14}.

\smallskip

Since the solution of the SHE can be interpreted as the partition function of a continuum directed polymer via a generalized Feynman-Kac formula~\cite{BC95},  our result for the two-dimensional directed polymer implies a similar result for the 2d SHE. More precisely, if we consider the mollified 2d SHE
\begin{equation}\label{she2}
\frac{\partial u^\eps}{\partial t} = \frac{1}{2}\Delta u^\eps + \beta_\eps \dot{W}^\eps  u^\eps , \qquad u^\eps(0, \cdot)  \equiv 1,
\end{equation}
where $\dot{W}^\eps$ is the space-mollification of $\dot{W}$ via convolution with a smooth probability density $j_\eps(x):= \eps^{-2}j(x/\eps)$ on $\R^2$, and the noise strength is scaled as $\beta_\eps= \hbeta \sqrt{\frac{2\pi}{\log \eps^{-1}}}$ for some $\hbeta>0$, then for each $(t,x)\in (0,\infty)\times \R^2$, $u^\eps(t,x)$ converges (as $\eps\to 0$) in distribution to the same universal
limit $\bs{Z}_{\hat\beta}$ in \eqref{eq:limit} as for the other marginally relevant models.
\smallskip

We hope that the method we develop and the universal structure we have uncovered
opens the door to further understanding of marginally relevant models in general, including
both statistical mechanics models that are not of directed polymer type,
as well as critical singular SPDEs with non-linearity. In particular, our results suggest that
for marginally relevant models there is a transition in the effect of disorder on an
intermediate disorder scale. Establishing this transition in general, as well as understanding
the behavior of the models at and above the transition point, will be the key challenges next.

 \section{The models and our results}
\label{sec:results}

In this section we define our models of interest and state our main results.
We will denote $\N := \{1,2,3,\ldots\}$ and $\N_0 := \N \cup \{0\}$.
\smallskip

\subsection{The models}\label{ss:models}

We first introduce the disorder $\omega$. Let
$\omega = (\omega_{{}_\sfX})$ be a family of i.i.d.\ random variables, indexed by $\sfX\in\N$ or
$\sfX= (x,n) \in \Z^d \times \N_0$,
depending on the model. Probability and expectation for $\omega$ will be
denoted respectively by $\bbP$ and $\bbE$. We assume that
\begin{equation*}
	\bbE[\omega_1] = 0 \,, \qquad \ \bbvar[\omega_1] = 1 \,,
	\qquad \ \exists \beta_0 > 0: \ \ \lambda(\beta) :=
	\log \bbE[e^{\beta \omega_1}] < \infty \ \ \forall |\beta| < \beta_0 \,.
\end{equation*}

Next we define the class of models we consider.
We fix a reference probability law $\P$ (which will typically be the law of a random walk
or a renewal process) representing the ``pure'' model.
The disordered model is then a Gibbs perturbation
$\P_{N,\beta}^\omega$ of $\P$, indexed by the parameters $N \in \N$ (polymer length),
$\beta \ge 0$ (disorder strength) and the disorder $\omega$:
\begin{equation*}
	\frac{\dd \P_{N,\beta}^\omega}{\dd\P}(\cdot)
	:= \frac{e^{H_{N,\beta}^\omega(\cdot)}}{Z_{N,\beta}^\omega}
\end{equation*}
for a suitable Hamiltonian $H_{N,\beta}^\omega$. The normalizing constant
\begin{equation*}
	Z_{N,\beta}^\omega := \E\left[e^{H_{N,\beta}^\omega} \right]
\end{equation*}
is the \emph{disordered partition function} and will be the focus of this paper.

Different reference laws $\P$ and Hamiltonians $H_{N,\beta}^\omega$ give rise to different models. The first class of models we will consider are \emph{directed polymers in random environment} on $\Z^{d+1}$.

\begin{definition}[Directed polymers on $\Z^{d+1}$]\label{def:dp}\rm
Let $S = (S_n)_{n\in\N_0}$ be a random walk on $\Z^d$ with i.i.d.\ increments.
For $(x,t) \in \Z^d \times \N_0$ we denote by $\P_{x,t}$ the law of
$(S_n)_{n \ge t}$ started at $x$ at time $t$, and we denote $\P:= \P_{0,0}$ for simplicity.
The partition function of the
\emph{directed polymer in random environment}
is defined by
\begin{equation}\label{eq:Zdp}
	Z_{N,\beta}^\omega(x,t) := \E_{x,t} \left[
	e^{\sum_{n=t+1}^N \left( \beta \, \omega_{(n,S_n)} - \lambda(\beta) \right)} \right]
\end{equation}
with $Z_{N,\beta}^\omega := Z_{N,\beta}^\omega(0,0)$.
\end{definition}

We will also consider \emph{pinning models}, which can be viewed as directed polymers on $\Z^{d+1}$ with
disorder present only at $x=0$ (i.e.\ $\omega(n,x) = 0$ for $x \ne 0$). In this case, what really matters are the return times of the random walk $S$
to $0$, which form a \emph{renewal process}.

\begin{definition}[Pinning models]\label{def:pinning}\rm
Let $(\tau = (\tau_n)_{n\in\N_0}, \P_t)$ be a renewal process started
at $t\in\N_0$, i.e., $\P_t(\tau_0 = t) = 1$
and $(\tau_n - \tau_{n-1})_{n\in\N}$ are i.i.d.\ $\N$-valued random variables.
If $t=0$ we write $\P = \P_0$. The partition function of the
\emph{pinning model} started at $t\in\N_0$ equals
\begin{equation}\label{eq:Zpin}
	Z_{N,\beta}^{\omega}(t) :=
	\E_t \left[ e^{\sum_{n=t+1}^N
	\left( \beta \omega_n - \lambda(\beta) \right) \ind_{\{n \in \tau\}}} \right] \,,
\end{equation}
with $Z_{N,\beta}^{\omega} := Z_{N,\beta}^{\omega}(0)$,
where we have identified $\tau$ with the random set $\{\tau_0, \tau_1, \ldots\}\subset \N_0$.
\end{definition}

\begin{remark}\rm
In the pinning model it is customary to have a
bias parameter $h \in \R$, i.e.\ $-\lambda(\beta)$ is replaced by $-\lambda(\beta)+h$
in \eqref{eq:Zpin}. In this paper we set $h=0$ because in the regime we are interested
in, the effects of $\beta$ and $h$ can be decoupled. This will be treated elsewhere.
\end{remark}

Note that $Z_{N,\beta}^\omega$ in \eqref{eq:Zdp}--\eqref{eq:Zpin} has been normalized so that
$\bbE [Z_{N,\beta}^\omega]= 1$
(due to $-\lambda(\beta)$).
The key question we consider (in connection with disorder relevance) is the following:

\begin{enumerate}
\renewcommand{\theenumi}{\textbf{Q.}}%
\renewcommand{\labelenumi}{{\theenumi}}
\item\label{enu} \emph{Can one tune the disorder
strength $\beta = \beta_N \to 0$ as $N\to\infty$ in such a way that the partition function
$Z_{N,\beta_N}^\omega$ converges in law to a non-degenerate random variable?}
\end{enumerate}
The answer depends crucially on the random walk $S$ and the renewal process
$\tau$. Assume that $S$ and $\tau$ are in the domain of attraction of a stable law, with respective
index $\alpha \in (0,2]$ and $\alpha \in (0,1)$.  Informally,
this means that $\P(|S_1| > n) \approx n^{-\alpha}$ and $\P(\tau_1 > n) \approx n^{-\alpha}$
(except for $\alpha = 2$, where $\bbE[|S_1|^2] < \infty$
or, more generally, $x \mapsto \E[|S_1|^2 \ind_{\{|S_1|\le x\}}]$ is slowly varying).
\smallskip

It was shown in~\cite{CSZ13} that question \ref{enu} has an affirmative answer for directed polymers on $\Z^{1+1}$ with $\alpha \in (1, 2]$
and for pinning models with $\alpha \in (1/2, 1)$,  which is a manifestation of \emph{disorder relevance}; while disorder is irrelevant for directed polymers on $\Z^{1+1}$ with $\alpha\in (0,1)$
and for pinning models with $\alpha \in (0, 1/2)$. However, the \emph{marginal cases}
\begin{aenumerate}
\item\label{it:a} directed polymers on $\Z^{2+1}$ with $\alpha = 2$ (e.g., finite variance);
\item\label{it:b} directed polymers on $\Z^{1+1}$ with $\alpha = 1$ (e.g., Cauchy tails);
\item\label{it:c}
pinning models with tail exponent $\alpha = 1/2$ (e.g., the renewal arising from the return times of the simple symmetric random walk on $\Z$ to the origin),
\end{aenumerate}
fall out of the scope of the method in~\cite{CSZ13}.

In this paper, we develop a novel approach to answer question \ref{enu}
affirmatively for marginally relevant models.
Even though our techniques are of wider applicability, we stick for simplicity
to models of type \eqref{it:a}--\eqref{it:c} above.
Let us state our precise assumptions, in the form of local limit theorems,
where we allow for arbitrary slowly varying function $L(\cdot)$. However, we suggest to
keep in mind the basic case when $L(\cdot)$ is constant, say $L(\cdot) \equiv 1$.

\begin{hypothesis}[Local Limit Theorem]\label{hypoth}
\rm\label{hyp}
Assume that the directed polymer in Definition~\ref{def:dp} and the pinning model in Definition~\ref{def:pinning} satisfy the following local limit theorems, for some slowly varying function $L$.
\begin{aenumerate}
\item {[$d=2$]} \emph{Directed polymer on $\Z^{2+1}$ with $\alpha = 2$ $($short range$)$.}

\begin{equation} \label{eq:llt2}
	\sup_{x\in\Z^2} \left\{ L(n)^2 n \, \P(S_n = x) -
	g\bigg(\frac{x}{L(n)\sqrt{n}} \bigg) \right\}
	\xrightarrow[n\to\infty]{} 0 \,,
\end{equation}
where $g(x) := \frac{1}{2\pi} e^{-\frac{1}{2}|x|^2}$ denotes the standard Gaussian density
on $\R^2$.

\item {[$d=1$]}
\emph{Directed polymer on $\Z^{1+1}$ with $\alpha = 1$ $($long-range with Cauchy tails$)$.}

\begin{equation} \label{eq:llt1}
	\sup_{z\in\Z} \left\{ L(n)^2 n \, \P(S_n = x) - g\bigg(\frac{x}{L(n)^2 n} \bigg) \right\}
	\xrightarrow[n\to\infty]{} 0 \,,
\end{equation}
where $g(x) := \frac{1}{\pi} \frac{1}{1+x^2}$ denotes the Cauchy density on $\R$.

\item {[$d=0$]} \emph{Pinning model with $\alpha = \frac{1}{2}$.}
\begin{equation} \label{eq:llt0}
	L(n) \sqrt{n} \; \P(n\in\tau) \xrightarrow[n\to\infty]{} c \in (0,\infty) \,.
\end{equation}

\end{aenumerate}
\end{hypothesis}

\begin{remark}\rm\label{rem:suff}
Conditions \eqref{eq:llt2}--\eqref{eq:llt1}
hold whenever $S$ is an aperiodic random walk on $\Z^d$
in the domain of attraction of the Gaussian ($d=2$), resp.\ Cauchy ($d=1$) distribution:
\begin{equation}\label{eq:CLT}
	\cL\Big(\frac{S_n}{\phi(n)}\Big) \ \xrightarrow[\ n\to\infty\ ]{weak} \ g(x) \, \dd x
	\qquad \quad \text{with} \qquad
	\phi(n) := \big(L(n)^2 \, n\big)^{1/d} \,,
\end{equation}
by Gnedenko's
local limit theorem, cf.~\cite[Theorem 8.4.1]{cf:BinGolTeu}
(we denote by $\cL(\cdot)$ the law of a random variable).

Condition \eqref{eq:llt0} holds whenever $\P(\tau_1 = n) \sim c' \frac{L(n)}{n^{3/2}}$
as $n\to\infty$~\cite[Thm.~B]{D97}.
\end{remark}

\begin{remark}[2d Simple random walk]\rm\label{rem:ape}
When $S$ is the simple symmetric random walk on $\Z^2$, due to periodicity, \eqref{eq:llt2} still holds
(with $L(\cdot ) \equiv 1$)
provided the $\sup$ is restricted to the sub-lattice $z \in \{(a,b) \in \Z^2: a+b = n \ (\text{mod } 2)\}$
(whose cells have area $2$) and $g(\cdot)$ is replaced by $2g(\cdot)$. Consequently, relation \eqref{eq:RN} holds with $C = 2 \|g\|_2^2$. Our main results Theorems~\ref{thm:subcritical},
\ref{thm:short-cor} and~\ref{thm:field} below apply with no further change.
\end{remark}

A crucial common feature among all models \eqref{it:a}--\eqref{it:c} above concerns the so-called expected {\em replica overlap},
defined for a general random walk $S$ or renewal process $\tau$ by
\begin{equation} \label{eq:replica}
	\Ro_N := \begin{cases}
	\displaystyle \, \E\Big[\sum_{n=1}^N \ind_{\{S_n = S'_n\}} \Big] = \sum_{1 \le n \le N, \, x \in \Z^d}
	\!\! \P(S_n = x)^2 & \\
	\rule{0pt}{1.6em}
	\displaystyle \E\Big[ \sum_{n=1}^N \ind_{\{n \in \tau \cap \tau'\}} \Big]= \quad \sum_{1 \le n \le N}
	\P(n \in \tau)^2 &
	\end{cases} \,,
\end{equation}
where $S'$ and $\tau'$ are independent copies of $S$ and $\tau$.
For models satisfying Hypothesis~\ref{hypoth}, a Riemann sum approximation using
\eqref{eq:llt2}--\eqref{eq:llt0} yields
\begin{equation} \label{eq:RN}
	\Ro_N \underset{N\to\infty}{\sim} C \, \sum_{n=1}^N \frac{1}{L(n)^2 n} \,,
	\qquad \ \text{where} \qquad
	C = \begin{cases}
	\|g\|_2^2 & \text{(directed polymers)}\\
	c^2 & \text{(pinning)}
	\end{cases} \,.
\end{equation}
This shows that \emph{$N \mapsto \Ro_N$ is a slowly varying function},
cf.\ \cite[Proposition 1.5.9a]{cf:BinGolTeu}, a fact which plays a crucial in our analysis.
Whether $\Ro_N$ stays bounded or diverges as $N\to\infty$
will determine whether disorder is relevant or irrelevant. This leads to

\begin{definition}[Marginal overlap condition]\label{def:margrel}\rm
A directed polymer or a pinning model
is said to satisfy the {\em marginal overlap condition}, if $\Ro_N \to \infty$
as a slowly varying function when $N\to\infty$, where $\Ro_N$ is defined in \eqref{eq:replica}.
\end{definition}

\noindent
Under Hypothesis~\ref{hypoth},
the {\em marginal overlap condition} is satisfied when $\Ro_N\to \infty$,
which by \eqref{eq:RN} holds if $L(n)$ stays bounded, or more generally,
does not grow too fast as $n\to\infty$.
We suggest the reader to keep in mind the basic case $L(n) \equiv 1$,
for which $\Ro_N \sim C \log N$.

\smallskip

Our main result, to be stated in the next subsection, is that question \ref{enu} has an
affirmative answer for models of directed polymer type
which satisfy Hypothesis \ref{hypoth} and the marginal overlap condition.
This is a signature of marginal disorder relevance in the spirit of \cite{CSZ13}.
The recent results of Berger and Lacoin
\cite{BL15a,BL15b} on free energy and critical curves
reinforce this picture.

\medskip

\subsection{Results for directed polymer and pinning models}
\label{sec:main}

We are now ready to state our main results: Theorem~\ref{thm:subcritical} on the convergence of partition function with a fixed starting point; Theorem~\ref{thm:short-cor} on the joint limit of partition functions with different starting points, where multi-scale correlations emerge; and Theorem~\ref{thm:field} on the Gaussian fluctuations of the partition functions as a random field indexed by the starting points.

\begin{theorem}[Limit of partition functions]\label{thm:subcritical}
Let $Z_{N,\beta}^\omega$ be the partition function of
a directed polymer or a pinning model
(cf.\ Definitions~\ref{def:dp} and~\ref{def:pinning}).
Assume that Hypothesis~\ref{hyp} holds
and the replica overlap $\Ro_N$ in \eqref{eq:replica} and \eqref{eq:RN}
diverges as $N\to\infty$. Then, defining
\begin{equation}\label{eq:betascaling}
	\beta_N := \frac{\hat\beta}{\sqrt{\Ro_N}} \,, \qquad
	\text{with} \quad \hat\beta \in (0,\infty) \,,
\end{equation}
the following convergence in distribution holds:
\begin{equation} \label{eq:conve}
	Z_{N,\beta_N}^\omega \xrightarrow[\,N\to\infty\,]{d}
	\bs{Z}_{\hat\beta} :=
	\begin{cases}
	\exp\bigg(\sigma_{\hat\beta} \, W_1 - \frac{\sigma_{\hat\beta}^2}{2}
	\bigg)
	& \text{if } \hat \beta < 1 \\
	0 & \text{if } \hat\beta \ge 1
	\end{cases} \,.
\end{equation}
where $W_1$ is a standard Gaussian random variable and
\begin{equation}\label{eq:sigmahatbeta}
	\sigma_{\hat\beta}^2 := \log \frac{1}{1-\hat\beta^2} \,.
\end{equation}
Moreover, for $\hat\beta < 1$ one has $\lim_{N\to\infty}\bbE[(Z_{N,\beta_N}^\omega)^2]
= \bbE[(\bs{Z}_{\hat\beta})^2]$.
\end{theorem}
\begin{remark}\rm
Note that for $\hat\beta < 1$, $\bs{Z}_{\hat\beta}$ is log-normal.
Let $(W_t)_{t\ge 0}$ be a standard Brownian motion. We will in fact prove that
\begin{equation}\label{eq:conve0}
	Z_{N,\beta_N}^\omega
	\xrightarrow[\,N\to\infty\,]{d}
	\displaystyle
	\exp\bigg( \displaystyle\int_0^1\textstyle
	\frac{\hat\gb }{\sqrt{1-\hat\gb^2 \,t}} \, \dd W_t-
	\frac{1}{2} \displaystyle\int_0^1\textstyle
	\frac{\hat\gb^2\ }{1-\hat\gb^2\,t} \, \dd t \bigg) \overset{d}{=} \bs{Z}_{\hat\beta}
	 \,.
\end{equation}
This more involved expression for $\bs{Z}_{\hat\beta}$ hints at a remarkable underlying
multi-scale and chaos structure, which is common to
all models that satisfy Hypothesis~\ref{hyp} and the marginal overlap condition.
The heuristics for this structure will be
explained in Sec.~\ref{heuristics}.

It is even possible to identify the limiting distribution \emph{of
the whole process $(Z_{N,\beta_N}^\omega)_{\hbeta\in (0,1)}$}.
Denoting by $(W^{(r)}_t)_{t \ge 0, \, r \in \N}$ a countable family of independent
Brownian motions, we have the convergence
in distribution of $Z_{N,\beta_N}^\omega$ as $N\to\infty$,
jointly for $\hbeta \in (0,1)$, to the process
\begin{equation}
	\prod_{r=1}^\infty \exp\bigg( \displaystyle \int_0^1
	\hbeta^r \, t^{\frac{r-1}{2}} \, \dd W^{(r)}_t-
	\tfrac{1}{2} \int_0^1
	\hbeta^{2r} \, t^{r-1} \, \dd t \bigg)
	\overset{d}{=} \bs{Z}_{\hat\beta} \,.
\end{equation}
This can be extracted from the proof of Lemma \ref{prop:limCG}, and we will omit the details.

It is worth noting the non-trivial dependence
of $\sigma_{\hat\gb}^2$ on $\hat\beta$, cf.\ \eqref{eq:sigmahatbeta}. On the one hand it
distinguishes from other scalings such as $\gb_N \ll 1/\sqrt{\Ro_N}$, which lead to
a trivial behavior, and on the other hand it marks the transition
from weak ($\bs{Z}_{\hat\beta} > 0$) to strong ($\bs{Z}_{\hat\beta} = 0$) disorder.
\end{remark}

\begin{remark}\rm
During the completion of this paper, Alberts, Clark and
Koci\'c showed in \cite{AKS15} that for the
marginally relevant directed polymer model on the diamond hierarchical lattice, with either edge or
site disorder, there is also a transition for the partition
function in an intermediate disorder regime
with some critical value $\hbeta_c$.
Their proof relies on the recursive structure of the hierarchical lattice.
A difference with respect to our results is that,
for $\hbeta\leq \hbeta_c$, the partition function converges to $1$ and has Gaussian fluctuations.
It would be interesting to apply our
approach to better understand the source of this difference.
\end{remark}

\begin{remark}\rm
One may wonder whether the assumption of finite exponential moments
$\bbE[e^{\beta \omega_1}] < \infty$ can be relaxed. Indeed,
for the usual (short-range) directed polymer model in dimension $d=1$,
in the intermediate disorder regime
it is enough to assume finite six moments, as conjectured in
Alberts-Khanin-Quastel \cite{AKQ14a}
and proved by Dey-Zygouras \cite{DZ16}.
The heuristic in dimension $d=1$ is that if $\bbP(\omega_1 > t) \sim t^{-a}$,
the typical maximum of the disorder random variables visited
by the random walk by time $N$ is $N^{\frac{3}{2a}}$.
The intermediate disorder scaling in dimension
$d=1$ is $\beta_N=\hat \beta N^{-1/4} $, so one has
$\beta_N N^\frac{3}{2a} \to 0$ when
$a>6$, allowing for a truncation argument.
In dimension $d=2$, the typical maximum is
$N^{2/a}$, while $\beta_N=\hat \beta/\sqrt{\log N}$,
so $\beta_N N^{2/a} \to \infty$ irrespective of $a$.
This suggests that in the critical dimension
$d=2$, things
are more subtle and we are reluctant to make any  claim.

\end{remark}

\smallskip

Next we study the partition functions
$Z_{N,\beta}^\omega(\sfX)$
as a random field,
indexed by the polymer's starting position
$\sfX = (x,t) \in \Z^d \times \N_0$ with $d\in \{1,2\}$
(for directed polymers), resp.\ $\sfX = t \in \N_0$ (for pinning models).
Assuming Hypothesis \ref{hypoth} with a slowly varying $L(\cdot)$
and a divergent overlap $\Ro_N$, and recalling \eqref{eq:CLT}, we define
\begin{equation}\label{eq:phiarrow}
	\phi^{\leftarrow}(|x|) \; := \; \min\big\{n\in\N_0: \ \phi(n) \geq |x| \big\}
	\; = \; \min\big\{n\in\N_0: \ (nL(n)^2)^{1/d}\geq |x| \big\} \,.
\end{equation}
By \eqref{eq:CLT}, $\phi^{\leftarrow}(|x|)$ is
the time at which the random walk $S$ has a fluctuation of order $|x|$.
Then, for each $\sfX=(x,t)\in \Z^d \times \N_0$ with $d\in \{0,1,2\}$, we set
\begin{equation}\label{Xtripnorm}
	\vvvert \sfX \vvvert := \begin{cases}
	t & \text{if } d=0 \\
	t \vee \phi^{\leftarrow}(|x|)  & \text{if } d=1,2	\end{cases}\,.
\end{equation}
We suggest to keep in mind the special case $L(n) \equiv 1$,
for which $\vvvert \sfX \vvvert = t \vee |x|^d$.

Theorem~\ref{thm:subcritical} gives the limiting distribution of
the individual partition functions $Z_{N,\beta_N}^\omega(\sfX)$,
and it is natural to ask about the joint distributions.
In the special case $L(n) \equiv 1$, i.e.\ $\Ro_N\sim C\log N$,
partition functions $Z_{N,\beta_N}^\omega(\sfX)$
and $Z_{N,\beta_N}^\omega(\sfX')$
with macroscopically distant starting points
$\vvvert \sfX - \sfX' \vvvert = N^{1+o(1)}$
become \emph{asymptotically independent}
as $N\to\infty$, while
an interesting correlation structure emerges \emph{on all intermediate scales}
$\vvvert \sfX - \sfX' \vvvert= N^{\zeta+o(1)}$,
for any $\zeta \in (0,1)$.
For general slowly varying functions $L(n)$, when $\Ro_N$ is not necessarily
logarithmic, intermediate scales
are encoded by $\Ro_{\vvvert \sfX - \sfX' \vvvert}/\Ro_N =\zeta+o(1)$.
This is the content of the next theorem,
where we use the shorthand notation $:e^Y: \ = e^{Y - \frac{1}{2}\bbvar[Y]}$.

\begin{theorem}[Multi-scale correlations]\label{thm:short-cor}
Let $Z_{N,\beta}^\omega(\sfX)$ be the partition function of
a directed polymer $($or pinning$)$ model started at
$\sfX=(x,t) \in \Z^d \times \N_0$ (cf.\ Def.~\ref{def:dp}
and~\ref{def:pinning}), such that Hypothesis~\ref{hyp} holds
and the replica overlap $\Ro_N$ in \eqref{eq:replica}--\eqref{eq:RN}
diverges as $N\to\infty$.

Consider a finite collection of space-time points $(\sfX_N^{(i)})_{1 \le i \le r}$, such that
as $N\to\infty$,
\begin{equation} \label{eq:condi}
	\begin{split}
	\forall\, 1\leq k,l\leq r \,: \qquad
	& \Ro_{N-t_N^{(k)}}/\Ro_N =1-o(1) \,, \\
	& \Ro_{\vvvert \sfX_N^{(k)} - \sfX_N^{(l)} \vvvert}/\Ro_N = \zeta_{k,l}+o(1)
	\ \ \text{for some} \ \
	\zeta_{k,l} \in [0,1]\,.
\end{split}
\end{equation}
Then, for $\beta_N=\hbeta/\sqrt{\Ro_N}$ with $\hat\beta \in (0,1)$,
the following joint convergence in distribution holds:
\begin{equation}\label{fddconv}
	\big( Z_{N,\gb_N}^\go(\sfX_N^{(i)}) \big)_{1 \le i \le r}
	\ \xrightarrow[N\to\infty]{d} \
	\big( \colon e^{\sfY_i}\colon \big)_{1 \le i \le r} \,,
\end{equation}
where $(\sfY_i)_{1 \le i \le r}$ are jointly Gaussian random variables with
 \begin{equation}\label{meanvar}
 \bbE[\sfY_i] = 0 \,, \qquad
 \bbcov[\sfY_i , \,\sfY_j] = \log \frac{1-\hat\beta^2 \zeta_{i,j}}{1-\hat\beta^2} \,.
 \end{equation}
 \end{theorem}
\medskip

Lastly, we study $Z_{N,\beta_N}^\omega(\sfX)$  as a space-time random field on the macroscopic scale $\vvvert\sfX\vvvert \approx N$,
showing that it satisfies a law of large numbers with Gaussian fluctuations. For $\sfX = (x,t) \in \Z^d \times \N_0$,
we define space-time rescaled variables as follows (recall $L(\cdot)$ from Hypothesis~\ref{hyp}
and $\phi(\cdot)$ from \eqref{eq:CLT}):
\begin{equation}\label{eq:rescale}
	\widehat \sfX_N :=
	\big( \hat x_N, \hat t_N \big) :=
	\bigg( \frac{x}{\phi(N)} \,, \frac{t}{N} \bigg)  \,,
\end{equation}
where pinning models correspond to $d=0$ and
we drop $x$.
We first observe that $\Ro_N L(N)^2 \to \infty$ as $N\to\infty$,
by \eqref{eq:RN} and \cite[Prop.~1.5.9a]{cf:BinGolTeu}.
We are going to show that one has
\begin{equation} \label{eq:Zheur}
	Z_{N,\beta_N}^\omega(\sfX) \approx 1 + \frac{1}{\sqrt{\Ro_N L(N)^2}}
	\, G\big(\widehat\sfX_N\big)
\end{equation}
where $G(\,\cdot\,)$ is a generalized Gaussian random field on $\R^d \times [0,1]$,
with a logarithmically divergent covariance kernel (see \eqref{eq:kappakernel} below).
To make \eqref{eq:Zheur} precise, we fix a continuous test function
$\psi: \R^d \times [0,1] \to \R$ with compact support and define
\begin{equation} \label{eq:J0}
	J^\psi_N := \frac{1}{\phi(N)^d N}
	\sum_{\sfX \in \bbZ^d \times \N_0}
	\Big\{ \sqrt{\Ro_N L(N)^2} \,\big(Z_{N,\gb_N}^{\go}(\sfX)-1\big)
	\Big\} \, \psi\big( \widehat \sfX_N\big) \,,
\end{equation}
where the pre-factor is the correct Riemann-sum normalization, in agreement with \eqref{eq:rescale}.
We can now formulate our next result.

\begin{theorem}[Fluctuations of the rescaled field]\label{thm:field}
Let $Z_{N,\beta}^\omega(\sfX)$ be the partition function of
a directed polymer or pinning model started at
$\sfX \in \Z^d \times \N_0$ (cf.\ Def.~\ref{def:dp}
and~\ref{def:pinning}), such that Hypothesis~\ref{hyp} holds
and the replica overlap $\Ro_N$ in \eqref{eq:replica}--\eqref{eq:RN}
diverges as $N\to\infty$.

Fix any continuous function $\psi: \R^d \times [0,1] \to \R$ with compact support,
and let $\beta_N=\hbeta/\sqrt{\Ro_N}$ with $\hbeta < 1$.
Then $J_N^\psi$ in \eqref{eq:J0}
converges in distribution as $N\to\infty$ to a centered Gaussian random
variable $N(0, \sigma_\psi^2)$ with variance
\begin{equation}\label{eq:sigmapsi}
	\sigma_\psi^2 := \frac{\hbeta^2}{1-\hbeta^2}
	\int_{(\R^d \times [0,1])^2} \psi(x,t) \,
	K\big((x,t), (x',t')\big) \, \psi(x',t') \, \dd x \, \dd t \, \dd x' \, \dd t' \,,
\end{equation}
where the covariance kernel is given by
\begin{equation}\label{eq:kappakernel}
	K\big( (x_1, t_1), (x_2, t_2) \big) :=
	\begin{cases}
	\displaystyle
	\frac{1}{2} \int_{|t_1-t_2|}^{2-(t_1+t_2)} \frac{1}{s} \,
	g\bigg( \frac{x_1-x_2}{s^{1/d}} \bigg) \, \dd s
	& (\mbox{directed polymers}) \\
	\displaystyle
	\rule{0pt}{2em}
	\int_{t_1 \vee t_2}^{1} \frac{c^2}{\sqrt{s-t_1}\sqrt{s-t_2}} \, \dd s
	& (\mbox{pinning})
	\end{cases} \,.
\end{equation}
\end{theorem}
\begin{remark}\rm
Observe that the kernel $K$ diverges logarithmically near the diagonal:
\begin{equation*}
	K\big( (x_1, t_1), (x_2, t_2) \big) \sim
	C \, \log \frac{1}{| (x_1, t_1) - (x_2, t_2) |} \qquad \text{as}
	\quad | (x_1, t_1) - (x_2, t_2) | \to 0 \,.
\end{equation*}
Note that Gaussian fields with such logarithmically divergent covariance kernels have played a central role in the theory of Gaussian Multiplicative Chaos (see e.g.~\cite{RV13}).
\end{remark}

\subsection{Results for the 2d stochastic heat equation}\label{ss:SHE}

We now state the analogues of Theorems \ref{thm:subcritical}, \ref{thm:short-cor} and \ref{thm:field} for the $2d$ SHE
\begin{equation}\label{2dSHE}
\frac{\partial u}{\partial t} = \frac{1}{2}\Delta u + \beta u \dot{W}, \qquad u(0, x)  = 1 \ \forall \, x\in \R^2.
\end{equation}

To make sense of \eqref{2dSHE}, we first
mollify the space-time white noise $\dot{W}$. Let $j\in C^\infty_c(\R^2)$ be a probability density on $\R^2$ with $j(x)=j(-x)$, and let $J := j*j$. For $\epsilon>0$, let
$j_\epsilon(x) := \epsilon^{-2} j(x/\epsilon)$. The mollified noise $\dot{W}^\epsilon$ is defined formally by
$\dot{W}^\epsilon(t, x)  := \int_{\R^2} j_\epsilon(x-y) \dot{W}(t, y)\dd y$,
so that
$$
\int_{\R\times \R^2} f(t,x) \dot{W}^\epsilon(t, x)\dd t\dd x := \int_{\R\times\R^2} \Big(\int_{\R^2} f(t,x) j_\epsilon(y-x) \dd x\Big)\dot{W}(t,y) \dd t\dd y
\quad \forall\, f\in L^2(\R\times \R^2).
$$
For fixed $x$, the process $t \mapsto \int_0^t \dot{W}^\epsilon(s,x) \, \dd s$ is a
Brownian motion with variance $\| j \|_2^2$. Then we
consider the mollified equation (with It\^o integration, and $\beta = \beta_\eps$ possibly
depending on~$\eps$)
\begin{equation}\label{2dSHEeps}
\frac{\partial u^\eps}{\partial t} = \frac{1}{2}\Delta u^\eps + \beta_\eps u^\eps \dot{W}^\eps, \qquad u^\eps(0, \cdot)  \equiv 1,
\end{equation}
whose solution admits the generalized Feynman-Kac representation \cite[Sec.~3 and (3.22)]{BC95}
\begin{equation}\label{Feynman}
u^\eps(t,x) = \E_x\Big[\exp\Big\{\beta_\eps \int_0^t \dot{W}^\eps(t-s, B_s)\dd s -\frac{1}{2}\beta_\eps^2 \,\,\bbE\Big[\Big(\int_0^t \dot{W}^\eps(t-s, B_s)\dd s\Big)^2\Big] \Big\}\Big],
\end{equation}
where $\E_x$ is expectation w.r.t.\ $(B_s)_{s\geq 0}$, a standard Brownian motion in $\R^2$ with
$B_0=x$ and $\bbE$ denotes the expectation with respect to the White noise.
By a time reversal in $\dot{W}^\eps$, we note that $u^\eps(t,x)$ has the same distribution
(for fixed $(t,x)$) as
\begin{align}
\tilde u^\eps(t,x) & \, := \, \E_x\Big[\exp\Big\{\beta_\eps \int_0^t \dot{W}^\eps(s, B_s)\dd s -\frac{1}{2}\beta_\eps^2 \E\Big[\Big(\int_0^t \dot{W}^\eps(s, B_s)\dd s\Big)^2\Big]\Big\}\Big] \nonumber \\
& \, = \, \E_x\Big[\exp\Big\{\beta_\eps \int_0^t \int_{\R^2} j_\eps(B_s-y) \dot{W}(s, y) \dd s\dd y - \frac{1}{2}\beta_\eps^2 t \Vert j_\eps\Vert_2^2\Big\}\Big]  \nonumber \\
& \, = \, \E_{\eps^{-1}x}\Big[\exp\Big\{\beta_\eps \int_0^{\eps^{-2}t} \int_{\R^2} j(B_{\tilde s}-\tilde y) \dot{\tilde W}(\tilde s, \tilde y) \dd \tilde s\dd \tilde y - \frac{1}{2}\beta_\eps^2 (\eps^{-2}t) \Vert j\Vert_2^2\Big\}\Big] \,, \label{Feynman1.1}
\end{align}
where in the last step we made the change of variables $(\eps \tilde y, \eps^2 \tilde s):=(y,s)$, and $\dot{\tilde W}(\tilde s, \tilde y)\dd \tilde s \dd \tilde y:=\eps^{-2}\dot W(\eps^2 \tilde s, \eps \tilde y)\dd (\eps^2 \tilde s) \dd (\eps\tilde y)$ is another two-dimensional space-time white noise.
(One can actually extend \eqref{Feynman1.1} so that the equality in law between $u^\eps(t,x)$
and $\tilde u^\eps(t,x)$ holds jointly for all $t \in [0,1]$ and $x \in \R^2$, see \eqref{FKrec} below.)

Relation \eqref{Feynman1.1} suggests that
we can interpret $\tilde u^\eps(t,x)$ as the partition function of a directed Brownian polymer in $\R^2$ in a white noise space-time random environment at inverse temperature $\beta_\eps$, with starting point $\eps^{-1}x$ and polymer length $\eps^{-2}t$.
A consequence of our results for the short-range directed polymer on $\Z^2$ is the following analogue of Theorems \ref{thm:subcritical} and \ref{thm:short-cor}, combined into a single theorem.
Let us agree that $\vvvert \sfX \vvvert := t \vee |x|^2$.

\begin{theorem}[Limits of regularized solutions]\label{T:SHE}
Let $u^\eps(t,x)$ be the solution of the regularized 2d SHE \eqref{2dSHEeps}, with $\beta_\eps= \hbeta \sqrt{\frac{2\pi}{\log \eps^{-1}}}$ for some $\hbeta\in (0,\infty)$. Following the notation in Theorem~\ref{thm:short-cor}, consider a finite collection of space-time points $\sfX_\eps^{(i)} =(x_\eps^{(i)}, t_\eps^{(i)})$, $1 \le i \le r$, such that as $\eps\to 0$,
\begin{equation*}
	\forall i,j \in \{1,\ldots,r\}: \quad t^{(i)}_\eps = \eps^{o(1)}, \quad
	\vvvert \sfX_\eps^{(i)} - \sfX_\eps^{(j)} \vvvert = \eps^{2(1-\zeta_{i,j})+o(1)}
	\ \ \text{for some} \ \
	\zeta_{i,j} \in [0,1]\,.
\end{equation*}
Then for $\hat\beta < 1$, $\big( u^\eps(\sfX_\eps^{(i)}) \big)_{1 \le i \le r}$ converge in joint distribution  to the same limit $\big( \colon e^{\sfY_i}\colon \big)_{1 \le i \le r}$ as in \eqref{fddconv} as $\eps \to 0$, with $\bbE[(u^\eps(\sfX_\eps^{(i)}))^2] \to \bbE[( \colon e^{\sfY_i}\colon)^2]$; while for $\hbeta\geq 1$, $u^\eps(\sfX_\eps^{(i)})\Rightarrow 0$.
\end{theorem}
\begin{remark} \rm
Applying Hopf-Cole transformation to \eqref{2dSHEeps}, we note that $h^\eps(t,x):= \log u^\eps(t,x)$ is the solution of the regularized 2d KPZ equation
\begin{equation}
\frac{\partial h^\eps}{\partial t} = \frac{1}{2}\Delta h^\eps + \frac{1}{2} |\nabla h^\eps|^2
+ \beta_\eps \dot{W}^\eps - \beta_\eps^2 \eps^{-2} \Vert j\Vert_2^2,
\qquad h^\eps(0, \cdot)  \equiv 0.
\end{equation}
where the last term $-\beta_\eps^2 \eps^{-2} \Vert j\Vert_2^2$ is the It\^o correction.
Theorem~\ref{T:SHE} can therefore be reformulated for the 2d KPZ equation, showing that when $\hbeta\in (0,1)$, the solution $h^\eps$ has pointiwse Gaussian limits as $\eps\to 0$.
\end{remark}

Here is the analogue of Theorem~\ref{thm:field}.
\begin{theorem}[Fluctuations of the solution field]\label{T:SHEfield}
Let $u^\eps(t,x)$ be as in Theorem~\ref{T:SHE} with $\hbeta\in (0,1)$.  Let
$\psi: \R^2 \times [0,1] \to \R$ be continuous with compact support, and let
\begin{equation}\label{SHEfield}
J_\eps^\psi := \sqrt{\frac{\log \eps^{-1}}{2\pi}} \int_{\R^2\times [0,1]} (u^\eps(t,x)-1)\, \psi(x,1-t)\, \dd x \dd t.
\end{equation}
Then $J_\eps^\psi$ converges in distribution as $\eps\to 0$ to the same Gaussian random
variable $N(0, \sigma_\psi^2)$ as in Theorem~\ref{thm:field} for the directed polymer model on
$\Z^{2+1}$.
\end{theorem}

\begin{remark} \rm
For simplicity, we have formulated our results for the 2d SHE with $u^\eps(0, \cdot)\equiv 1$. However, it can be easily extended to general $u(0,\cdot)$. As it will become clear in the proof (or the heuristics in Section~\ref{heuristics}), for $\hbeta <1$, the limit of $u^\eps(t,x)$ depends only on the white noise $\dot W$ in an infinitesimal time window $[t-o(1), t]$ as $\eps\to 0$ (for directed polymer of length $N$, the partition function similarly depends only on the disorder in a time window $[1, N^{1-o(1)}]$). Therefore if we set the noise to be zero in the time window $[0, t-o(1)]$, then apply the Feynman-Kac formula \eqref{Feynman} first from time $t$ to $t-o(1)$, and then to $0$, then we will see that the limit of $u^\eps(t,x)$ depends on the initial condition only via a factor $\E_x[u^\eps(0,B_t)]$.
\end{remark}

\begin{remark} \rm
Bertini-Cancrini~\cite{BC98} showed that if in \eqref{2dSHEeps},
$\beta_\eps:= \sqrt{\frac{2\pi}{\log \eps^{-1}} + \frac{\lambda}{(\log \eps^{-1})^2}}$
for some $\lambda \in \R$, which corresponds to a finer window around $\hbeta=1$
in our notation, then $u^\eps$ is tight in a suitable space of distributions, and
the two-point function $\bbE[u^\eps(t,x) u^\eps(t,y)]$
converges to a non-trivial limit. However, they could not identify the limit of $u^\eps$.
Combined with our result that $u^\eps(t,x)$ converges in probability to $0$ for each $x\in \R^2$ when
$\hbeta=1$, this suggests that the random measure $u^\eps(t, x)\dd x$ may have a non-trivial
limit as $\eps\to 0$, which is singular w.r.t.\ the Lebesgue measure. \end{remark}

\begin{remark} \label{rem:she}{\rm
We note a formal connection between the 2d SHE and Gaussian multiplicative chaos (GMC), which typically considers random measures $M_\beta(\dd x):=e^{\beta X_x- \beta^2\bbE[X_x^2]/2} \dd x$ on $[0,1]^d$ for some Gaussian field $(X_x)_{x\in [0,1]^d}$. When the covariance kernel of $X$ is divergent on the diagonal, $X$ is a generalized function and to define $M_\beta(\dd x)$, one first replaces $X$ by its mollified version $X^\eps$ and defines $M^\eps_\beta(\dd x)$ and then takes the limit $\eps\to 0$ $($see \cite{RV13} for a survey$)$.
For the 2d SHE, the exponential weight in \eqref{Feynman} can be seen as the analogue of $e^{\beta X^\eps_x- \beta^2\bbE[(X^\eps_x)^2]/2}$ for the mollified Gaussian field $X^\eps$, except now the Gaussian field $X^\eps$ is indexed by $C([0,t], \R^2)$ endowed with the Wiener measure.
As $\eps\to 0$, its covariance kernel $K^\eps(\cdot,\cdot)$ can be seen to diverge logarithmically
in probability, if it is regarded as a random variable defined on $C([0,t], \R^2)^2$ endowed with
the product Wiener measure. We note that shortly after the completion of this paper,
Mukherjee et al.~\cite{MSZ16} used techniques from GMC to prove the existence of a
weak to strong disorder transition for the SHE in $d\geq 3$.
}
\end{remark}

\section{Heuristics}\label{heuristics}

In this section we illustrate the core of our approach,
emphasizing the main ideas and keeping the exposition at a heuristic level.
In \S\ref{sec:relevant} we
recall the approach developed in \cite{CSZ13} to deal with the disorder relevant regime,
then in \S\ref{ss:heur} we explain how it
fails for marginally relevant models and how does the marginal overlap condition arise.

\subsection{Heuristics for disorder relevant regime}
\label{sec:relevant}

For simplicity, we use the pinning model to illustrate the general approach developed in \cite{CSZ13}
to identify limits of partition functions in a suitable continuum and weak disorder limit.

We first rewrite the partition function \eqref{eq:Zpin} for $t=0$:
since $e^{x \ind_{\{n\in\tau\}}} = 1 + (e^x-1) \ind_{\{n\in\tau\}}$ for all $x\in\R$,
we get
\begin{equation}\label{eq:Zstart}
	Z^\omega_{N,\gb} =
	\E \Bigg[ \prod_{n=1}^N \big(1 + \gb\, \eta_n \, \ind_{\{n\in\tau\}}
	\big) \Bigg] \qquad \text{where} \quad
	\eta_n := \frac{e^{\gb\omega_n -\gl(\beta)} - 1}{\beta} \,.
\end{equation}
A binomial expansion of the product in \eqref{eq:Zstart} then yields
(setting $n_0 = 0$)
\begin{equation} \label{eq:Zstart2}
	Z^\omega_{N, \gb} =
	1 + \sum_{k=1}^{N} \;\gb^k
	\sum_{1\leq n_1<\cdots<n_k\leq N}
	\, \prod_{j=1}^k q_{n_j-n_{j-1}}
	\, \prod_{i=1}^k \eta_{n_i}
	\qquad \text{where} \quad q_{n} := \P(n\in \tau) \,.
\end{equation}
We have thus rewritten $Z^\omega_{N, \gb}$ as a multi-linear polynomial of the i.i.d.\
random variables $(\eta_{n})_{n\in\N}$, sometimes called a \emph{polynomial chaos} expansion.

Assume for simplicity that the underlying renewal process $\tau$ satisfies
\begin{equation}\label{ren00}
\P(\tau_1-\tau_0=n) \sim \frac{C}{n^{1+\alpha}} \qquad \mbox{as } n\to\infty
\end{equation}
for some $C>0$ and $\alpha \in (0,1)$, which implies the local limit theorem~\cite[Thm.~B]{D97}.
\begin{equation}\label{doney}
q_n:= \P(n\in \tau) \sim \frac{\tilde C}{n^{1-\alpha}} = \frac{\alpha \sin (\pi\alpha)}{C\pi} \cdot \frac{1}{n^{1-\alpha}}.
\end{equation}

Recalling \eqref{eq:Zstart}, we have $\bbE[\eta_{n}] =0$, and by Taylor expansion,
\begin{equation} \label{eq:Taylor}
 \bbvar[\eta_{n}] \sim 1  \qquad \mbox{as } \beta\to 0.
\end{equation}
Since the ``influence'' of each $\eta_n$ on $Z^\omega_{N, \gb}$ is small, we can apply a
\emph{Lindeberg principle} (see e.g.~\cite{CSZ13, MOO10, NPR10}) to replace $(\eta_n)_{n\in\N}$
by i.i.d.\ standard Gaussian random variables without changing the limiting distribution of $Z^\omega_{N, \gb}$
as $N\to\infty$.

Standard i.i.d.\ Gaussian $(\eta_n)_{n\in\N}$ can be defined from a white noise $W(\dd t)$ on $[0,\infty)$, with
\begin{equation} \label{eq:QN}
	\eta_{n} \, :=\,
	\sqrt{N} \int_{\frac{n}{N}}^{\frac{n+1}{N}} W(\dd s) , \qquad n\in\N.
\end{equation}
Setting $t_i :=n_i/N$ for each $i\in\N$, the series \eqref{eq:Zstart2} then becomes a series of stochastic integrals
\begin{eqnarray}
	Z^\omega_{N, \gb} &\approx&
	1 + \sum_{k=1}^{N} \, (\gb N^{\frac{1}{2}})^k
	\idotsint \limits_{0 < t_1 < \cdots < t_k < 1}
	\ \prod_{j=1}^k
	q_{\lfloor Nt_j\rfloor -\lfloor Nt_{j-1}\rfloor}
	\, \prod_{i=1}^k W(\dd t_i) \nonumber \\
	&\approx&
	1 + \sum_{k=1}^{N} \, (\tilde C\gb N^{\alpha-\frac{1}{2}})^k
	\idotsint \limits_{0 < t_1 < \cdots < t_k < 1}
	\ \prod_{j=1}^k
	(t_j-t_{j-1})^{\ga-1}
	\, \prod_{i=1}^k W(\dd t_i) \,,    \label{eq:Zstart2bis}
\end{eqnarray}
where we have applied \eqref{doney} that $q_{Nt} \sim \tilde C(Nt)^{\alpha-1}$.

In the disorder relevant regime $\alpha\in (1/2, 1)$, we note that $g(t):=t^{\ga-1}$ is square-integrable in $t\in [0,1]$ and the stochastic integrals in \eqref{eq:Zstart2bis}
are all well-defined. In particular, in the weak disorder limit
\begin{equation}\label{eq:betaNold}
	\beta_N := \frac{\hat\beta}{\tilde C N^{\alpha-\frac{1}{2}}}\,, \qquad
	\text{with} \quad \hat\beta \in (0,\infty) \,,
\end{equation}
relation \eqref{eq:Zstart2bis} suggests that
as $N\to\infty$, the partition function $Z^\omega_{N, \gb_N}$ converges in law to
\begin{equation} \label{eq:Zlimit}
	\bs{Z}_{\hat\beta}^W :=
	1 + \sum_{k=1}^{\infty} \, \hat\beta^k
	\idotsint \limits_{0 < t_1 < \cdots < t_k < 1}
	\ \prod_{j=1}^k
           (t_j-t_{j-1})^{\ga-1}
	\, \prod_{i=1}^k W(\dd t_i) \,.
\end{equation}
The limit $\bs{Z}_{\hat\beta}^W$ can then be used to define a continuum disordered pinning model \cite{CSZ16}.

For the marginal case $\alpha=1/2$, the above approach
breaks down because $1/\sqrt{t}$ just fails to be square-integrable in $[0,1]$ and the
stochastic integrals in \eqref{eq:Zstart2bis} become undefined. Nevertheless, for each $k\in\N$, we
note that the second moment of the $k$-th term in \eqref{eq:Zstart2} diverges as $N\to\infty$,
which hints at marginal relevance of disorder.

For directed polymer models, exactly the same phenomenon appears. The approach of~\cite{CSZ13} sketched above applies to the short-range directed polymer on $\Z^{1+1}$ and the long-range directed polymer on $\Z^{1+1}$ with tail exponent $\alpha \in (1,2)$, and breaks down exactly at the marginal cases, which include the short-range directed polymer on $\Z^{2+1}$ and the long-range directed polymer on $\Z^{1+1}$ with tail exponent $\alpha=1$.

\subsection{Heuristics for marginal relevant regime}\label{ss:heur}
We now sketch the heuristics behind our proof of Theorem~\ref{thm:subcritical}. Again, we use the pinning model to illustrate our approach, focusing on the marginal case where the renewal process satisfies \eqref{ren00} with $\alpha=1/2$.

For simplicity, while retaining the key features, we assume that $(\eta_n)_{n\in\N}$ are i.i.d.\ standard normal, and in light of \eqref{doney}, we assume for simplicity that $q_n=1/\sqrt{n}$. Then $Z^\omega_{N, \gb}$ in \eqref{eq:Zstart2} simplifies to
\begin{equation}\label{Zsimplified}
	Z_N = 1+\sum_{k=1}^N \;\gb_N^k
	\sum_{1\leq n_1<\cdots<n_k\leq N}
	\, \prod_{j=1}^k \frac{\eta_{n_j}}{\sqrt{n_j-n_{j-1}}}  \,.
\end{equation}

The first observation, which follows from a direct calculation, is that for each $k\in\N$, the associated inner sum in $Z_N$ has second moment
$$
\E\left[\Big(\sum_{1\leq n_1<\cdots<n_k\leq N}
	\, \prod_{j=1}^k \frac{\eta_{n_j}}{\sqrt{n_j-n_{j-1}}} \Big)^2\right]
	=
	\sum_{1\leq n_1<\cdots<n_k\leq N}
	\, \prod_{j=1}^k \frac{1}{n_j-n_{j-1}}   \sim (\log N)^k \sim \Ro_N^k,
$$
where $\Ro_N$ is the expected replica overlap defined in \eqref{eq:replica} and satisfies the {\em marginal overlap condition}. This suggests that if there is a non-trivial weak disorder limit for $Z_N$, then we should choose $\beta_N:=\hbeta/\sqrt{\Ro_N}$ for some $\hbeta>0$. Furthermore, note that $\E[Z_N^2]\to (1-\hbeta^2)^{-1}$ for $\hbeta\in (0,1)$ and $\E[Z_N^2]\to \infty$ for $\hbeta\geq 1$, with a transition occurring at $\hbeta_c=1$.

We assume from now on $\beta_N:=\hbeta/\sqrt{\log N} \sim \hbeta/\sqrt{\Ro_N}$ in \eqref{Zsimplified} with $\hbeta\in (0,1)$, so that
\begin{equation}\label{Zsimplified2}
	Z_N = 1+\sum_{k=1}^N \hbeta^k Z_N^{(k)} \quad \mbox{with} \quad Z_N^{(k)}:= \frac{1}{(\log N)^{\frac{k}{2}}} \sum_{1\leq n_1<\cdots<n_k\leq N}
	\, \prod_{j=1}^k \frac{\eta_{n_j}}{\sqrt{n_j-n_{j-1}}}  \,.
\end{equation}
To prove Theorem~\ref{thm:subcritical}, that $Z_N$ converges in law to a log-normal random variable, we will identify the limit of
$Z_N^{(k)}$ for each $k\in\N$, where an interesting structure appears. Below are the key observations.
\medskip

{\bf (A)} {\em An elementary observation.}  Let $(W(t))_{t\geq 0}$ be a standard Brownian motion on $\R$. For any $\delta>0$, let $W_\delta(t):= W(\delta t)/\sqrt{\delta}$, which is another standard Brownian motion correlated with $W$. A simple covariance calculation then shows that as $\delta\downarrow 0$, $W$ and $W_\delta$ become asymptotically independent. Such asymptotic independence due to separation of scales also extends to higher-dimensional white noise, which will be crucial in our analysis.
\smallskip

{\bf (B)} {\em Identifying the time scale.} Next, we identify the intrinsic time scale appearing in the limit of $Z_N^{(1)}$, and $Z_N^{(k)}$ in general. Note that for any $0\leq a<b\leq 1$, $\frac{1}{\sqrt{\log N}} \sum_{n=N^a}^{N^b} \frac{\eta_n}{\sqrt n}$ converges in distribution to a Gaussian random variable with mean zero and variance $b-a$. Therefore to approximate the sum in $Z_N^{(1)}$ by a stochastic integral, we should make the change of variable $n=N^a$, which gives
\begin{equation}\label{ZN1approx}
\frac{1}{\sqrt{\log N}} \sum_{n=N^a}^{N^b} \frac{\eta_n}{\sqrt n} \approx \int_a^b
W^{(1)} (\dd s) \quad \forall\, 0\leq a<b\leq 1,
\end{equation}
where $W^{(1)}$ is  a standard Brownian motion. In particular,
$Z_N^{(1)}\approx \int_0^1 W^{(1)} (\dd s) = W^{(1)}_1$. This indicates that the
correct time scale is exponential $t\to N^t$, rather than linear $t\to N t$.
\smallskip

{\bf (C)} {\em Identifying the structure.} Finally, we identify the limit of $Z_N^{(2)}$, where the key structure already emerges. An $L^2$ calculation shows that as $N\to\infty$, we can relax the range of summation:
\begin{equation}
Z_N^{(2)}=\frac{1}{\log N}\!\!
	\sum_{0 < n < m \le N}\! \frac{\eta_{n} \, \eta_{m}}{\sqrt{n} \sqrt{m-n}} \approx \frac{1}{\sqrt{\log N}}
	\!\!\sum_{0 < n_1\le N} \!\frac{\eta_{n_1}}{\sqrt{n_1}} \Big(\frac{1}{\sqrt{\log N}} \sum_{0 < n_2\le N} \!\!\!\frac{\eta_{n_1+n_2}}{\sqrt{n_2}}\Big).
\end{equation}
Using the approximation \eqref{ZN1approx} with $n_1=:N^{s_1}$ and similarly for the sum over $n_2=:N^{s_2}$,
\begin{equation}\label{eq:Z12}
Z_N^{(2)} \approx \int_0^1 {\rm d}W^{(1)}_{s_1} \Big(\frac{1}{\sqrt{\log N}}
\sum_{0 \le n_2\le N} \!\!\!\frac{\eta_{N^{s_1}+n_2}}{\sqrt{n_2}}\Big) \approx \int_0^1
W^{(1)} (\dd s_1) \int_0^1 W^{(2; s_1)}(\dd s_2),
\end{equation}
where given $s_1$, $W^{(2; s_1)}$ is a standard Brownian motion with
\begin{equation}\label{ZN1approxx}
\frac{1}{\sqrt{\log N}} \sum_{n=N^{s_1}}^{N^{s_1}+N^{s_2}} \frac{\eta_n}{\sqrt n}  \approx
W^{(2; s_1)}(s_2).
\end{equation}
To understand the relation between $W^{(1)}_{s_1}$ and $W^{(2;s_1)}_{s_2}$ and
make sense of the stochastic integral in \eqref{eq:Z12},
we distinguish between the cases $s_2<s_1$
and $s_2>s_1$.

\begin{itemize}
\item \textbf{Case $s_2 < s_1$:} In this case, $N^{s_2} \ll N^{s_1}$, and observation {\bf(A)} shows that
in the limit $N\to\infty$, the white noise $(W^{(2; s_1)}(\dd s_2))_{0 \le s_2 \le s_1}$
becomes asymptotically independent of  $(W^{(1)}(\dd s_1))_{0 \le s_1 \le 1}$. Indeed, by \eqref{ZN1approx} and
\eqref{ZN1approxx}, we note that the increments of $W^{(2;s_1)}$ in a small time window $[s_2, s_2+\Delta]$
is defined from $\eta_n$ with $n\in [N^{s_1} +N^{s_2}, N^{s_1} +N^{s_2+\Delta}]$, which is an infinitesimal window contained in the range of indices $[N^{s_1}, N^{s_1+\Delta}]$ used to define the increments of $W^{(1)}$ on $[s_1, s_1+\Delta]$. In other words, the white noise
$(W^{(2; s_1)}(\dd s_2))_{0 \le s_2 \le s_1}$ is effectively obtained by sampling
$W^{(1)}(\dd s_1)$ in an infinitesimal window in $[s_1, s_1+\Delta]$.
A covariance calculation as
in {\bf (A)}  shows that in the limit $N\to\infty$, $W^{(1)}(\dd s_1)$
and $W^{(2; s_1)}(\dd s_2)$
are independent for all $a, s_1\in [0,1]$ and $b\in [0, s_1]$. Furthermore,
using the Fourth Moment Theorem, it can be shown that
\begin{equation}
	\quad \big( \Gamma(\dd s_1, \dd s_2)
	\,:=\, W^{(1)}(\dd s_1) \cdot W^{(2; s_1)}(\dd s_2) \big)_{0 \le s_2 < s_1 \le 1}
\end{equation}
\emph{is a two-dimensional white noise,
independent of $(\dd W^{(1)}_{s_1})_{0 \le s_1 \le 1}$}.

\item \textbf{Case $s_2 > s_1$:} In this case, $N^{s_1} \ll N^{s_2}$ as $N\to\infty$. Therefore,
 the range of indices $[N^{s_1} +N^{s_2}, N^{s_1} +N^{s_2+\Delta}]$ essentially coincide with $[N^{s_2}, N^{s_2+\Delta}]$, which are the indices of $\eta$ used to define respectively the increments of $W^{(2;s_1)}$ and $W^{(1)}$
in a small window $[s_2, s_2+\Delta]$. This implies that in the limit $N\to\infty$, we have
$W^{(2; s_1)}(\dd s_2) = W^{(1)}(\dd s_2)$.
\end{itemize}
\medskip

By the above considerations, we can now rewrite the approximation \eqref{eq:Z12} as
\begin{equation} \label{eq:Z2appr}
	Z^{(2)}_N \,\approx\, \int_{0 \le s_2 < s_1 \le 1}
	\Gamma(\dd s_1, \dd s_2) \,+\,
	\int_{0 \le s_1 < s_2 \le 1} W^{(1)}(\dd s_1) \, \dd W^{(1)}(\dd s_2) \,,
\end{equation}
where the first term is a normal random variable with mean zero and variance $1/2$, independent of the second term, which can be rewritten as
\begin{equation*}
	\int_0^1  W^{(1)}(\dd s_2) \, \bigg( \int_0^{s_2}  W^{(1)}(\dd s_1) \bigg)
	\,=\,\int_0^1 W^{(1)}(s_2) \,  W^{(1)}(\dd s_2) \,=\,
	\frac{(W^{(1)}(1))^2 - 1}{2} \,.
\end{equation*}
	
When we consider the limit of $Z_N^{(k)}$ for $k\geq 3$, similar separation of scales appears when
we make the change of time scale $n_i=N^{s_i}$.  The limit of $Z_N^{(k)}$ admits a
decomposition similar to \eqref{eq:Z2appr} (but more complicated), involving independent
white noises of various different dimensions up to dimension $k$.

\smallskip

So far we focused on pinning models, but everything can be extended to directed polymer models,
whose partition function admits a polynomial chaos expansion
analogous to \eqref{eq:Zstart2}: see \eqref{eq:Zpoly0} below.
Remarkably, the structure is the same as for the pinning model: if we make
the change of time variable $n_i=N^{a_i}$ and
a change of space variable $z_i=x_i n_i^{1/d}$ (assuming $L(\cdot)\equiv 1$ in
Hypothesis~\ref{hypoth}),
then similar to \eqref{ZN1approx}, $Z_N^{(1)}$ can be approximated by
$\int_0^1 \int_{\R^d} W^{(1)}(\dd t \, \dd x)$ for a white noise $W^{(1)}$ on
$\R^d\times [0,\infty)$. Concerning $Z_N^{(2)}$, in analogy with \eqref{eq:Z12},
for each $s_1>0$ and $x_1\in \R^d$, we have an independent
white noise $(W^{(2; s_1, x_1)}(\dd s_2 \, \dd x_2))_{s_2\in [0, s_1], x_2\in \R^d}$,
which is  effectively obtained by sampling $W^{(1)}$ in an infinitesimal space-time window
around $(s_1, x_1)$, while $(W^{(2; s_1, x_1)}(\dd s_2 \, \dd x_2))_{s_2>s_1, x_2\in\R^d}
= (W^{(1)}(\dd s_2 \, \dd x_2))_{s_2>s_1, x_2\in\R^d}$.

\section{Proof steps for Theorem~\ref{thm:subcritical}}\label{ss:proofsteps}

Since the proof of Theorem~\ref{thm:subcritical} (for $\hbeta\in (0,1)$) is long and modular,
we list here the proof steps. These contain four approximations {\bf (A1)--(A4)}, plus one key step {\bf (K)}
which identifies the building blocks of the limiting partition function. The local limit theorems
\eqref{eq:llt2}--\eqref{eq:llt0} in Hypothesis \ref{hypoth} will only be used in the approximation step {\bf (A3)}. The other steps only use the marginal overlap condition, i.e., $\Ro_N$ is a divergent slowly varying function.

The proof steps are the same for pinning $(d=0)$ and directed polymer models $(d=1, 2)$,
so we follow a unified notation.
The starting point is a polynomial chaos
expansion for the partition function of directed polymers,
in analogy with \eqref{eq:Zstart}-\eqref{eq:Zstart2} for pinning:
\begin{equation} \label{eq:Zpoly0}
\begin{aligned}
Z^\omega_{N, \gb_N}&= 1+\sum_{k=1}^N \hbeta^k Z_N^{(k)},  \qquad (\beta_N=\hbeta/\sqrt{\Ro_N}), \\
\mbox{where}\qquad  Z_N^{(k)} &= \frac{1}{\Ro_N^{k/2}}
		\sumtwo{1\leq n_1<\cdots<n_k\leq N}{z_1, z_2, \ldots, z_k \in \bbZ^d}
	\, \prod_{j=1}^k q_{n_j-n_{j-1}}(z_j-z_{j-1})
	\, \prod_{i=1}^k \eta_{(n_i,z_i)} \,,
\end{aligned}
\end{equation}
with $n_0:=0$, $z_0:=0$ and
\begin{equation} \label{eq:uffano}
	q_n(z) :=\P(S_n=z) \,, \qquad
	\eta_{(n,z)} = \eta^{(N)}_{(n,z)} := \frac{e^{\beta_N \omega_{(n,z)}
	-\lambda(\beta_N)} - 1}{\beta_N} \,.
\end{equation}
Note that relation \eqref{eq:Zpoly0}
applies also to the pinning model, if we view it as a directed polymer on $\Z^0 := \{0\}$
(cf.~Hypothesis~\ref{hypoth}) and identify $q_n(0)$ with $q_n = \P(n\in\tau)$.

\smallskip

As a preliminary step, we can approximate $Z^{\omega}_{N, \gb_N}$ from \eqref{eq:Zpoly0}
in $L^2$ (uniformly in $N$) by $Z^{\omega, K}_{N, \gb_N} := 1+\sum_{k=1}^{K} \hbeta^k Z_N^{(k)}$
if $K$ is large, but fixed, since for $\hbeta\in (0,1)$,
\begin{equation}
\lim_{K\to\infty} \Vert Z^{\omega}_{N, \gb_N} - Z^{\omega, K}_{N, \gb_N}\Vert_2^2 = \lim_{K\to\infty} \sum_{k=K+1}^N \hbeta^{2k} \Vert Z_N^{(k)}\Vert_2^2 \leq \lim_{K\to\infty} \sum_{k=K+1}^\infty \hbeta^{2k} =0 \,,
\end{equation}
where we used the fact that $\Vert Z_N^{(k)}\Vert_2^2 \le 1$,
as one checks by \eqref{eq:Zpoly0} and \eqref{eq:replica}
(see \eqref{eq:norma1} below).
\emph{We can therefore focus on identifying the limit $Z^{\omega, K}_{N, \gb_N}$ as $N\to\infty$,
and send $K\to\infty$ later.}
\medskip

Our first step is to approximate $Z_N^{(k)}$ in \eqref{eq:Zpoly0} as follows.
\begin{itemize}
\item[{\bf (A1)}] For each $k\in\N$, define $\widehat Z_N^{(k)}$ by enlarging the range of summation for $Z_N^{(k)}$
in \eqref{eq:Zpoly0}, allowing the time increments
$n_1$, $n_2-n_1$, \ldots, $n_k-n_{k-1}$ to vary freely in $\{1,\ldots,N\}$,
and show that $\Vert Z_N^{(k)} - \widehat Z_N^{(k)}\Vert_2^2\to 0$ as $N\to\infty$.
\end{itemize}
Note that this allows us to replace $Z^{\omega, K}_{N, \gb_N}$ by
\begin{equation}\label{ZA1K}
Z_{N, \beta_N}^{{\bf (A1)}}:=1+\sum_{k=1}^{K} \hbeta^k \widehat Z_N^{(k)}, \quad \mbox{with} \quad \Big\Vert  Z^{\omega, K}_{N, \gb_N} - Z_{N, \beta_N}^{{\bf (A1)}}\Big\Vert_2^2 \asto{N} 0.
\end{equation}

Let us now consider $M$ arbitrary and
for each $\widehat Z_N^{(k)}$ partition the range $\{1, \ldots, N\}$ for
each variable $n_1$, $n_2-n_1$, \ldots, $n_k - n_{k-1}$ into $M$ blocks
$I_1, I_2, \cdots, I_M$, defined by (with $\sft_0 := 0$)
\begin{align}\label{aimP}
	I_i := \big(\sft_{i-1}, \sft_i \big] \qquad
	\text{with} \qquad
	\sft_i &  := \sft^{N,M}_i
	:= \min\Big\{ m \in \{1,\ldots,N\}\colon
	\  \Ro_m \geq \tfrac{i}{M}\,\Ro_{N} \Big\} \,.
\end{align}
Note that for $\Ro_N = \log N$ we have $I_i = (N^{\frac{i-1}{M}}, N^{\frac{i}{M}}]$.
We can then write

\begin{equation}\label{genk}
\begin{split}
	\widehat Z_N^{(k)} &\phantom{:}= \frac{1}{M^{\frac{k}{2}}} \sum_{1\leq i_1,...,i_k\leq M}  \,
	\Theta^{N;M}_{i_1,...,i_k} \,, \qquad  \text{where} \\
	\Theta^{N;M}_{i_1,...,i_k} & :=
	\left(\frac{M}{\Ro_N}\right)^{k/2} \!\!\!\!\!\!
	\sumtwo{n_1-n_0 \in I_{i_1}, \, n_2 - n_1 \in I_{i_2}, \ldots, \, n_k - n_{k-1} \in I_{i_k}}
	{z_1, z_2, \ldots, z_k \in \Z^d}
	\, \prod_{j=1}^k q_{n_j-n_{j-1}}(z_j-z_{j-1})
	\, \prod_{i=1}^k \eta_{(n_i,z_i)} \,,
\end{split}	
\end{equation}
where $(n_0, z_0) = (0,0)$.

\begin{remark}\rm
The intervals $(I_i)_{1\leq i\leq M}$ encode the right time scale, as explained in {\bf (B)} in Section \ref{ss:heur},
because  $\Ro_{\sft_i} - \Ro_{\sft_{i-1}} \sim \frac{1}{M}\Ro_N$. The sum over $i_1,\ldots,i_k$ in $\widehat Z_N^{(k)}$ in \eqref{genk} corresponds to a discretization of the stochastic integrals that will arise in the limit $N\to\infty$.
\end{remark}

To ensure a proper separation of scales later on, define
\begin{equation}\label{eq:DM0}
\{1...,M\}^k_{\sharp}:= \{ \bi = (i_1,...,i_k)\in \{1,...,M\}^k:  |i_j-i_{j'}|\geq 2 \mbox{ for all } j\neq j'\}.
\end{equation}
Our second approximation shows that the contributions to $\widehat Z_N^{(k)}$ in \eqref{genk} from summation indices
$\bi \in \{1, \ldots, M\}^k\backslash  \{1, \ldots, M\}^k_\sharp$ is small for large $M$, uniformly in large $N$, i.e.,
\begin{itemize}
\item[{\bf (A2)}]
\hfill $\displaystyle \lim_{M\to\infty} \limsup_{N\to\infty}
\bigg\Vert \sum_{\bi \in \{1, \ldots, M\}^k \setminus \{1, \ldots, M\}^k_\sharp}
\frac{1}{M^{\frac{k}{2}}} \,\Theta^{N;M}_{\bi}\bigg\Vert_2^2 = 0.$ \hfill\,
\end{itemize}
Therefore we can restrict the sum over $\bi$ in $\widehat Z_N^{(k)}$ to $\bi \in \{1, \ldots, M\}^k_\sharp$. Note that this implies
we can further replace $Z_{N, \beta_N}^{{\bf (A1)}}$ in \eqref{ZA1K} by
\begin{equation}\label{ZA2K}
Z_{N, \beta_N}^{{\bf (A2)}} := 1+\sum_{k=1}^{K} \frac{\hbeta^k }{M^{\frac{k}{2}}} \sum_{\bi\in \{1,\ldots, M\}^k_\sharp}
	\!\!\!\!\!\Theta^{N;M}_{\bi}, \quad \mbox{with} \quad
\lim_{M\to\infty}\limsup_{N\to\infty} \Vert Z_{N, \beta_N}^{{\bf (A1)}}- Z_{N, \beta_N}^{{\bf (A2)}} \Vert_2^2 =0.
\end{equation}

We now try to identify the limit of $\Theta^{N;M}_{\bi}$ as $N\to\infty$. The heuristics sketched in Section \ref{ss:heur}  for $Z_N^{(1)}$ and $Z_N^{(2)}$ suggest the following:
\begin{itemize}
\item Case $k=1$: the family $(\Theta^{N;M}_{i})_{1\leq i\leq M}$
converges in distribution to i.i.d.\ standard normal random variables $(\zeta_i)_{1\leq i\leq M}$.

\item Case $k=2$: for $i_1\leq i_2-2$, the family $\Theta^{N;M}_{i_1,i_2}$ converges in distribution to $\zeta_{i_1}\zeta_{i_2}$,
while for $i_1\geq i_2+2$, the family $\Theta^{N;M}_{i_1,i_2}$ converges in distribution to a family of i.i.d.\ standard
normal random variables $\zeta_{i_1, i_2}$ independent of $(\zeta_i)_{1\leq i\leq M}$.
\end{itemize}
For $k\geq 3$, the limit of $\Theta^{N;M}_{i_1,...,i_k}$ also turns out to be a product of
independent standard normal random variables $\zeta_\cdot$,
with \emph{one $\zeta_\cdot$ for each running maxima in the sequence $(i_1,...,i_k)$}.
More precisely, let us say that
\begin{equation}\label{1domseq}
\bi:=(i_1, \ldots, i_k)\in \{1...,M\}^k \mbox{  is a {\em dominated sequence}  if  } i_1>i_2,\ldots, i_k.
\end{equation}
Then each $\bi\in \{1, \ldots, M\}^k_\sharp$ can be divided
into consecutive dominated sequences $\bi^{(1)}:=(i_1, \ldots, i_{\ell_2-1})$,
$\bi^{(2)}:=(i_{\ell_2}, \ldots, i_{\ell_3-1})$, \ldots, $\bi^{(\mathfrak{m})}:=(i_{\ell_{\mathfrak{m}}}, \ldots, i_k)$,
where $i_{ \ell_1}=i_1< \cdots<i_{\ell_{\mathfrak{m}}}$ are the successive running maxima
of $(i_1, \ldots, i_k)$.

Our third approximation step shows that the random variable
$\Theta^{N;M}_{\bi}$ in \eqref{genk} admits the following asymptotic factorization:

\begin{itemize}
\item [{\bf (A3)}] For all $M,k\in\N$ and for each $\bi:=(i_1, \ldots, i_k)\in \{1...,M\}^k_{\sharp}$,
\begin{equation}
\lim_{N\to\infty} \Big\Vert\Theta^{N;M}_{\bi} - \Theta^{N;M}_{\bi^{(1)}} \Theta^{N;M}_{\bi^{(2)}}\cdots \Theta^{N;M}_{\bi^{(\mathfrak{m})}}\Big\Vert_2^2 = 0,
\end{equation}
where $(\bi^{(1)}, \ldots, \bi^{(\mathfrak{m})})$ is the decomposition of $\bi$ into dominated sequences.
\end{itemize}
Note that this allows us to further replace $Z_{N, \beta_N}^{{\bf (A2)}}$ in \eqref{ZA2K} by
\begin{equation}\label{ZA3K}
Z_{N, \beta_N}^{{\bf (A3)}} := \! 1+ \!\sum_{k=1}^{K}  \frac{\hbeta^k}{M^{\frac{k}{2}}} \sum_{\bi\in \{1,\ldots, M\}^k_\sharp}  \!\!\!\!\!\!
	\Theta^{N;M}_{\bi^{(1)}} \Theta^{N;M}_{\bi^{(2)}}\cdots \Theta^{N;M}_{\bi^{(\mathfrak{m})}} \quad\mbox{with} \quad \Vert Z_{N, \beta_N}^{{\bf (A2)}}- Z_{N, \beta_N}^{{\bf (A3)}} \Vert_2^2  \asto{N} 0.
\end{equation}
\smallskip

We are now reduced to identifying the limit of $\Theta^{N;M}_{\bi}$ when $\bi$ are dominated
sequences. Denote
\begin{equation}\label{eq:DM}
D_M := \Big\{
\bi \in \textstyle\bigcup_{k=1}^\infty\{1...,M\}^k_{\sharp}:
\ \bi \text{ is a dominated sequence}
\Big\} \,.
\end{equation}
Here is the key step in the proof of Theorem~\ref{thm:subcritical}:
\smallskip

\begin{itemize}
\item[{\bf (K)}] As $N\to\infty$, the family of random variables $(\Theta^{N;M}_{\bi})_{\bi\in D_M}$ converges in distribution to a family of i.i.d.\ standard normal random variables $(\zeta_{\bi})_{\bi\in D_M}$.
\end{itemize}
In particular, this implies that
\begin{equation}\label{eq:cutoff}
Z_{N, \beta_N}^{{\bf (A3)}} \xrightarrow[N\to\infty]{d}
	\bs{Z}_\hbeta^{M,K}
	:= 1 + \!\!\sum_{k=1}^{K} \frac{\hbeta^k }{M^{\frac{k}{2}}} \!\!\sum_{\bi \in \{1, \ldots, M\}^k_\sharp}
	 \prod_{l=1}^{\mathfrak{m}(\bi)} \zeta_{\bi^{(l)}} \,.
\end{equation}

To complete the proof of Theorem~\ref{thm:subcritical} for $\hbeta\in (0,1)$, we first take the limit $K\to\infty$. By the fact that $\hat\gb<1$, it is clear
that $\bs{Z}_\hbeta^{M,K}$ converges
as $K\to\infty$ to
\begin{equation}\label{eq:cutofff}
\bs{Z}_\hbeta^{(M)}
	:= 1 + \!\!\sum_{k=1}^{\infty} \frac{\hbeta^k }{M^{\frac{k}{2}}} \!\!\sum_{\bi \in \{1, \ldots, M\}^k_\sharp}
	 \prod_{l=1}^{\mathfrak{m}(\bi)} \zeta_{\bi^{(l)}} ,
\end{equation}
\emph{uniformly in $L^2$ with respect to $M$}.
Therefore it only remains to take the limit $M\to\infty$ and show that
\begin{itemize}
\item[{\bf (A4)}]
\hfill$\displaystyle
\bs{Z}_\hbeta^{(M)}
\ \xrightarrow[M\to\infty]{d} \ \bs{Z}_{\hat\beta} =
e^{\int_0^1\frac{\hat\gb }{\sqrt{1-\hat\gb^2 \,t}} \dd W_t-
\frac{1}{2} \int_0^1\frac{\hat\gb^2\ }{1-\hat\gb^2\,t} \dd t}.
$\hfill\,
\end{itemize}
\medskip

We will prove the key step {\bf (K)} in Section \ref{main_estimate}. The approximation
steps {\bf (A1)--(A4)} will be carried out in Section \ref{sec:CG}, which then implies
Theorem \ref{thm:subcritical}. The main tool to prove {\bf (K)} is a
{\em fourth moment theorem} for polynomial chaos expansions,
due to de Jong~\cite{dJ87,dJ90}, Nualart and Peccati~\cite{NP05}
and Nourdin, Peccati and Reinert~\cite{NPR10}. The following versions
is an extension to random variables with possibly unbounded third moment,
based on the Lindeberg principle proved in \cite{CSZ13}
(which extends~\cite{R79, MOO10}).

\begin{theorem}[Fourth moment theorem]\label{T:4mom}
For each $N\in\N$, let $(\eta_{N,t})_{t\in \T}$ be independent
random variables with mean $0$ and variance $1$, indexed by a countable set
$\T$. Assume that $(\eta_{N,t}^2)_{N\in\N,\, t\in\T}$
are uniformly integrable. Fix $k\in\N$ and $d_1, \ldots, d_k\in \N$. For each
$1\leq i\leq k$, let $\Phi_{N}^{(i)}(\eta_{N,\cdot})$ be a multi-linear polynomial in
$(\eta_{N,t})_{t\in \T}$ of degree $d_i$, i.e.,
$$
\Phi_{N}^{(i)}(\eta_{N,\cdot})
= \sum_{I\subset \T, \, |I|= d_i} \phi_{N}^{(i)}(I)
\prod_{t\in I} \eta_{N,t} \qquad \mbox{for some real-valued } \phi_{N}^{(i)}(\cdot).
$$
Assume further that:
\begin{itemize}
\item[\rm (i)] For all $1\leq i,j\leq k$,
$\bbE[\Phi_{N}^{(i)}(\eta_{N, \cdot})
\Phi_{N}^{(j)}(\eta_{N, \cdot})]\to V(i,j)$ for some matrix $V$ as $N\to\infty$;

\item[\rm (ii)] For each $1\leq i\leq k$, $\bbE[\Phi_{N}^{(i)}(\xi_{\cdot})^4]\to 3\,
V(i,i)^2$
as $N\to\infty$, where we have replaced $(\eta_{N,t})_{t\in \T}$ by i.i.d.\
standard normal random variables $(\xi_t)_{t\in \T}$;

\item[\rm (iii)] The maximal influence of each
variable $\eta_{N,t}$ on the polynomials of degree one
among $(\Phi_{N}^{(i)}(\eta_{N,\cdot}))_{1\leq i\leq k}$ is asymptotically negligible,
i.e., for each $1\leq i\leq k$,
\begin{equation}\label{infcond}
\max_{t\in\T} |\phi_N^{(i)} (\{t\})| \to 0\qquad  \text{as} \quad N\to\infty.
\end{equation}
\end{itemize}
Then $(\Phi_{N}^{(i)}(\eta_{N,\cdot}))_{1\leq i\leq k}$ converge in law to a
centered Gaussian vector with covariance $V$.
\end{theorem}

\begin{proof}
If we replace $(\eta_{N,t})_{t\in\bbT}$ by
standard Gaussians $(\xi_t)_{t\in\bbT_N}$,
Theorem~\ref{T:4mom} holds without the need of assuming condition (iii), thanks to
\cite[Theorem 7.6]{NPR10}, which is a multi-dimensional extension of the fourth moment
theorem \cite{dJ90, NP05}.

To justify the replacement with Gaussians, we show that the vectors
$(\Phi_{N}^{(i)}(\eta_{N,\cdot} ))_{1 \le i \le k}$ and
$(\Phi_{N}^{(i)}(\xi_\cdot ))_{1 \le i \le k}$ have the same limit in law as $N\to\infty$.
By the Cr\'amer-Wold device, it is enough to consider a linear
combination $\Phi_N = \sum_{i=1}^k c_i \Phi_N^{(i)}$, which is a multilinear
polynomial with degree $d := \max_{1 \le i \le k} d_i$ and with
variance $\sigma^2_N \le k \sum_{i=1}^k c_i^2 \sum_{I \subseteq \T} \phi_N^{(i)}(I)^2$
(by Cauchy-Schwarz).
By the Lindeberg principle in \cite[Theorem~2.6]{CSZ13}, for
any smooth and bounded $f:\R\to\R$ there is $C = C_{d, f} < \infty$ such that
for every $M > 0$
\begin{gather*}
	\Big| \bbE\big[ f\big( \Phi_{N}(\eta_{N,\cdot} ) \big] -
	\bbE\big[ f \big( \Phi_{N}(\xi_\cdot) \big] \Big| \le
	C \, \sigma_N^2 \Bigg\{ m_2^{>M} +
	 M^{d} \, \max_{t\in \T}
	\sqrt{\mathrm{Inf}_t(\Phi_N) } \Bigg\} \,, \\
	\text{where} \qquad
	m_2^{>M} := \max_{X \in \bigcup_{N\in\N, t\in\T} \{\eta_{N,t}, \, \xi_t\}}
	\bbE[X^2 \, \ind_{\{|X| > M\}}] \,, \qquad
	\mathrm{Inf}_t(\Phi_N) = \sum_{i=1}^k \sum_{I\ni t} \phi_N^{(i)} (I)^2 \,.
\end{gather*}
By the uniform integrability assumption on $\eta_{N,t}^2$,
we can fix $M > 0$ large enough so that
$m_2^{>M}$ is as small as we wish.
Since $\sup_{N\in\N} \sigma^2_N < \infty$ by assumption (i),
the proof is completed if we
show that $\max_{t\in\T} \sqrt{\mathrm{Inf}_t(\Phi_N)} \to 0$ as $N\to\infty$.
For polynomials $\Phi_N^{(i)}$ of degree $d_i = 1$
this holds by assumption (iii),
while for $d_i \ge 2$ it is a consequence of the fourth-moment assumption (ii),
as shown in \cite[Proposition~1.6 and (1.9)]{NPR10}.
\end{proof}

\section{Proof of key step for Theorem~\ref{thm:subcritical}}\label{main_estimate}

In this section we prove the key step {\bf (K)} in the proof of Theorem~\ref{thm:subcritical},
formulated in Section~\ref{ss:proofsteps},
which asserts that the building blocks of the chaos expansion have
asymptotic Gaussian behavior.
This result actually holds in great generality and is of independent interest,
so it is worth stating explicitly the assumptions we need.

We work on $\Z^d$ for fixed $d\in\N_0$
(with $\Z^0 := \{0\}$).
For every $n\in\N$, we fix a function $q_n(\cdot) \in L^2(\Z^d)$
---not necessarily a probability kernel---
and we define (cf.\ \eqref{eq:replica})
\begin{equation}\label{eq:RNq}
	\Ro_N := \sum_{n=1}^N \| q_n \|^2
	= \sum_{n=1}^N \left( \sum_{x\in\Z^d} q_n(x)^2 \right) \,.
\end{equation}
In the following sections we will focus
on the special cases when $q_n(x) = \P(S_n = x)$ or $q_n = \P(n\in\tau)$,
with $S$ or $\tau$ satisfying the local limit theorems in Hypothesis~\ref{hypoth}.
However, in this section \emph{we only need to assume that
$\Ro_N$ is a slowly varying function which diverges
as $N\to\infty$}.
The basic case to keep in mind is $\Ro_N \sim C \log N$.

Let us fix $M\in\N$ and split
\begin{align}\label{I-partition}
\{1,2,\ldots, N\} = I_1 \cup I_2 \cup \dots \cup I_M,
\end{align}
where the intervals $I_i$ are defined by \eqref{aimP}.
This definition ensures that each $I_i$ contributes equally to $\Ro_N$, since\footnote{Note
that $\Ro_n - \Ro_{n-1} = o(\Ro_n)$ as $n\to\infty$,
by the slowly varying property,
hence $\Ro_{\sft_i} \sim \frac{i}{M} \Ro_N$.}
\begin{equation}\label{eq:Rsum}
	\sum_{m \in I_i}
	\|q_m\|^2 = \Ro_{\sft_i} - \Ro_{\sft_{i-1}}
	\underset{\,N\to\infty\,}{\sim} \frac{\Ro_N}{M} \,.
\end{equation}

\begin{definition}\label{def:Efin}
Let $E^\fin :=
\bigcup_{k \in \N} E^k$ be the set of all \emph{finite sequences $\bz = (z_1, \ldots, z_k)$
taking values in a given set $E$}. For $\bz \in E^k \subseteq E^\fin$, we denote by
$|\bz| = k$ the length of $\bz$.
We also let $\N^k_\uparrow$ be the subset of \emph{increasing} sequences $\bn \in \N^k$,
i.e., $n_1 < n_2 < \ldots < n_k$, and analogously we set $\N^\fin_\uparrow
:= \bigcup_{k\in\N} \N^k_\uparrow$.

\end{definition}

We focus on the random variables $\Theta^{N;M}_{i_1, \ldots, i_k}$ introduced
in \eqref{genk}, which can be conveniently reformulated as follows.
Given $\bi \in \{1,\ldots,M\}^\fin$,
we define a set of increasing sequences \emph{$\bn \in \N^\fin_\uparrow$
that are compatible with $\bi$}, denoted by $\bn \prec \bi$, as follows:
\begin{equation}\label{eq:i<n}
	\bn \prec \bi \quad \ \iff \quad \
	|\bn| = |\bi| \quad \text{and} \quad
	n_1 - n_0 \in I_{i_1}, \ \ldots,
	\ n_{|\bi|} - n_{|\bi|-1} \in I_{i_{|\bi|}} \,,
\end{equation}
where $n_0=0$. We can then define
\begin{gather} \label{thetai}
	\Theta^{N,M}_{\bi} :=
	\left(\frac{M}{\Ro_{N} }\right)^{\frac{|\bi|}{2}}
	\sum_{\bn \in \N^\fin_\uparrow: \ \bn \prec \bi} Q_{\bn} \,,
	\qquad \text{with}\\
	\label{eq:bq0}
	Q_{\bn} := \sum_{\bx \in (\Z^d)^{|\bn|}} \,
	\prod_{j=1}^{|\bn|} q_{n_j - n_{j-1}} (x_j - x_{j-1}) \,
	\eta_{(n_j,x_j)}
	= \sum_{\bx \in (\Z^d)^{|\bn|}} q_{\bn}(\bx) \, \eta_{(\bn,\bx)} \,,
\end{gather}
where $n_0 = x_0 = 0$ and we have introduced the further abbreviations
\begin{gather}
	\label{eq:qeta}
	q_{\bn}(\bx) := \prod_{j=1}^{|\bn|} q_{n_j - n_{j-1}} (x_j - x_{j-1}) \,, \qquad
	\eta_{(\bn,\bx)} := \prod_{j=1}^{|\bn|} \eta_{(n_j,x_j)} \,.
\end{gather}
Here $(\eta_{(n,x)})_{(n,x) \in \N \times \Z^d}$
are independent random variables with
$\bbE[\eta_{(n,x)}]=0$, $\bbvar[\eta_{(n,x)}]=1$.
(In our case, cf.\ \eqref{eq:uffano}, we actually have
$\bbvar[\eta_{(n,x)}] = 1 + o(1)$ as $N\to\infty$, because
\begin{equation*}
	\bbvar[\eta_{(n,x)}]= \frac{\exp(\lambda(2\beta_N) - 2\lambda(\beta_N)) - 1}{\beta_N^2}
	= 1 + O(\beta_N) \,,
\end{equation*}
since $\lambda(\beta) = \frac{1}{2}\beta^2 + O(\beta^3)$ as $\beta \to 0$.
To lighten notation, we assume that $\bbvar[\eta_{(n,x)}]=1$.)

We allow $\eta_{(n,x)} = \eta^{(N)}_{(n,x)}$
to depend on $N\in\N$, as in \eqref{eq:uffano}.
We only need to assume that
\emph{the squares $((\eta^{(N)}_{(n,x)})^2)_{N\in\N, \,(n,x) \in \N \times \Z^d}$ are
uniformly integrable}. Note that this holds for \eqref{eq:uffano}, as one easily checks
by showing boundedness of $\bbE [(\eta^{(N)}_{(n,x)})^4]$, see \cite[eq. (6.27)]{CSZ13}.

We now state our main result, which generalizes the key step {\bf (K)} in Section \ref{ss:proofsteps}.
Recall that the space $D_M \subseteq \{1,\ldots,M\}^\fin$ of
\emph{dominated sequences} is defined as (cf.\ \eqref{eq:DM0} and \eqref{eq:DM}):
\begin{equation}\label{eq:DM2}
	D_M := \big\{ \bi \in \{1,\ldots, M\}^\fin: \ \
	i_1 > i_2, i_3, \ldots, i_{|\bi|}, \
	\ |i_j - i_{j'}| \ge 2 \ \forall j \ne j' \big\} \,.
\end{equation}

\begin{proposition}\label{thm:main1}
Assume that $\Ro_N$ in \eqref{eq:RNq}
is a slowly varying function which diverges
as $N\to\infty$. For every fixed $M\in\N$, the random variables
$(\Theta^{N,M}_{\bi})_{\bi \in D_M}$ indexed by dominated sequences in $D_M$,
converge jointly
in law as $N\to\infty$ to i.i.d.\ standard Gaussians $(\zeta_{\bi})_{\bi \in D_M}$.
\end{proposition}

\begin{proof}

We observe that, by \eqref{eq:bq0},
\begin{gather}\label{eq:Qprop}
	\bbE[Q_{\bn}] = 0 \,, \qquad
	\bbE[Q_{\bn} \, Q_{\bn'}] =
	\|q_{\bn}\|^2 \, \ind_{\{\bn = \bn'\}} \,, \\
	\label{qnvec}
	\text{where} \qquad
	\|q_{\bn}\|^2 := \prod_{j=1}^{|\bn|} \| q_{n_j - n_{j-1}} \|^2
	= \prod_{j=1}^{|\bn|} \left( \sum_{x \in \Z^d} q_{n_j - n_{j-1}} (x)^2 \right) \,.
\end{gather}
We stress that $\bbE[Q_{\bn} \, Q_{\bn'}]=0$ for $\bn \ne \bn'$,
because $\bn = (n_1, \ldots, n_{|\bn|})$ then contains some value, say $n_j$,
which does not appear
in $\bn'$ (or the other way around), so the random variables
$\eta_{(n_j, x_j)}$ appearing in the product
$Q_{\bn} \, Q_{\bn'}$
are unpaired and the expectation yields zero.

It is now easy to see that the random variables $\Theta^{N,M}_{\bi}$,
for $\bi \in D_M$,
are uncorrelated and have asymptotically (as $N\to\infty$ for fixed $M$) unit variance. In fact
\begin{equation*}
	\bbE\big[ \Theta^{N,M}_{\bi}  \Theta^{N,M}_{\bi'} \big] = 0
	\qquad \forall\, \bi \ne \bi' \,,
\end{equation*}
because if $\bn \prec \bi$ and $\bn' \prec \bi'$, then $\bn \ne \bn'$ by \eqref{eq:i<n}, and
hence $\bbE[Q_{\bn} \, Q_{\bn'}] = 0$ by \eqref{eq:Qprop}.
Next, by \eqref{thetai} and \eqref{eq:Qprop},
\begin{equation} \label{eq:varcom}
	\bbE\big[ (\Theta^{\,N,M}_{\bi} \,)^2\big]
	= \left(\frac{M}{\Ro_{N} }\right)^{|\bi|}
	\sum_{\bn \prec \bi} \|q_{\bn}\|^2
	= \left(\frac{M}{\Ro_{N} }\right)^{|\bi|}
	\prod_{j=1}^{|\bi|} \left( \sum_{m \in I(i_j)} \|q_m\|^2\right)
	\xrightarrow[\,N\to\infty\,]{} 1 \,,
\end{equation}
where in the last step we used \eqref{eq:Rsum}.

We now apply the multi-dimensional version of the fourth moment theorem, Theorem~\ref{T:4mom},
to prove that $(\Theta^{N,M}_{\bi})_{\bi\in D_M}$ converge to i.i.d.\ standard Gaussians. We have
just verified condition (i) in Theorem~\ref{T:4mom}, and the influence condition (iii) on $\Theta^{N,M}_{\bi}$ with
$|\bi|=1$ clearly holds. It only remains to verify condition (ii), i.e., assuming that $(\eta_{(n,x)})_{(n,x) \in \N \times \Z^d}$ are
i.i.d.\ standard normal, we need to show that
\begin{equation}\label{eq:tosho}
	\lim_{N\to\infty} \bbE \big[ (\Theta^{N,M}_{\bi})^4 \big] = 3
	\qquad \forall\, \bi \in D_M \,.
\end{equation}
Recalling \eqref{thetai}, we can write
\begin{equation}\label{sumTheta}
\bbE \big[(\Theta^{N;M}_{\bi})^4  \big]	
= \left(\frac{M}{\Ro_{N}}\right)^{2|\bi|}\!\!
\sum_{\ba, \bb, \bc, \bd \prec \bi} \bbE [Q_{\ba} \, Q_{\bb} \, Q_{\bc} \, Q_{\bd}] \,,
 \end{equation}
where by \eqref{eq:bq0} and \eqref{eq:qeta},
\begin{equation}\label{prodQ}
\bbE [Q_{\ba} \, Q_{\bb} \, Q_{\bc} \, Q_{\bd}] =
 \sum_{\bx, \by, \bz, \bw \in (\Z^d)^{|\bi|}} \!\!
 q_{\ba}(\bx) \, q_{\bb}(\by) \, q_{\bc}(\bz) \, q_{\bd}(\bw)
\, \bbE \big[  \eta_{(\ba,\bx)} \,
\eta_{(\bb,\by)} \, \eta_{(\bc,\bz)} \, \eta_{(\bd,\bw)} \big] \,.
 \end{equation}

Let $(\ba,\bx) = ((a_1,x_1), (a_2,x_2), \ldots, (a_{|\bi|}, x_{|\bi|}))
\in (\N \times \Z^d)^{|\bi|}$ denote the
sequence of space-time points determined by $\ba$ and $\bx$,
and let $\sfp \in \N$ be the number of \emph{distinct space-time points}
in the union of the four sequences $(\ba,\bx)$, $(\bb, \by)$, $(\bc, \bz)$, $(\bd, \bw)$:
\begin{equation}\label{eq:sfp}
	\sfp := \big| (\ba,\bx) \cup (\bb, \by) \cup (\bc, \bz) \cup (\bd, \bw) \big|
	= \sum_{(n,r) \in \N \times \Z^d} \ind_{\{(n,r) \in
	(\ba,\bx) \cup (\bb, \by) \cup (\bc, \bz) \cup (\bd, \bw)\}} \,.
\end{equation}

The first step toward \eqref{eq:tosho}
is to show that we can restrict the two sums in \eqref{sumTheta}-\eqref{prodQ}
to configurations of $(\ba,\bx), (\bb, \by), (\bc, \bz), (\bd, \bw)$ satisfying
\begin{equation} \label{eq:confi}
	\sfp = 2 |\bi| \,.
\end{equation}
Indeed, we can rule out the two cases $\sfp>2|\bi|$ and $\sfp<2|\bi|$ as follows.

\medskip

\noindent
{\bf Case 1.} $\sfp > 2 |\bi|$.
Since there are $4|\bi|$ space-time points (including multiplicity) in the four sequences
$(\ba,\bx)$, $(\bb, \by)$, $(\bc, \bz)$, $(\bd, \bw)$,
there must be at least one space-time point, say $(a_m, x_m)$,
which will not be matched in pair with one of the elements
in $(\bb, \by)\cup (\bc, \bz)\cup (\bd, \bw)$.
Then the expectation in \eqref{prodQ} vanishes because $\eta_{(a_m, x_m)}$ is not paired to
any other $\eta$ random variable in  $\eta_{(\bb,\by)}$, $\eta_{(\bc,\bz)}$ or $\eta_{(\bd,\bw)}$ (recall \eqref{eq:qeta}).
Therefore the contribution to the sum in \eqref{sumTheta} is zero in this case. \qed
\medskip

\noindent
{\bf Case 2.} $\sfp < 2|\bi|$. Recalling \eqref{eq:sfp}, for each $1\leq p< 2|\bi|$, set
\begin{equation}\label{cl2}
  \mathcal{S}_N^{(p)}:= \left(\frac{M}{\Ro_{N}}\right)^{2|\bi|}
 \!\!\!\!\!\!\!  \sumtwo{(\ba,\bx), (\bb, \by), (\bc, \bz)} {(\bd, \bw) \text{ such that }
   \sfp = p}
  \!\!  q_{\ba}(\bx) \, q_{\bb}(\by) \, q_{\bc}(\bz) \, q_{\bd}(\bw)
\, \bbE \big[  \eta_{(\ba,\bx)} \,
\eta_{(\bb,\by)} \, \eta_{(\bc,\bz)} \, \eta_{(\bd,\bw)} \big] \,.
\end{equation}
It suffices to show that $\mathcal{S}_N^{(p)} \to 0$ as $N \to \infty$,
for each $p < 2|\bi|$.

To lighten notation, we assume that $q_n(x) \ge 0$ (just
replace $q_n(x)$ by $|q_n(x)|$ in the following arguments).
Furthermore, we first consider the simplifying case when
\begin{equation}\label{eq:simple}
	q_n(x) \le 1 \,.
\end{equation}
We will use the fact that $\bbE \big[  \eta_{(\ba,\bx)} \,
\eta_{(\bb,\by)} \, \eta_{(\bc,\bz)} \, \eta_{(\bd,\bw)} \big]=0$ unless the individual $\eta$ variables match in pairs or quadruples,
since we have assumed the $\eta$'s to be i.i.d.\ standard normals in our attempt to verify Theorem~\ref{T:4mom}~(ii). In any event, note that
\begin{equation} \label{momentSn2}
	\big| \bbE \big[  \eta_{(\ba,\bx)} \,
	\eta_{(\bb,\by)} \, \eta_{(\bc,\bz)} \, \eta_{(\bd,\bw)} \big] \big| \le 3^{|\bi|}.
\end{equation}

If $\sfp = p$, then we can relabel $(\ba,\bx) \cup (\bb, \by) \cup (\bc, \bz) \cup (\bd, \bw) =
\{ (f_1, h_1), (f_2, h_2), \ldots, (f_p, h_p) \}$,
with $f_1 \le f_2 \le \cdots  \le f_p$, and we set $f_0 := 0$ (since $(f_i, h_i)$ are distinct space-time points, when $f_i = f_{i+1}$
we must have $h_i \ne h_{i+1}$). The sums in \eqref{sumTheta} and \eqref{prodQ} can then be rewritten as sums over
$(f_j, h_j)_{1\leq j\leq p}$, with another sum over all admissible assignments of $(\ba, \bx), (\bb, \by), (\bc, \bz), (\bd, \bw)$ to points in
$(f_j, h_j)_{1\leq j\leq p}$.

We start by summing over all admissible values of $(f_p, h_p)$. Denoting by
$m \in \{2, 4\}$ the number
of space-time points in $(\ba,\bx)$, $(\bb, \by)$, $(\bc, \bz)$, $(\bd, \bw)$ assigned to $(f_p,h_p)$ (for $m\in \{1,3\}$, the expectation in \eqref{prodQ} vanishes). The factors in \eqref{prodQ} involving
$(f_p,h_p)$ are
\begin{equation*}
	\prod_{i=1}^m q_{f_p - f_{r_i}}(h_p - h_{r_i})
\end{equation*}
for some $r_1, \ldots, r_m\in \{0,1,\ldots, p-1\}$.
Using the assumption \eqref{eq:simple} that $q_n\leq 1$, we get
\begin{equation}\label{Sn2-Cauchy}
  \begin{split}
   & \sum_{(f_p, h_p)} \prod_{i=1}^{m} q_{f_p-f_{r_i}}(h_p-h_{r_i})  \leq
   \sum_{(f_p,h_p)} q_{f_p-f_{r_1}}(h_p-h_{r_1})\, q_{f_p-f_{r_2}}(h_p-h_{r_2})\\
    &\qquad\qquad \leq \,\,
    \left(\sum_{(f_p,h_p)}
    q_{f_p-f_{r_1}}(h_p-h_{r_1})^2\,\right)^{1/2}
    \left(\sum_{(f_p,h_p)} q_{f_p-f_{r_2}}(h_p-h_{r_2})^2\,\right)^{1/2}\\
   &\qquad\qquad \leq \sum_{1\leq n \leq N} \sum_{x\in \bbZ^d} q^2_n(x) =\Ro_{N} \,.
 \end{split}
\end{equation}
The last inequality holds because the range of $f_p - f_{r_i}$ is contained in $\{1,\ldots, N\}$, by \eqref{eq:i<n}.

We can iterate this estimate, summing successively over $(f_{p-1},h_{p-1})$,
$(f_{p-2},h_{p-2})$, \ldots, $(f_1,h_1)$.
This, together with \eqref{momentSn2},
shows that for fixed $M\in\N$, as $N\to\infty$,
  \begin{eqnarray}\label{Sn2est}
  \mathcal{S}_N^{(p)} \leq
  3\hat\Ci \,
  \left(\frac{M}{\Ro_{N}}\right)^{2|\bi|} \Ro_{N}^{p}
  = O\Big( \Ro_{N}^{p-2|\bi|}\Big),
  \end{eqnarray}
where $\hat\Ci$ depends only on $|\bi|$ and $p$ and bounds the number of ways of assigning
$(\ba, \bx)$, $(\bb, \by)$, $(\bc, \bz)$, $(\bd, \bw)$ to $(f_\ell, h_\ell)_{1 \le \ell \le p}$.
Since $ p<2|\bi|$, relation \eqref{Sn2est}
shows that $\mathcal{S}_N^{(p)}$ converges to zero as $N$ tends to infinity.

We now show how to remove the assumption $q_n\leq 1$ in \eqref{eq:simple}. Setting
\begin{equation} \label{eq:qninf}
	\|q_n\|_\infty := \max_{x\in\Z^d} \, q_n(x) \,,
\end{equation}
the r.h.s.\ of \eqref{Sn2-Cauchy} is replaced
by $\Ro_N \big( \max_{1\leq n\leq N}\|q_n\|_\infty^{m-2} \big)$.
As we sum over $(f_{p}, h_{p})$, \ldots, $(f_1, p_1)$, we collect exactly
$4|\bi|-2p$ factors of $\max_{1\leq n\leq N}\|q_n\|_\infty$. Consequently, \eqref{Sn2est} becomes
  \begin{eqnarray}\label{Sn2estbis}
  \mathcal{S}_N^{(p)} = O\Big( \Ro_{N}^{p-2|\bi|}
  \big(\max_{1\leq n\leq N}\|q_n\|_\infty\big)^{4|\bi| - 2p}  \Big) \,.
\end{eqnarray}
However, by \eqref{eq:RNq} and \eqref{eq:qninf},
\begin{equation}\label{eq:ifty}
\max_{1\leq n\leq N}\|q_n\|_\infty^2 \le \max_{1\leq n\leq N}(\Ro_n - \Ro_{n-1}) = o(\Ro_N),
\end{equation}
since $\Ro_N$ is slowly varying and divergent. Therefore $\mathcal{S}_N^{(p)} \to 0$ also in the general case.\qed

\medskip

Continuing with the proof of \eqref{eq:tosho}, we may now restrict the sums in \eqref{sumTheta} and \eqref{prodQ}
to configurations satisfying $\sfp=2|\bi|$ (recall \eqref{eq:sfp}). This means that
\emph{the $4|\bi|$ space-time points among
$(\ba, \bx)$, $(\bb, \by)$, $(\bc, \bz)$, $(\bd, \bw)$ match exactly in pairs}
(i.e.\ coincide two by two).

As before, let $(f_i, h_i)_{1\leq i\leq p}$, with $f_1 \le f_2 \le \ldots \le f_p$ and $p = 2|\bi|$, be the distinct space-time
points occupied by $(\ba, \bx)\cup (\bb, \by)\cup (\bc, \bz)\cup(\bd, \bw)$.
In principle one could have $f_i = f_{i+1}$ (necessarily with $h_i \ne h_{i+1}$),
but such configurations give a negligible contribution in \eqref{sumTheta}, because this leaves at most
$p-1$ free coordinates $f_j$ to sum over, each of which gives by \eqref{Sn2-Cauchy} a contribution of
at most $\Ro_N$ (assuming $q_n\leq 1$; otherwise use
\eqref{eq:ifty}), while the prefactor in \eqref{cl2}
decays as $\Ro_N^{-p}$. As a consequence,
we may assume that $f_1 < f_2 < \ldots < f_p$,
which means that
\emph{the time points among $\ba, \bb, \bc, \bd$ have to match exactly
in pairs}.

\smallskip

We now make a further restriction.
Let $[\ba] := [a_1, a_{|\bi|}] \subseteq \R$ be the smallest interval
containing all the points in the (increasing) sequence $\ba = (a_1, a_2, \ldots, a_{|\bi|})$.
Then $[\ba] \cup [\bb] \cup [\bc] \cup [\bd]$ is a union of disjoint
closed intervals (\emph{connected components}) whose number can range from one to four.
We now show that we can restrict the sum
in \eqref{sumTheta} to configurations of $\ba, \bb, \bc, \bd$ with
\emph{exactly two connected components}. We distinguish between two cases.

\medskip
\noindent
{\bf Case 3}. {\em Three or four connected components.}
Since $|\ba| = |\bb| = |\bc| = |\bd| = |\bi|$, we must have
\begin{equation*}
	|\ba \cup \bb\cup \bc\cup \bd|
	\geq 3|\bi| \,,
\end{equation*}
therefore also $\sfp \geq 3|\bi|$,
cf.\ \eqref{eq:sfp}, which has been excluded in {\bf Case~1}.\qed

\medskip
\noindent
{\bf Case 4.} {\em One connected component.}
Similar to \eqref{cl2}, it suffices to focus on
\begin{equation}\label{cl2bis}
 \mathcal{\hat{S}}_N := \left(\frac{M}{\Ro_{N}}\right)^{2|\bi|}
 \!\!\!\!\!\!\!\!\!  \sumthree{(\ba,\bx), (\bb, \by), (\bc, \bz), (\bd, \bw)}
 {\text{matching in pairs and}}{\text{forming one connected component}}
  \!\!\!\!\!\!  q_{\ba}(\bx) \, q_{\bb}(\by) \, q_{\bc}(\bz) \, q_{\bd}(\bw)
\end{equation}
and show that $\mathcal{\hat{S}}_N \to 0$ as $N \to\infty$
(note that $\bbE \big[  \eta_{(\ba,\bx)} \,
\eta_{(\bb,\by)} \, \eta_{(\bc,\bz)} \, \eta_{(\bd,\bw)} \big] = 1$
because of the ``matching in pairs'' condition). We will show that the ``one connected component''
condition effectively leads to the loss of a degree of freedom in the summation.

Without loss of generality, assume that $a_1 = \min\{a_1,b_1,c_1,d_1\}$
is the smallest among all time indices in $\ba, \bb,\bc, \bd$. Then it has to match either
$b_1$, $c_1$ or $d_1$. Say $a_1=b_1=f_1$. It follows that $c_1=f_{u}$ for some
$u \in \{2, \ldots, p\}$.
The constraints of matching in pairs and $[\ba]\cup[\bb]\cup[\bc]\cup[\bd]$ having
{\it one connected component} imply that either $c_1\leq a_{\bar k}$
or $c_1\leq b_{\bar k}$ for some $\bar k \ge 2$; w.l.o.g., assume that $c_1\leq a_{\bar k}$.
Since $\ba \prec \bi$, by \eqref{aimP} and \eqref{eq:i<n}, this implies
\begin{align*}
	f_1 = a_1 \,\,<\,\, f_u = c_1 \le \,\, a_{\bar k} &\le \,\,
	a_1 + \sft_{i_2} + \sft_{i_3} + \ldots + \sft_{i_{\bar k}}\\
	&\le \,\, f_1 + (\bar k-1) \, \sft_{i_1 - 2}
	\,\,\le \,\, f_1 + |\bi| \, \sft_{i_1 - 2}  \,,
\end{align*}
where the last inequality holds because
$i_\ell \le i_1 - 2$ for all $\ell \in \{2,\ldots, |\bi|\}$,
since $\bi$ is a dominated sequence, cf.\ \eqref{eq:DM2}.
Also note $f_1 = a_1 \ge \sft_{i_1-1}$, again by \eqref{aimP} and \eqref{eq:i<n}.
Therefore
\begin{equation}\label{eq:constra}
	f_1 = a_1 \ge \sft_{i_1-1} \qquad \text{and} \qquad
	f_u = c_1 \in (f_1, f_1 + \bar m_1] \,, \qquad
	\text{where} \quad
	\bar m_1 := |\bi| \, \sft_{i_1-2} \,.
\end{equation}

We can now sum \eqref{cl2bis}
over the variables $(f_1,h_1), \ldots, (f_p, h_p)$
subject to \eqref{eq:constra} for some $2\leq u\leq p$.
The sum over $(f_p, h_p)$ has already been estimated in \eqref{Sn2-Cauchy} with $m=2$,
and is bounded by
$\Ro_N$. The same bound $\Ro_N$ applies to the sum
over $(f_{\ell}, h_{\ell})$ for each $\ell = p-1, p-2, \ldots, u+1$. The sum over $(f_u, h_u)$, in view of \eqref{eq:constra}, is bounded by
\begin{equation*}
	\sum_{f_u \in (f_1, f_1 + \bar m_1]} \,
   \sum_{h_u \in \Z^d} q_{f_u - f_{r_1}}(h_u - h_{r_1})\,
   q_{f_u-f_{r_2}}(h_u-h_{r_2}) \,,
\end{equation*}
for some $r_1, r_2 \in \{0, \ldots, u-1\}$.
Since $f_u = c_1$ is the first index of the sequence
$\bc$, we have either $r_1 = 0$ or $r_2=0$; w.l.o.g., assume $r_1=0$. We then have
(recall \eqref{eq:qeta})
\begin{equation}\label{Sn2-Cauchybis}
  \begin{split}
   & \sumtwo{f_u \in (f_1, f_1 + \bar m_1]}{h_u \in \Z^d} q_{f_u}(h_u)\, q_{f_u-f_{r_2}}(h_u-h_{r_2})\\
    &\qquad \leq \,\,
    \left(\sumtwo{f_u \in (f_1, f_1 + \bar m_1]}{h_u \in \Z^d}
    q_{f_u}(h_u)^2\,\right)^{1/2}
    \left(\sumtwo{f_u \in (f_1, f_1 + \bar m_1]}{h_u \in \Z^d}
    q_{f_u-f_{r_2}}(h_u-h_{r_2})^2\,\right)^{1/2}\\
   &\qquad \leq
   \left( \sumtwo{f_1 < n \leq f_1 + \bar m_1}{x\in \bbZ^d} q^2_n(x) \right)^{1/2}
   \left( \sumtwo{1\leq n \leq N}{x\in \bbZ^d} q^2_n(x) \right)^{1/2}
   =\sqrt{\Ro_{f_1 + \bar m_1} - \Ro_{f_1}}
   \sqrt{\Ro_N} \,.
 \end{split}
\end{equation}
Let us recall from \eqref{aimP} that $\sft_i = \sft_i^{N,M}$ satisfies
$\Ro_{\sft_i} \sim \frac{i}{M} \Ro_N$ as $N\to\infty$ (for fixed $M\in\N$).
It follows that if $j < i$, then $\sft_j = o(\sft_i)$ as $N\to\infty$.\footnote{If
$\sft_j \in [\epsilon \, \sft_i,
\sft_i]$ for $\epsilon > 0$, the slowly varying property of
$\Ro_N$ would yield $\Ro_{\sft_j} \sim \Ro_{\sft_i}$, contradicting
\eqref{aimP}.} Since $\bar m_1 = |\bi| \sft_{i_1-2}$ by \eqref{eq:constra}
while $f_1 = a_1 \ge \sft_{i_1-1}$, it follows that $\bar m_1 = o(f_1)$,
and hence $\Ro_{f_1 + \bar m_1} \sim \Ro_{f_1}$. This implies that
the r.h.s.\ of \eqref{Sn2-Cauchybis} equals
$\sqrt{o(1) \, \Ro_{f_1}} \sqrt{\Ro_N} = o(1) \, \Ro_N$
as $N\to\infty$.

We can now  sum over the remaining variables $(f_\ell, h_\ell)$
for $\ell = u-1, u-2, \ldots, 1$ as we did before, with each sum bounded by $\Ro_N$ as shown in
\eqref{Sn2-Cauchy}, which gives
\begin{equation}\label{case3b}
\mathcal{\hat{S}}_N \leq
  \hat\Ci \,
 \left(\frac{M}{\Ro_{N}}\right)^{2|\bi|}
 \, \Ro_{N}^p \, o(1)  \,,
\end{equation}
where $\hat\Ci$ is again a combinatorial factor independent of $N$.
Since $p= 2|\bi|$, for any fixed $M\in\N$,
the r.h.s.\ of \eqref{case3b} vanishes as $N\to\infty$.
\qed

\medskip

To complete the proof of \eqref{eq:tosho}, it only remains to show that \eqref{eq:tosho} holds if the joint sums in \eqref{sumTheta} and \eqref{prodQ} are restricted such that $|(\ba, \bx)\cup (\bb, \by) \cup (\bc, \bz)\cup(\bd, \bw)|=2|\bi|$ and $[\ba]\cup[\bb]\cup[\bc]\cup[\bd]$ contains two connected components.

\medskip
\noindent
{\bf Case 5.} {\em Two connected components.} In this case, $\ba, \bb, \bc, \bd$ must coincide two by two, i.e.,
\begin{equation} \label{eq:match}
	\ba = \bb \ \ \text{and} \ \ \bc = \bd \,, \qquad
	\text{or} \qquad \ba = \bc \ \ \text{and} \ \ \bb = \bd \,,  \qquad
	\text{or} \qquad \ba = \bd \ \ \text{and} \ \ \bb = \bc \,.
\end{equation}
This extends further to $(\ba, \bx)$, $(\bb, \by)$, $(\bc, \bz)$, $(\bd, \bw)$. By symmetry, each of the three cases gives the same contribution,
which leads to the factor $3$ in the r.h.s.\ of \eqref{eq:tosho}.
We can thus focus on the case $(\ba, \bx) = (\bb, \by)\neq (\bc, \bz) = (\bd, \bw)$.

Restricting the sums in \eqref{sumTheta} and \eqref{prodQ} to
$(\ba,\bx) = (\bb,\by) \ne (\bc,\bz) = (\bd, \bw)$, we obtain
\begin{equation} \label{eq:quala}
\left(\frac{M}{\Ro_{N}}\right)^{2|\bi|}
	\sumtwo{\ba, \bc \prec \bi, \ \bx, \bz \in (\Z^d)^{|\bi|}}{[\ba] \cap [\bc]=\emptyset}
	q_{\ba}(\bx)^2 \, q_{\bc}(\bz)^2 \,.
\end{equation}
Note that if we ignore the restriction
$[\ba] \cap [\bc]=\emptyset$, then the sum factorizes and we obtain
\begin{equation} \label{eq:quala2}
	\left(\frac{M}{\Ro_{N}}\right)^{2|\bi|}  \left(\sum_{\ba \prec \bi, \ \bx \in \Z^d}
	q_{\ba}(\bx)^2\right)
	\left(\sum_{\bc \prec \bi, \ \bz \in \Z^d} q_{\bc}(\bz)^2\right)
	\xrightarrow[\,N\to\infty\,]{} 1
\end{equation}
by the same variance calculation as in \eqref{eq:varcom}. This proves \eqref{eq:tosho}, because the
terms in \eqref{eq:quala} with $[\ba] \cap [\bc]\neq \emptyset$ are negligible by the same bounds as
in {\bf Case 4} (cf.~\eqref{Sn2-Cauchybis}), where $[\ba]\cup[\bb]\cup[\bc]\cup[\bd]$ contains
one connected component.
\end{proof}

\section{Proof of Theorem \ref{thm:subcritical}}
\label{sec:CG}

In this section we will first prove the approximation steps {\bf (A1)--(A4)} outlined in
Section~\ref{ss:proofsteps}, and then conclude the proof of Theorem~\ref{thm:subcritical}.

Recall that the first step {\bf (A1)} enlarges the range of summation for $Z_N^{(k)}$ in \eqref{eq:Zpoly0} to $1\leq n_1, n_2-n_1, \ldots, n_k-n_{k-1}\leq N$.
\begin{lemma}[Approximation (A1)]\label{lem:appr1}
For each $k\in\N$, let $Z_N^{(k)}$ be as in \eqref{eq:Zpoly0}, and let
\begin{equation}\label{hatZNk}
\widehat Z_N^{(k)} := \frac{1}{\Ro_N^{k/2}}
	\sumtwo{1\leq n_1, n_2-n_1, \ldots, n_k-n_{k-1}\leq N}{z_1, z_2, \ldots, z_k \in \bbZ^d}
	\, \prod_{j=1}^k q_{n_j-n_{j-1}}(z_j-z_{j-1})
	\, \prod_{i=1}^k \eta_{(n_i,z_i)}.
\end{equation}
Then
\begin{equation}\label{ZhatZdiff}
\lim_{N\to\infty} \, \bbE \big[ ( \widehat Z_N^{(k)}-Z_N^{(k)})^2\big] = 0.
\end{equation}
\end{lemma}
\begin{proof}
Recall from \eqref{eq:bq0} that for $\bn=(n_1, \ldots, n_k)\in \N^k_\uparrow :=\{\bn\in \N^k: n_1<n_2<\cdots <n_k\}$,
\begin{equation}\label{Qbn2}
	Q_{\bn} := \sum_{\bx \in (\Z^d)^{|\bn|}} \,
	\prod_{j=1}^{|\bn|} q_{n_j - n_{j-1}} (x_j - x_{j-1}) \,
	\eta_{(n_j,x_j)}
	\qquad \text{(with $n_0 = x_0 = 0$)}\,.
\end{equation}
We can then write the difference
 \begin{eqnarray*}
\widehat Z_N^{(k)}-Z_N^{(k)}
= \frac{1}{\Ro_N^{k/2}} \sum_{\bn\in\bbN^k_\uparrow}   \big( \ind_{1\leq n_1-n_0, \ldots, n_k-n_{k-1}\leq N} - \ind_{0<n_1<\cdots< n_k \leq N} \big) Q_{\bn}.
\end{eqnarray*}
Since $\E[Q_{\bn}Q_{\bn'}] = \ind_{\{\bn=\bn'\}} \|q_{\bn}\|^2$ for $\bn, \bn'\in\bbN^k_\uparrow$ by \eqref{eq:Qprop}, we observe that
$$
\bbE \big[ ( \widehat Z_N^{(k)}-Z_N^{(k)})^2\big] = \bbE \big[ ( \widehat Z_N^{(k)})^2\big] - \bbE \big[(Z_N^{(k)})^2\big] -2
\bbE\big[(\widehat Z_N^{(k)}-Z_N^{(k))})Z_N^{(k)}\big]= \bbE \big[ ( \widehat Z_N^{(k)})^2\big] - \bbE \big[(Z_N^{(k)})^2\big] .
$$
On the other hand, recalling \eqref{eq:replica},
\begin{equation}\label{eq:norma1}
\begin{split}
	\bbE\big[(Z_N^{(k)})^2\big] \leq \bbE\big[(\widehat Z_N^{(k)})^2\big]
	& = \frac{1}{\Ro_N^{k}} \sumtwo{1\leq
	n_1, n_2-n_{1}, \ldots, n_k - n_{k-1} \leq N}{z_1, z_2, \ldots, z_k \in \bbZ^d}
	\, \prod_{j=1}^k q_{n_j-n_{j-1}}(z_j-z_{j-1})^2 \\
	& = \frac{1}{\Ro_N^{k}} \Bigg( \sumtwo{1 \le n \le N}{z \in \Z^d} q_n(z)^2
	\Bigg)^k = 1.
\end{split}
\end{equation}
To prove \eqref{ZhatZdiff}, it then suffices
to show that $\liminf_{N\to\infty} \bbE\big[(Z_N^{(k)})^2\big]\geq 1$, which holds since
\begin{eqnarray*}
\bbE\big[(Z_N^{(k)})^2\big]
&=& \frac{1}{\Ro_N^k} \sum_{\bn \in \bbN^k_\uparrow}  \ind_{0<n_1<\cdots< n_k \leq N} \,\|q_{\bn}\|^2  \\
&\geq& \frac{1}{\Ro_N^k} \sum_{\bn\in \bbN^k_\uparrow} \ind_{1\leq n_1-n_0, \ldots,
n_k-n_{n-1} \leq \frac{N}{k}} \|q_{\bn}\|^2 = \frac{\Ro_{N/k}^k}{\Ro_N^k},
\end{eqnarray*}
which tends to $1$ as $N\to\infty$ by the assumption that $\Ro_N$ is slowly varying in $N$.
\end{proof}
\medskip

The approximation step {\bf (A2)} in Section \ref{ss:proofsteps} bounds the contributions of near-diagonal terms when the summations in $\widehat Z_N^{(k)}$ in \eqref{hatZNk} are divided into blocks.

\begin{lemma}[Approximation (A2)]\label{lem:appr2} Recall from \eqref{genk} the definition of the block variables
\begin{equation}\label{ThetaNMi}
\Theta^{N;M}_{\bi} :=
	\left(\frac{M}{\Ro_N}\right)^{k/2} \!\!\!\!
	\sum_{n_1 \in I_{i_1}, \, \ldots, \, n_k - n_{k-1} \in I_{i_k}} Q_{\bn}, \quad \quad \bi=(i_1, \ldots, i_k)\in \{1, \ldots, M\}^k,
\end{equation}
with $Q_{\bn}$ as in \eqref{Qbn2}, and $I_i= \big(\sft_{i-1}, \sft_i \big]$ defined as in \eqref{aimP} such that $\Ro_{\sft_i} \sim  \tfrac{i}{M}\,\Ro_{N}$. Then
\begin{equation}
\lim_{M\to\infty} \limsup_{N\to\infty}
\bbE\bigg[\bigg(\sum_{\bi \in \{1, \ldots, M\}^k \setminus \{1, \ldots, M\}^k_\sharp}
\frac{1}{M^{\frac{k}{2}}} \,\Theta^{N;M}_{\bi}\bigg)^2\bigg] = 0,
\end{equation}
where $\{1,...,M\}^k_\sharp$ was defined in \eqref{eq:DM0}, which consists of $\bi$ with $|i_j -i_{j'}|\geq 2$ for all $j\neq j'$.
\end{lemma}

\begin{proof}
Denote $\{1,...,M \}^k_{*}:= \{1,...,M \}^k\backslash \{1,...,M \}^k_\sharp$. Note that
\begin{align*}
\bbE\bigg[\bigg(\sum_{\bi \in \{1,...,M \}^k_{*}}
\frac{1}{M^{\frac{k}{2}}} \,\Theta^{N;M}_{\bi}\bigg)^2\bigg]
= \frac{1}{\Ro^k_N} \sum_{\bi\in \{1,...,M \}^k_{*} }\sum_{n_1 \in I_{i_1}, \, \ldots, \, n_k - n_{k-1} \in I_{i_k}} \|q_{\bn}\|^2.
\end{align*}
Recall from \eqref{eq:Rsum} that $\sum_{m\in I_i}\|q_m\|^2\sim \Ro_N/M$ as $N\to\infty$, while $\|q_{\bn}\|^2 = \prod_{j=1}^{|\bn|} \| q_{n_j - n_{j-1}} \|^2$,
we can therefore sum  $n_k,n_{k-1},...,n_1$ successively to obtain
\begin{align*}
\limsup_{N\to\infty} \bbE\bigg[\bigg(\sum_{\bi \in \{1,...,M \}^k_{*}}
\frac{1}{M^{\frac{k}{2}}} \,\Theta^{N;M}_{\bi}\bigg)^2\bigg] &\leq
\sum_{\bi \in \{1, \ldots, M\}^k_*}\, \frac{1}{M^{k}},
\end{align*}
which tends to $0$ as $M\to\infty$, since the constraint $\bi \in  \{1,...,M \}^k  \setminus \{1,...,M \}^k_{\sharp}$ reduces the number of free indices in $\bi=(i_1, \ldots, i_k)$.
\end{proof}
\medskip

The approximation step {\bf (A3)} concerns the asymptotic factorization of $\Theta^{N;M}_{\bi}$ into a product of $\Theta^{N;M}_{\bi^{(j)}}$, indexed by dominated sequences $\bi^{(1)}, \ldots, \bi^{(\mathfrak{m})}$ forming $\bi$ (cf.~\eqref{1domseq}). Recall that each $\bi=(i_1, \ldots, i_k)\in \{1, \ldots, M\}^k_\sharp$ can be divided into $\mathfrak{m}=\mathfrak{m}(\bi)$ consecutive dominated sequences $\bi^{(1)}:=(i_1, \ldots, i_{\ell_2-1})$, $\bi^{(2)}:=(i_{\ell_2}, \ldots, i_{\ell_3-1})$, \ldots, $\bi^{(\mathfrak{m})}:=(i_{\ell_{\mathfrak{m}}}, \ldots, i_k)$,  where $i_{ \ell_1}=i_1<i_{\ell_2}< \cdots< i_{\ell_{\mathfrak{m}}}$ are the successive running maxima of $(i_1, \ldots, i_k)$.

\begin{lemma}[Approximation (A3)]\label{lem:appr3}
For each $\bi=(i_1, \ldots, i_k)\in \{1...,M\}^k_{\sharp}$, we have
\begin{equation}\label{eq:appr3}
\lim_{N\to\infty} \bbE\Big[\big(\Theta^{N;M}_{\bi} - \Theta^{N;M}_{\bi^{(1)}} \Theta^{N;M}_{\bi^{(2)}}\cdots \Theta^{N;M}_{\bi^{(\mathfrak{m}(\bi))}}\big)^2\Big] = 0,
\end{equation}
where
$ (\bi^{(1)},...,\bi^{(\mathfrak{m}(\bi))})$ is the decomposition of $\bi$ into dominated sequences.
\end{lemma}
\begin{proof}
We first prove \eqref{eq:appr3} for $\mathfrak{m}(\bi)=2$, with $\ell_1=1$ and $\ell_2$ denoting the indices of the two running maxima of $\bi$. Recall that
\begin{equation}\label{Thetagain}
\Theta^{N;M}_{i_1,...,i_k} =
	\left(\frac{M}{\Ro_N}\right)^{k/2} \!\!\!\!
	\sumtwo{n_1 \in I_{i_1}, \, \ldots, \, n_k - n_{k-1} \in I_{i_k}}
	{x_1, \ldots, x_k \in \Z^d}
	\, \prod_{j=1}^k q_{n_j-n_{j-1}}(x_j-x_{j-1})
	\, \prod_{i=1}^k \eta_{(n_i,x_i)}.
\end{equation}
Note that if we replace $q_{n_{\ell_2}-n_{\ell_2-1}}(x_{\ell_2}-x_{\ell_2-1})$ by $q_{n_{\ell_2}}(x_{\ell_2})$ and replace the range of summation $n_{\ell_2}-n_{\ell_2-1}\in I_{i_{\ell_2}}$ by $n_{\ell_2}\in I_{i_{\ell_2}}$, then the above expression for $\Theta^{N;M}_{\bi}$ becomes that for  $\Theta^{N;M}_{\bi^{(1)}} \Theta^{N;M}_{\bi^{(2)}}$. We will show that these replacements are justified because using that $\ell_2$ is a running maximum of $\bi$,
one has $n_{\ell_2} \gg n_{\ell_2-1}$ and the local limit theorem of Hypothesis \ref{hypoth} can then be applied to replace $q_{n_{\ell_2}-n_{\ell_2-1}}(x_{\ell_2}-x_{\ell_2-1})$ by $q_{n_{\ell_2}}(x_{\ell_2})$.

First note that the summands in \eqref{Thetagain} for $\Theta^{N;M}_{\bi}$  are all orthogonal, and the dominant $L^2$ contribution comes from $x_1,\ldots, x_k$ with $|x_j-x_{j-1}|$ of the order
$$
\phi(n_j -n_{j-1}):= ((n_j-n_{j-1})L(n_j-n_{j-1})^2)^{1/d}
$$
for each $1\leq j\leq k$. Indeed, by the local limit theorem of Hypothesis \ref{hypoth} and a Riemann sum approximation,
$$
\frac{\bbE\Big[ \Big(\sum_{|x| \leq K \phi(n)} q_n(x) \eta_{(n,x)}\Big)^2\Big]}{\bbE\Big[ \Big(\sum_{x} q_n(x) \eta_{(n,x)}\Big)^2\Big] }  = \frac{\sum_{|x| \leq K \phi(n)} q_n(x)^2}{\sum_{x} q_n(x)^2}
\asto{n} \frac{\int_{|x|\leq K} g^2(x) \dd x}{\int g^2(x) \dd x} \asto{K} 1.
$$
Therefore by choosing $K$ large, we can approximate $\Theta^{N;M}_{\bi}$ arbitrarily closely in $L^2$ by
$$
\widetilde \Theta^{N;M}_{\bi} =
	\left(\frac{M}{\Ro_N}\right)^{k/2} \!\!\!\!
	\sumtwo{n_1 \in I_{i_1}, \, \ldots, \, n_k - n_{k-1} \in I_{i_k}}
	{|x_1|\leq K\phi(n_1), \ldots, |x_k-x_{k-1}|\leq K\phi(n_k-n_{k-1})}
	\, \prod_{j=1}^k q_{n_j-n_{j-1}}(x_j-x_{j-1})
	\, \prod_{i=1}^k \eta_{(n_i,x_i)}.
$$
Similarly we can approximate $\Theta^{N;M}_{\bi^{(1)}} \Theta^{N;M}_{\bi^{(2)}}$ arbitrarily closely in $L^2$ by $\widetilde \Theta^{N;M}_{\bi^{(1)}} \widetilde \Theta^{N;M}_{\bi^{(2)}}$, which differs from $\widetilde \Theta^{N;M}_{\bi}$ in:
\begin{itemize}
\item the factor $q_{n_{\ell_2}}(x_{\ell_2})$ instead of
$q_{n_{\ell_2}-n_{\ell_2-1}}(x_{\ell_2}-x_{\ell_2-1})$;

\item the range of summation
$n_{\ell_2}\in I_{i_{\ell_2}}$ and $|x_{\ell_2}|\leq K\phi(n_{\ell_2})$,
instead of $n_{\ell_2}-n_{\ell_2-1}\in I_{i_{\ell_2}}$ and
$|x_{\ell_2}-x_{\ell_2-1}|\leq K\phi(n_{\ell_2}-n_{\ell_2-1})$.
\end{itemize}

We now show that these differences are negligible in $L^2$ contributions. By assumption,
$$
n_j-n_{j-1} \in I_{i_j}=\big(\sft_{i_{j}-1}, \sft_{i_{j}} \big] \quad \mbox{for all }  1\leq j\leq \ell_2-1,
$$
where $\sft_a$ is chosen with $\Ro_{\sft_{a}} \sim  \tfrac{a}{M}\,\Ro_{N}$. Since $\Ro_N$ is slowly varying and divergent, we have $\sft_1\ll \sft_2\ll \sft_3\ll \cdots$ as $N\to\infty$. In particular, we have the uniform bound
\begin{equation}\label{nellbd}
n_{\ell_2-1} = \sum_{j=1}^{\ell_2-1} (n_j-n_{j-1}) \leq \sum_{j=1}^{\ell_2-1} \sft_{i_j}
= O(t_{i_{1}}) = o(\sft_{i_{\ell_2}-1}),
\end{equation}
where the last bound holds because the assumption $\bi\in \{1, \ldots, M\}^k_\sharp$ and
$\ell_2$ being a running maximum ensures that $i_j
\le i_{1} <i_{\ell_2}-1$ for all $1\leq j\leq \ell_2-1$. Therefore when we switch from
the range of summation in $\widetilde \Theta^{N;M}_{\bi}$ from
$n_{\ell_2}\in n_{\ell_2-1}+(\sft_{i_{\ell_2}-1}, \sft_{i_{\ell_2}}]$  to
$n_{\ell_2}\in (\sft_{i_{\ell_2}-1}, \sft_{i_{\ell_2}}]$, the difference is negligible in $L^2$
as $N\to\infty$.

Similarly, we have the uniform bound
\begin{equation}\label{xellbd}
| x_{\ell_2-1}| \leq \sum_{j=1}^{\ell_2-1} |x_j-x_{j-1}| \leq K \sum_{j=1}^{\ell_2-1} \phi(n_j-n_{j-1}) \leq K \sum_{j=1}^{\ell_2-1}
\phi(\sft_{i_j})  \ll \phi(\sft_{i_{\ell_2}-1}),
\end{equation}
and when we switch the range of summation in $\widetilde \Theta^{N;M}_{\bi}$ from $|x_{\ell_2}- x_{\ell_2-1}| \leq K\phi(n_{\ell_2}-n_{\ell_2-1})$ to $|x_{\ell_2}|\leq K\phi(n_{\ell_2})$, the difference is again negligible in $L^2$ as $N\to\infty$
(recall that by construction $I_{i_{\ell_2}} \ni n_{\ell_2} - n_{\ell_2-1}
\ge \sft_{i_{\ell_2}-1}$).

Having justified the switch of the range of summation for $n_{\ell_2}$ and $x_{\ell_2}$ in $\widetilde \Theta^{N;M}_{\bi}$ to $n_{\ell_2}\in I_{i_{\ell_2}}$ and $|x_{\ell_2}|\leq K\phi(n_{\ell_2})$, we note finally that switching $q_{n_{\ell_2}-n_{\ell_2-1}}(x_{\ell_2}-x_{\ell_2-1})$ to $q_{n_{\ell_2}}(x_{\ell_2})$ also leads to a negligible difference in $L^2$ as $N\to\infty$, because uniformly in $x_{\ell_2-1}$ and $n_{\ell_2-1}$ with bounds as in \eqref{nellbd} and \eqref{xellbd}, and uniformly in $n_{\ell_2}\in I_{i_{\ell_2}}$ and $|x_{\ell_2}|\leq K\phi(n_{\ell_2})$, we have
\begin{equation}
\Big|\frac{q_{n_{\ell_2}-n_{\ell_2-1}}(x_{\ell_2}-x_{\ell_2-1})}{q_{n_{\ell_2}}(x_{\ell_2})} -1\Big| \asto{N} 0,
\end{equation}
which follows readily from the local limit theorem for $q(\cdot)$ in Hypothesis \ref{hypoth}.

This completes the proof of \eqref{eq:appr3} when $\bi$ has two running maxima. In general when $\bi$ has $\mathfrak{m}$ running maxima, occurring at indices $\ell_1=1$, $\ell_2, \ldots, \ell_{\mathfrak{m}}$, the argument is the same: we just replace
$q_{n_{\ell_j}-n_{\ell_j-1}}(x_{\ell_j}-x_{\ell_j-1})$ by $q_{n_{\ell_j}}(x_{\ell_j})$ and replace the range of summation $n_{\ell_j}-n_{\ell_j-1}\in I_{i_{\ell_j}}$ by $n_{\ell_j}\in I_{i_{\ell_j}}$, one $j$ at a time.
\end{proof}
\medskip

As explained in Section~\ref{ss:proofsteps}, for a fixed $M\in\N$, which is the number of blocks $(I_i)_{1\leq i\leq M}$ that partition $[1,N]$ (cf. \eqref{I-partition}),
 the polymer partition function $Z^\omega_{N, \gb_N}$ (with $\gb_N=\hbeta/\sqrt{\Ro_N}$ for some $\hbeta<1$) is approximated in distribution in the $N\to\infty$ limit
by the random variable $\bs{Z}_\hbeta^{(M)}$ in \eqref{eq:cutofff}.
The last step {\bf (A4)} is to show that as $M\to\infty$, $\bs{Z}_\hbeta^{(M)}$ converges to the
log-normal random variable $\bs{Z}_\hbeta$ in Theorem \ref{thm:subcritical}.

\begin{lemma}[Step (A4)]\label{prop:limCG}
Let $(\zeta_{\bi})_{\bi \in D_M}$ be i.i.d.\ standard normal random variables indexed by finite dominated sequences in $D_M$ as defined in \eqref{eq:DM}. Let $\hbeta\in (0,1)$, and let
\begin{equation}\label{bsZhbeM}
	\bs{Z}_\hbeta^{(M)}
	:= 1 + \sum_{k=1}^{\infty} \  \sum_{\bi \in \{1, \ldots, M\}^k_\sharp} \
	\frac{\hbeta^k}{M^{\frac{k}{2}}} \prod_{l=1}^{\mathfrak{m}(\bi)} \zeta_{\bi^{(l)}} \,,
\end{equation}
where $(\bi^{(1)}, \ldots, \bi^{(\mathfrak{m}(\bi))})$ is the decomposition of $\bi$ into dominated sequences. Then
\begin{equation}\label{A4again}
\bs{Z}_\hbeta^{(M)}
\xrightarrow[M\to\infty]{d}
\bs{Z}_\hbeta = \exp\left( \int_0^1\frac{\hat\gb  }{\sqrt{1-\hat\gb^2\,t}} \dd W(t)-
\frac{1}{2} \int_0^1\frac{\hat\gb^2}{1-\hat\gb^2\,t} \dd t \right),
\end{equation}
where $W$ is a standard one dimensional Wiener process.
\end{lemma}
\begin{proof}
Grouping $\bi=(i_1, \ldots, i_k)$ according to the indices of its running maxima $\ell_1=1<\ell_2<\cdots \ell_m\leq k$,
as well as the values of the running maxima $1\leq i_{\ell_1}<\cdots <i_{\ell_m}\leq M$, which we denote by $\bi \sim (\vec \ell, i_{\vec \ell})$, we can write (with $\ell_{m+1}:=k+1$)
\begin{equation}\label{bsZhb2}
\bs{Z}_\hbeta^{(M)} = 1+
\sum_{k=1}^\infty\sum_{m=1}^k \frac{\hat\gb^k}{M^{\frac{k}{2}}}\,
\sumtwo{1=\ell_1<\cdots<\ell_m\leq k}{1\leq i_{\ell_1}<i_{\ell_2}<\cdots <i_{\ell_m}\leq M }
\sum_{\bi\in \{1,\ldots, M\}^k_\sharp \atop \bi \sim (\vec \ell, i_{\vec \ell})} \prod_{j=1}^m \zeta_{(i_{\ell_j},...,i_{\ell_{j+1}-1})}.
\end{equation}
Let us replace the constraints $\bi\in \{1,\ldots, M\}^k_\sharp$ and $\bi \sim (\vec \ell, i_{\vec \ell})$ by sum over
$(i_{\ell_j+1}, \ldots, i_{\ell_{j+1}-1})\in \{1,\ldots, i_{\ell_j}-1\}^{\ell_{j+1}-\ell_j-1}$ for each $1\leq j\leq m$, i.e., approximate $\bs{Z}_\hbeta^{(M)}$ by
\begin{equation}
\bs{\widehat Z}_\hbeta^{(M)} = 1+
\sum_{k=1}^\infty\sum_{m=1}^k \frac{\hat\gb^k}{M^{\frac{k}{2}}}\,
\sumtwo{1=\ell_1<\cdots<\ell_m\leq k}{1\leq i_{\ell_1}<i_{\ell_2}<\cdots <i_{\ell_m}\leq M }
\sumthree{(i_{\ell_{r}+1}, \ldots, i_{\ell_{r+1}-1})}{ \in
\{1,\ldots, i_{\ell_{r}}-1\}^{\ell_{r+1}-\ell_{r}-1}}
{\text{for } r=1,\ldots, m}
\prod_{j=1}^m \zeta_{(i_{\ell_j},...,i_{\ell_{j+1}-1})},
\end{equation}
where we have extended the i.i.d.\ family $(\zeta_{\bi})_{\bi \in D_M}$ to include new independent standard normals $\zeta_{(a_1, \ldots, a_r)}$ indexed by dominated sequences $(a_1, \cdots, a_r)\in \{1, \ldots, M\}^r\backslash \{1, \ldots, M\}^r_\sharp$.

Note that $\bs{\widehat Z}_\hbeta^{(M)}$ contains more summands than $\bs{Z}_\hbeta^{(M)}$, and the summands are orthogonal.
A simple calculation shows that both $\Vert \bs{\widehat Z}_\hbeta^{(M)}\Vert_2^2$ and $\Vert \bs{Z}_\hbeta^{(M)}\Vert_2^2$ tend to
$1+\sum_{k=1}^\infty \hbeta^{2k}= (1-\hbeta^2)^{-1}$ as $M\to\infty$. Therefore
$$
\Vert \bs{\widehat Z}_\hbeta^{(M)} - \bs{Z}_\hbeta^{(M)}\Vert_2^2\asto{M} 0.
$$

For $a\in \N$ and $r\in\N$, let us now denote
 \begin{equation}\label{xra}
	\xi_r(a) := \sum_{(a_2, \ldots, a_r) \in \{1, \ldots, a-1\}^{r-1}}
	\zeta_{(a,a_2, \ldots, a_r)}.
\end{equation}
Denoting $r_j:=\ell_{j+1}-\ell_j$ (with $\ell_{m+1}:=k+1$) and $a_j:=i_{\ell_j}$, we can then rewrite $\bs{\widehat Z}_\hbeta^{(M)}$ as
\begin{eqnarray}\label{bigXi}
\bs{\widehat Z}_\hbeta^{(M)} & = & 1+
\sum_{k=1}^\infty\sum_{m=1}^k \frac{\hat\gb^k}{M^{\frac{k}{2}}}\,
\sumtwo{1=\ell_1<\cdots<\ell_m\leq k}{1\leq a_1< a_2<\cdots < a_m\leq M }
\prod_{j=1}^m \xi_{\ell_{j+1}-\ell_j}(a_j) \nonumber \\
&=&1+
\sum_{k=1}^\infty\sum_{m=1}^k \frac{\hat\gb^k}{M^{\frac{k}{2}}}\,
\sumtwo{r_1,\dots r_m\in \bbN}{ r_1+\cdots+ r_m=k }\,\, \sum_{1\leq a_1<a_2<\cdots <a_m\leq M}
 \prod_{j=1}^m \,\,\xi_{r_{j}}(a_j)  \nonumber\\
&=&1 + \sum_{m=1}^\infty \
	\sum_{r_1, \ldots, r_m \in \N}
	\sum_{1 \le a_1 < a_2 < \ldots < a_m \le M} \  \prod_{j=1}^m
	 \frac{\hbeta^{r_j}}{M^{\frac{r_j}{2}}}\ \xi_{r_j}(a_j)  \nonumber\\
&=& 1 + \sum_{m=1}^\infty \
	\sum_{r_1, \ldots, r_m \in \N}
	\sumtwo{0 < t_1 < t_2 < \ldots < t_m \le 1}{t_1, \ldots, t_m \in \frac{1}{M}\N}
	\ \prod_{j=1}^m \frac{\hbeta^{r_j}}{M^{\frac{r_j}{2}}}\ \xi_{r_j}(M t_j)  \nonumber\\
&=& 1 + \sum_{m=1}^\infty
	\sumtwo{0 < t_1 < t_2 < \ldots < t_m \le 1}{t_1, \ldots, t_m \in \frac{1}{M}\N}
 	\ \prod_{j=1}^m \bigg\{\sum_{r\in\bbN}\frac{\hbeta^{r}}{M^{\frac{r}{2}}}\ \xi_{r}(M t_j)\bigg\}, \label{bsZhb3}
\end{eqnarray}
where we could interchange summations because the series is $L^2$ convergent when $\hbeta\in (0,1)$.

We note that $(\hat\gb/\sqrt{M})^r \xi_r(Mt)$ are independent normal random variables for different values of $r\in\bbN$ and $t\in M^{-1}\bbN$, and hence the collection of random variables
\[
\Xi_{M,t} := \sum_{r\in\bbN}\frac{\hbeta^{r}}{M^{\frac{r}{2}}}\ \xi_{r}(M t), \qquad
t\in (0,1]\cap \frac{1}{M}\N,
\]
are also independent normal with mean zero and variance
$$
{\mathbb V}{\rm ar}(\Xi_{M,t}) =  \sum_{r\in\bbN}\frac{\hbeta^{2r}}{M^r}
{\mathbb V}{\rm ar}(\xi_r(Mt)) = \sum_{r\in\N} \frac{\hbeta^{2r}}{M^r} (Mt-1)^{r-1} = \frac{\hbeta^2}{M} \cdot \frac{1+\epsilon_M(t)}{1-\hbeta^2 t},
$$
where
\begin{align*}
|\epsilon_M(t)|= \frac{1}{M(1-\hat\gb^2 t)+1} \leq  \frac{1}{M(1-\hat\gb^2 )+1},
\end{align*}
which tends to $0$ uniformly in $t\in [0,1]$ as $M$ tends to $\infty$, provided $\hbeta<1$.
Therefore we can represent $\Xi_{M,t}$ in terms of a standard Wiener process $W$:
 \begin{align}\label{Xiencode}
\Xi_{M,t} = \frac{\hbeta (1+\epsilon_M(t))}{\sqrt{1-\hbeta^2 \, t}}
 \int_{t-\frac{1}{M}}^{t} \dd W_s, \qquad t\in [0,1]\cap \frac{1}{M}\N.
\end{align}
We can then write
\begin{align}\label{hatZMchaos}
\bs{\widehat Z}_\hbeta^{(M)} =
1 + \sum_{m=1}^\infty
	 \sumtwo{0 < t_1 < t_2 < \ldots < t_m \le 1}{t_1, \ldots, t_m \in \frac{1}{M}\N}  \,\,\prod_{j=1}^m
	 \frac{\hbeta(1+\epsilon_M(t))}{\sqrt{1-\hbeta^2\, t_j}} \int_{t_{j}-\frac{1}{M}}^{t_j} \dd W_s.
\end{align}
For $\hbeta<1$, it is easily seen that
$$
\bs{\widehat Z}_\hbeta^{(M)} \xrightarrow[\,M\to\infty\,]{L^2}  1 + \sum_{m=1}^\infty \ \
	\idotsint\limits_{0 < t_1 < \ldots < t_m < 1} \
	\prod_{j=1}^m \frac{\hbeta}{\sqrt{1-\hbeta^2\, t_j}} \dd W_{t_j}
	=  : \exp \Bigg\{ \int_0^1  \frac{\hbeta}{\sqrt{1-\hbeta^2\, t}} \dd W(t) \Bigg\} : \ ,
$$
where the last equality holds by the properties of the Wick exponential \cite[\S 3.2]{J97}.
Since the last expression is precisely $\bs{Z}_\hbeta$, the proof is completed.
\end{proof}

\begin{proof}[Proof of Theorem \ref{thm:subcritical}] When $\hbeta\in (0,1)$,  the convergence
of $Z^{\omega}_{N, \beta_N}$ to $\bs{Z}_\hbeta$ follows readily from the approximation
steps {\bf (A1)--(A4)} and the key step {\bf (K)}, as explained in Section \ref{ss:proofsteps}.
The convergence of the second moment
$\bbE[(Z_{N,\beta_N}^\omega)^2]\to\bbE[(\bs{Z}_{\hat\beta})^2] = \frac{1}{1-\hbeta^2}$
for $\hbeta\in (0,1)$ is a simple calculation, using
\eqref{ZA1K} and $\bbE\big[(\widehat Z_N^{(k)})^2\big] = 1$ (recall \eqref{eq:norma1}).

When $\hbeta\geq 1$, the convergence in law $Z^{\omega}_{N, \beta_N} \to 0$
follows a standard argument, which we include for completeness. Note that it suffices to show that for some $\theta\in(0,1)$, the fractional moment $\bbE[(Z_{N,\gb_N}^\go)^\theta]$ converges to zero as $N\to\infty$.

First we show that $\bbE[(Z_{N,\gb}^\go)^\theta]$ is non-increasing in $\gb$. Indeed,
\begin{align*}
\frac{\dd\ }{\dd \gb} \bbE[(Z_{N,\gb}^\go)^\theta] &=
  \theta \bbE\bigg[ \,\E \bigg[\sum_{n=1}^N (\go_{n,x_n} -\gl'(\gb)) \,e^{\sum_{n=1}^N({\gb\go_{n,x_n}-\gl(\gb))}} \bigg] \,(Z_{N,\gb}^{\go})^{\theta-1}\bigg]\\
   &=\theta \sum_{n=1}^N \E \bigg[ \,\bbE\Big[  \big(\go_{n,x_n}-\gl'(\gb)\,\big) \,e^{\sum_{i=1}^N({\gb\go_{i,x_i}-\gl(\gb))}}  \,(Z_{N,\gb}^{\go})^{\theta-1}\Big]\bigg] \\
   &=\theta \sum_{n=1}^N \E \bigg[ \,\widetilde\bbE\Big[  \big(\go_{n,x_n}-\gl'(\gb)\,\big) (Z_{N,\gb}^{\go})^{\theta-1}\Big]\bigg],
\end{align*}
where we have interpreted $e^{\sum_{i=1}^N({\gb\go_{i,x_i}-\gl(\gb))}}$ as a probability density
for a new law $\tilde\bbP$
which exponentially tilts $\omega_{i, x_i}$ for each $1\leq i\leq N$. Note that $(\go_{n, x_n}-
\gl'(\gb))$ is increasing in $\omega_{n, x_n}$, while $(Z_{N,\gb}^{\go})^{\theta-1}$ is decreasing in $\omega_{n, x_n}$ because $\theta\in(0,1)$. Therefore by the FKG inequality,
\begin{align*}
\frac{\dd\ }{\dd \gb} \bbE[(Z_{N,\gb}^\go)^\theta] &\leq
 \theta \sum_{n=1}^N \E \Big[ \,\widetilde \bbE\big[\go_{n,x_n}-\gl'(\gb)\big]\, \widetilde \bbE\big[(Z_{N,\gb}^{\go})^{\theta-1}\big]\Bigg]=0 \,,
\end{align*}
since
\begin{align*}
\widetilde \bbE\big[\go_{n,x_n}-\gl'(\gb)\big] =
\bbE\Big[  \big(\go_{n, x_n}-\gl'(\gb)\big) e^{{\gb\go_{n,x_n}-\gl(\gb)}} \Big]
=\frac{\dd\ }{\dd \gb} \bbE\Big[ \,e^{{\gb\go_{n,x_n}-\gl(\gb)}} \Big]=0.
\end{align*}

We have just shown that
$\bbE[(Z_{N,\hbeta'/\sqrt{\Ro_N}}^\go)^\theta]\leq \bbE[(Z_{N,\hbeta/\sqrt{\Ro_N}}^\go)^\theta]$,
 for any $\hat\gb<1\leq \hbeta'$. Since $Z_{N,\hbeta/\sqrt{\Ro_N}}^\go$
converges in distribution to $\bs{Z}_\hbeta$ when $\hbeta<1$, and $(Z_{N,\hbeta/\sqrt{\Ro_N}}^\go)^\theta$ is uniformly integrable,
because $\theta \in (0,1)$ and $\bbE[Z_{N,\hbeta/\sqrt{\Ro_N}}^\go]=1$,
by the first part of Theorem~\ref{thm:subcritical}
we then have
\begin{align*}
\limsup_{N\to\infty} \bbE[(Z_{N,\hbeta'/\sqrt{\Ro_N}}^\go)^\theta] & \leq \limsup_{N\to\infty} \bbE[(Z_{N,\hbeta/\sqrt{\Ro_N}}^\go)^\theta] \\
& =\bbE\bigg[\exp\Big(\theta \int_0^1\frac{\hat\gb\,}{\sqrt{1-\hat\gb^2\,t}} \dd W(t)-
\frac{\theta}{2} \int_0^1\frac{\hat\gb^2}{1-\hat\gb^2\,t} \dd t \Big)\bigg]\\
&= \exp\Big(
\frac{\theta(\theta-1)}{2} \int_0^1\frac{\hat\gb^2}{1-\hat\gb^2\,t} \dd t \Big) =\big(1-\hat\gb^2\big)^{-\frac{\theta(\theta-1)}{2}}.
\end{align*}
Letting $\hbeta\nearrow 1$ then shows that $\bbE[(Z_{N,\hbeta'/\sqrt{\Ro_N}}^\go)^\theta]\to 0$ as $N\to\infty$ whenever $\hbeta'\geq 1$.
\end{proof}

\section{Proof of Theorem \ref{thm:short-cor}}

To prove Theorem \ref{thm:short-cor}, we first need to extend Proposition \ref{thm:main1} to random
variables $\Theta^{N;M}_{\bi}$ which form the building blocks of partition functions
$Z^\omega_{N,\beta}(\sfX)$ with starting points $\sfX=(x,t)$ other than the origin. More precisely, as
in \eqref{genk}, define
\begin{equation}\label{eq:newTheta}
	\Theta^{N;M}_{\bi}(\sfX) :=
	\left(\frac{M}{\Ro_N}\right)^{k/2} \!\!\!\!
	\sumtwo{n_1 - n_0 \in I_{i_1}, \, n_2 - n_1 \in I_{i_2}, \ldots, \, n_k - n_{k-1} \in I_{i_k}}
	{z_1, z_2, \ldots, z_k \in \Z^d}
	\, \prod_{j=1}^k q_{n_j-n_{j-1}}(z_j-z_{j-1})
	\, \prod_{i=1}^k \eta_{(n_i,z_i)},
\end{equation}
except here $(z_0, n_0)$ is defined to be $\sfX$ instead of the origin.

For $\sfX=(x,t)\in \Z^d \times \N_0$ with $d\in \{0,1,2\}$, recall the definition of $\vvvert X\vvvert$ from \eqref{Xtripnorm}. We then
have the following extension of Proposition \ref{thm:main1}.

\begin{proposition}\label{P:keystepG}
Assume that Hypothesis~\ref{hyp} holds
and $\Ro_N$ in \eqref{eq:replica}--\eqref{eq:RN}
diverges as $N\to\infty$.
For $1\leq k\leq r$, let  $\sfX_N^{(k)}=(x_N^{(k)}, t_N^{(k)})$ be points in $\Z^d\times\N_0$,
such that
\begin{equation} \label{eq:hypRN}
	\forall 1\leq k, l\leq r: \qquad
	\Ro_{\vvvert \sfX_N^{(k)} - \sfX_N^{(l)} \vvvert}/\Ro_N = \zeta_{k,l}+o(1)
	\ \ \text{for some} \ \
	\zeta_{k,l} \in [0,1]\,.
\end{equation}
For $M\in\N$, let us denote by $\tilde D_M$ be the set of dominated sequences\
$\bi \in D_M$, cf.\ \eqref{eq:DM2}, for which $i_1/M$ is well separated from all the $\zeta_{k,l}$ in the following sense:
\begin{equation}\label{eq:separate}
	\tilde D_M := \{ \bi \in D_M: \ | i_1 / M - \zeta_{k,l} | > 1/M \ \forall 1 \le k,l \le r \} \,.
\end{equation}
Then the vector $(\Theta^{N;M}_{\bi}(\sfX_N^{(k)}))_{1 \le k \le r, \bi\in\tilde D_M}$ converges in law as $N \to \infty$
to a centered Gaussian vector $(\zeta_{\bi}^{(k)})_{1 \le k \le r, \, \bi \in \tilde D_M}$ with covariance matrix
\begin{equation}\label{eq:cov}
	\bbcov[\zeta_{\bi}^{(k)}, \zeta_{\bi'}^{(l)}] = \ind_{\{\bi = \bi'\}} \ind_{\{i_1/M > \zeta_{k,l}\}} \,.
\end{equation}
\end{proposition}
\begin{proof}
The random variable $\Theta^{N;M}_{\bi}((x,t))$
has the same law as $\Theta^{N;M}_{\bi}((0,0))$. Therefore,
the proof of Proposition \ref{thm:main1} readily implies that
for each $\bi \in \tilde D_M$ and $1\leq k\leq r$,
$$
\bbvar(\Theta^{N;M}_{\bi}(\sfX_N^{(k)})) \asto{N} 1 \quad \mbox{and} \quad  \bbE\big[\big(\Theta^{N;M}_{\bi}(\sfX_N^{(k)})\big)^4\big] \asto{N} 3.
$$
By the (multidimensional) Fourth Moment Theorem \ref{T:4mom}, it then only remains to show that
\begin{equation}\label{covconvg}
\bbcov(\Theta^{N;M}_{\bi}(\sfX_N^{(k)}), \Theta^{N;M}_{\bi'}(\sfX_N^{(l)})) \asto{N} \ind_{\{\bi = \bi'\}} \ind_{\{i_1/M > \zeta_{k,l}\}}
\quad \forall\, 1\leq k,l\leq r, \ \bi, \bi' \in \tilde D_M.
\end{equation}

Note that when the dominated sequences $\bi, \bi'$ are different,
there are unmatched $\eta$'s and consequently
$\bbcov(\Theta_{\bi}^{N;M}(\sfX^{(k)}_N), \Theta_{\bi'}^{N;M}(\sfX^{(l)}_N))=0$; and when $\bi=\bi'$, we have
\begin{align}
\bbE\big[\Theta_{\bi}^{N;M}(\sfX^{(k)}_N) \, & \Theta_{\bi}^{N;M}(\sfX^{(l)}_N) \big]
=  \left(\frac{M}{\Ro_{N}} \right)^{k}
\!\! \sumtwo{n_1 - t_N^{(k)}\in I_{i_1} \\ n_1-t_N^{(l)}\in I_{i_1} } \,
\!\! \sum_{z_1\in \bbZ^d}
 q_{n_1-t_N^{(k)}}(z_1-x_N^{(k)}) q_{n_1-t_N^{(l)}}(z_1-x_N^{(l)}) \nonumber \\
& \qquad\qquad\qquad\qquad
\times \sumtwo{n_2-n_1 \in I_{i_2}, \, \ldots, \, n_k - n_{k-1} \in I_{i_k}}
	{z_2, \ldots, z_k \in \Z^d} \, \prod_{j=2}^k q_{n_j-n_{j-1}}(z_j-z_{j-1})^2 \nonumber\\
& \qquad \sim \frac{M}{\Ro_{N}}
\!\!\sumtwo{n_1 - t_N^{(k)}\in I_{i_1} \\ n_1-t_N^{(l)}\in I_{i_1} } \,
\!\!\sum_{z_1\in \bbZ^d}
 q_{n_1-t_N^{(k)}}(z_1-x_N^{(k)}) q_{n_1-t_N^{(l)}}(z_1-x_N^{(l)}), \label{Thetacov}
\end{align}
where in the last step we used \eqref{eq:replica}
(recall that $q_n(x) = \P(S_n = x)$ and
we write $f(N)\sim g(N)$ as a shorthand for $\lim_{N\to\infty} f(N)/g(N)=1$.

We first consider the case $i_1/M>\zeta_{k,l}$, which implies
$(i_1-1)/M>\zeta_{k,l}$ since $\bi \in \tilde D_M$.
In this case,
since $I_{i_1} \ni n_1-t_N^{(k)} = n_1 - n_0 \ge \sft_{i_1-1} $, recalling assumption \eqref{eq:hypRN}
we have
\begin{equation*}
	n_1-t_N^{(k)} \gg |t_N^{(k)} - t_N^{(l)}| \,.
\end{equation*}
By Hypothesis~\ref{hyp},
the dominant contribution to \eqref{Thetacov} then comes from $z_1$ with
$$
|z_1-x_N^{(k)}| \gg |x_N^{(k)} - x_N^{(l)}| \quad \mbox{as } N\to\infty.
$$
By the same argument as in the proof of Lemma \ref{lem:appr3}, we can apply the local limit theorem in Hypothesis \ref{hypoth} and replace $(x_N^{(l)}, t_N^{(l)})$ in \eqref{Thetacov} by $(x_N^{(k)}, t_N^{(k)})$, which implies that
$$
\lim_{N\to\infty} \bbE\big[\Theta_{\bi}^{N;M}(\sfX^{(k)}_N) \, \Theta_{\bi}^{N;M}(\sfX^{(l)}_N) \big] =
\lim_{N\to\infty} \bbE\big[\Theta_{\bi}^{N;M}(\sfX^{(k)}_N)^2 \big] = 1.
$$

We next consider  the case $i_1/M<\zeta_{k,l}$, which implies $(i_1+1)/M<\zeta_{k,l}$ since $\bi \in \tilde D_M$. By the definitions \eqref{eq:hypRN} and \eqref{Xtripnorm} of
$\zeta_{k,l}$ and $\vvvert \sfX\vvvert$ , this implies
\begin{equation}\label{gapxt}
\mbox{either} \quad \Ro_{|t_N^{(k)} -t_N^{(l)}|}/\Ro_N = \zeta_{k,l}+o(1); \quad \mbox{or} \quad \Ro_{\phi^{\leftarrow}(|x_N^{(k)} -x_N^{(l)}|)}/\Ro_N = \zeta_{k,l}+o(1),
\end{equation}
where we recall by \eqref{eq:CLT} and \eqref{eq:phiarrow}
that $\phi^{\leftarrow}(|x|):= \min\{n\in\N_0: \phi(n) \geq |x| \}$
with $\phi(n) := (n L(n)^2)^{1/d}$.
We now show that \eqref{gapxt} forces either
$n_1$ or $z_1$ to vary in intervals with empty intersection.

In the first case in \eqref{gapxt}, we have $|t_N^{(k)} -t_N^{(l)}| \gg \sft_{i_1+1}\gg |I_{i_1}|$ as $N\to\infty$, where we recall that $I_{i_1}=(\sft_{i_1-1}, \sft_{i_1}]$ with $\Ro_{\sft_{i_1}}\sim \frac{i_1}{M} \Ro_N$. Therefore the constraints $n_1- t_N^{(k)}\in I_{i_1}$ and $n_1- t_N^{(l)}\in  I_{i_1}$ in \eqref{Thetacov} are incompatible and the sum equals zero.

In the second case in \eqref{gapxt}, we have $\phi^{\leftarrow}(|x_N^{(k)} -x_N^{(l)}|)\gg \sft_{i_1+1}$, and hence $|x_N^{(k)} -x_N^{(l)}|\gg \phi(\sft_{i_1+1})$.
Therefore for $N$ large, for any fixed $C > 0$
$$
\Big\{ z_1\in \Z^d: |z_1-x_N^{(k)}|\leq C \phi(\sft_{i_1+1})
\Big\} \,\cap\,
\Big\{z_1\in \Z^d: |z_1-x_N^{(l)}|\leq C \phi(\sft_{i_1+1})\Big\} = \emptyset.
$$
By Hypothesis \ref{hypoth}, uniformly in $n\in I_{i_1}=(\sft_{i_1-1}, \sft_{i_1}]$, the dominant contribution to $\sum_z q_n(z)$ and $\sum_z q_n^2(z)$ come from
the region $|z|\leq C\phi(\sft_{i_1+1})$.
Partitioning the sum in \eqref{Thetacov} according to whether
$|z_1-x_N^{(k)}|\leq C \phi(\sft_{i_1+1})$, or $|z_1-x_N^{(l)}|
\leq C \phi(\sft_{i_1+1})$, or neither,
it then follows
that the quantity in \eqref{Thetacov} tends to $0$ as $N\to\infty$, which concludes the proof of \eqref{covconvg}.
\end{proof}
\medskip

\begin{proof}[Proof of Theorem \ref{thm:short-cor}]
The approximation steps {\bf (A1)--(A3)} for the partition function $Z^\omega_{N, \beta_N}$
outlined in Section \ref{ss:proofsteps} (and proved in Section \ref{sec:CG})
also applies if the starting point of the polymer is different from the origin.
For the step {\bf (A1)}, in order to show that the constraint $n_0 < n_1 < \ldots < n_k \le N$
can be replaced by $1 \le n_1 - n_0 , \ldots, n_k - n_{k-1} \le N$, we need to use the assumption
$\Ro_{N-t_N^{(k)}}/\Ro_N =1-o(1)$ in \eqref{eq:condi}.

It follows that
we can approximate the partition functions
$(Z^\omega_{N, \beta_N}(\sfX_N^{(j)}))_{1\leq j\leq r}$ jointly in $L^2$
(with an error uniformly small in $N$, when $M$ is large, cf.\ \eqref{ZA3K}) by

\begin{equation}\label{sfX-appr}
Z_{N, \beta_N}^{{\bf (A3)}}(\sfX_N^{(j)}):= \! 1+ \!\sum_{k=1}^M  \frac{\hbeta^k}{M^{\frac{k}{2}}} \!\!\sum_{\bi\in \{1,\ldots, M\}^k_\sharp}  \!\!\!\!\!\!
	\Theta^{N;M}_{\bi^{(1)}}(\sfX_N^{(j)}) \Theta^{N;M}_{\bi^{(2)}}(\sfX_N^{(j)}) \cdots \Theta^{N;M}_{\bi^{(\mathfrak{m})}} (\sfX_N^{(j)}),
	\quad 1\leq j\leq r,
\end{equation}
where we recall that $\Theta^{N;M}_{\bi}(\sfX_N^{(j)})$ is defined
in \eqref{eq:newTheta}.

By Proposition \ref{P:keystepG}, as $N\to\infty$, $(Z_{N, \beta_N}^{{\bf (A3)}}(\sfX_N^{(j)}))_{1\leq j\leq r}$ converge jointly in distribution to 
\begin{align*}
\bs{Z}_\hbeta^{(M,j)} := 1+ \!\sum_{k=1}^M  \frac{\hbeta^k}{M^{\frac{k}{2}}} \sum_{\bi\in \{1,\ldots, M\}^k_\sharp}  \!\!\!\!\!\!
             \,\,\,\,\zeta_{\bi^{(1)}}^{(j)} \zeta_{\bi^{(2)}}^{(j)} \cdots \zeta_{\bi^{(\mathfrak{m})}}^{(j)}, \qquad 1\leq j\leq r.
\end{align*}
It only remains to prove the analogue of Lemma \ref{prop:limCG} and show that as $M\to\infty$, $(\bs{Z}_\hbeta^{(M,j)})_{1\leq j\leq r}$ converge jointly to the family of log-normal random variables $(\colon e^{\sfY_j}\colon)_{1 \le j\le r}$ in \eqref{fddconv}-\eqref{meanvar}.

Following the same steps as in the proof of Lemma \ref{prop:limCG} up to the resummation procedure in \eqref{bigXi}, we can approximate $\bs{Z}_\hbeta^{(M,j)}$ in $L^2$ (as $M\to\infty$) by
  \begin{align*}
 \bs{\widehat Z}_\hbeta^{(M,j)}  = 1 + \sum_{m=1}^\infty
	\sumtwo{0 < t_1 < t_2 < \ldots < t_m \le 1}{t_1, \ldots, t_m \in \frac{1}{M}\N}
 	\ \prod_{i=1}^m\ \Xi_{M, t_i}^{(j)},
\end{align*}
where
\begin{align*}
\Xi_{M,t}^{(j)} := \sum_{r\in\bbN}\frac{\hbeta^{r}}{M^{\frac{r}{2}}}\ \xi_{r}^{(j)}(M t) \quad \mbox{for }  t\in [0,1]\cap \frac{1}{M}\N,
\quad \mbox{and} \quad \xi_r^{(j)}(a) := \!\!\!\!\!\!\!\!\!\!\!\!\!\! \sum_{(a_2, \ldots, a_r) \in \{1, \ldots, a-1\}^{r-1}} \!\!\!\!\!\!\!\!\!
	\zeta_{(a,a_2, \ldots, a_r)}^{(j)}.
\end{align*}
Similar to \eqref{Xiencode}, we can encode the family of jointly Gaussian random variables $\Xi_{M,t}^{(j)}$ as
\begin{align}\label{Xiencode2}
\Xi_{M,t}^{(j)} = \frac{\hbeta (1+o(1))}{\sqrt{1-\hbeta^2 \, t}}
 \int_{t-\frac{1}{M}}^{t} \dd W_s^{(j)}, \qquad t\in [0,1]\cap \frac{1}{M}\N,
\end{align}
where $(W^{(j)})_{1\leq j \leq r}$ is a family of correlated Brownian motions
(the explicit form of the correlations will be derived in a moment).
Therefore for all $1\leq j\leq r$,
\begin{align*}\label{hatZMchaosj}
\bs{\widehat Z}_\hbeta^{(M, j)} =
1 + \sum_{m=1}^\infty
	 \sumtwo{0 < t_1 < t_2 < \ldots < t_m \le 1}{t_1, \ldots, t_m \in \frac{1}{M}\N}  \,\,\prod_{i=1}^m
	 \frac{\hbeta(1+o(1))}{\sqrt{1-\hbeta^2\, t_i}} \int_{t_{i}-\frac{1}{M}}^{t_i} \dd W^{(j)}_s\ \xrightarrow[\,M\to\infty\,]{L^2}
\  \ : e^{\sfY_j} : \ ,
\end{align*}
where $\sfY_j:=\int_0^1  \frac{\hbeta}{\sqrt{1-\hbeta^2\, t}} \dd W^{(j)}_t$. It now only remains to find the covariance between $(\sfY_j)_{1\leq j\leq r}$.

Note that for each $1\leq k, l\leq r$ and $s, t\in [0,1]\cap \frac{1}{M}\N$, by the definition of $\xi^{(j)}_r$ and Proposition \ref{P:keystepG}, we have
\begin{align*}
\bbE\big[\Xi_{M,t}^{(k)} \Xi_{M,s}^{(l)} \big] & =
\sum_{r\in\bbN}\frac{\hbeta^{2r}}{M^{r}}\ \bbE\big[  \xi_{r}^{(k)}(M t) \, \xi_{r}^{(l)}(M s) \big] \\
& = \ind_{\{t=s\}} \sum_{r\in\bbN}\frac{\hbeta^{2r}}{M^{r}} \sum_{(a_2, \ldots, a_r) \in \{1, \ldots, Mt-1\}^{r-1}} \bbE\big[
	\zeta_{(Mt,a_2, \ldots, a_r)}^{(k)} \zeta_{(Mt,a_2, \ldots, a_r)}^{(l)}\big] \\
& = \ind_{\{t=s\}} \sum_{r\in\bbN}\frac{\hbeta^{2r}}{M^{r}} \sum_{(a_2, \ldots, a_r) \in \{1, \ldots, Mt-1\}^{r-1}} \!\!\!\!\!\!\!\!\!\!\!
 \ind_{\{t>\zeta_{k,l}\}} = \frac{\ind_{\{t=s>\zeta_{k,l}\}}}{M}\cdot \frac{\hbeta^2(1+o(1))}{1-\hbeta^2 t}.
\end{align*}
Therefore for all $s, t\in [0,1]\cap \frac{1}{M}\N$, we have
$$
\bbE\Big[\int_{t-\frac{1}{M}}^t \dd W_u^{(k)} \int_{s-\frac{1}{M}}^s\dd W_u^{(l)}\Big]= \frac{\ind_{\{t=s>\zeta_{k,l}\}}}{M},
$$
and hence $\bbE[W^{(k)}(T) W^{(l)}(S)] = \int_{\zeta_{k,l}}^{T\wedge S} \dd t$ for all $0\leq S\leq T\leq 1$. This implies that
$$
\bbcov(Y_k, Y_l) = \bbE\Big[ \int_0^1  \frac{\hbeta\ \dd W^{(k)}_t}{(1-\hbeta^2\, t)^{\frac{1}{2}}} \int_0^1  \frac{\hbeta\ \dd W^{(l)}_t }{(1-\hbeta^2\, t)^{\frac{1}{2}}} \Big] = \int_{\zeta_{k,l}}^1 \frac{\hbeta^2}{1-\hbeta^2t} \dd t =  \log \frac{1-\hat\beta^2 \zeta_{k,l}}{1-\hat\beta^2},
$$
which concludes the proof.
\end{proof}

\section{Proof of Theorem~\ref{thm:field}}

In this section, we prove Theorem \ref{thm:field}. First we prove an analogue of Proposition \ref{thm:main1}, the key step {\bf (K)} in the proof of Theorem \ref{thm:subcritical}. The difference here is that we need to average over the starting point of the partition function. As in Theorem \ref{thm:field}, let $\psi: \R^d \times [0,1] \to \R$ be a continuous function with compact support. For any finite strictly increasing sequence $\bn=(n_1,...,n_{|\bn|})$ and
$0\le n_0 <n_1$, we then modify the definition of $Q_{\bn}$ in \eqref{eq:bq0} as follows:
\begin{equation}\label{Qpsi00}
Q_{(n_0, \bn)}^{\psi}
    := \sum_{x_0\in\Z^d} \sum_{\bx \in (\Z^d)^{|\bn|}} \,
	\Big(\prod_{j=1}^{|\bn|} q_{n_j - n_{j-1}} (x_j - x_{j-1}) \, \eta_{(n_j,x_j)} \Big)   \psi(\widehat \sfX_N),
\end{equation}
where
\begin{equation*}
	\widehat \sfX_N :=
	\bigg( \frac{x_0}{\phi(N)} \,, \frac{n_0}{N} \bigg) \qquad \mbox{with} \quad \phi(N):= (L(N)^2 N)^{1/d}.
\end{equation*}
To decompose $J^\psi_N$ in \eqref{eq:J0} as we decomposed the partition function in terms of the $\Theta$'s, for $\bi \in \{1,\ldots,M\}^k$, we need to modify the definition of $\Theta^{N,M}_{\bi}$ in \eqref{thetai} as follows:
\begin{align}\label{psi_theta}
\Theta_{\bi}^{N,M;\psi}  :=
\frac{L(N)}{\phi(N)^d N}
\left(\frac{M}{\Ro_{N} }\right)^{\frac{|\bi|-1}{2}}
\sum_{0 \leq n_0 < N} \, \sum_{\bn \prec \bi} Q_{(n_0,\bn)}^\psi.
\end{align}

The following analogue of Proposition~\ref{thm:main1}
is the key step in the proof of Theorem~\ref{thm:field}.

\begin{proposition}\label{KeyStep2}
Assume that Hypothesis \ref{hypoth} holds, and $\Ro_N$ in \eqref{eq:replica}
is a slowly varying function which diverges as $N\to\infty$.
For each $M\in\N$ and $\psi\in C_c(\R^d \times [0,1])$, the random variables
$(\Theta^{N,M,\psi}_{\bi})_{\bi \in D_M}$ converge in joint distribution to a family of independent
Gaussian random variables $(\zeta^\psi_{\bi})_{\bi\in D_M}$ with $\zeta^\psi_{\bi}=0$ if $i_1<M$;
and if $i_1=M$, then $\zeta^\psi_{\bi}$ has mean zero and variance
\begin{equation}
V^\psi:=
\int_{(\R^d\times [0,1])^2} \psi(x,t) K((x,t), (x',t')) \psi(x', t') \dd x \dd t \dd x' \dd t',
\end{equation}
where $K$ is defined in \eqref{eq:kappakernel}.
\end{proposition}

\begin{proof}
The proof is similar to that of Proposition \ref{thm:main1}. We will only highlight the changes in the proof. For simplicity, we assume $d\neq 0$. The case $d=0$ can be treated similarly.

First note that we can rewrite $\Theta^{N,M;\psi}_{\bi}$ in the following form:
\begin{align}\label{ThetaPsiRW}
\Theta^{N,M;\psi}_{\bi} = \frac{\big(\frac{M}{\Ro_{N} }\big)^{\frac{|\bi|-1}{2}}}{\phi(N)^{\frac{d}{2}} N^{\frac{3}{2}}} \sum_{\bx \in (\Z^d)^{|\bi|}}
\sumtwo{1\leq n_1\leq N}{n_2-n_1\in I_{i_2}, \ldots, n_{|\bi|}-n_{|\bi|-1}\in I_{|\bi|}} q^\psi_{\bn}(\bx)  \eta_{(\bn, \bx)},
\end{align}
where $\eta_{(\bn, \bx)}=\prod_{j=1}^{|\bn|}\eta_{(n_j, x_j)}$, and $q^\psi_{\bn}(\bx) := q^{N,\psi}_{i_1, n_1}(x_1) \prod_{j=2}^{|\bn|}q_{n_j-n_{j-1}}(x_j-x_{j-1})$ with
\begin{equation}\label{qpsi1}
q^{N,\psi}_{i_1, n_1}(x_1) := \sumtwo{x_0\in\Z^d}{n_0\in [0, n_1)\cap (n_1-I_{i_1})} \psi\Big(\frac{x_0}{\phi(N)}, \frac{n_0}{N}\Big) q_{n_1-n_0}(x_1-x_0).
\end{equation}
Note that the constraint $n_1 - n_0 \in I_{i_1}$ appearing in \eqref{psi_theta}
(inside $\bn \prec \bi$)
has been moved to \eqref{qpsi1}.
For simplicity, we will denote the summation constraints on $\bn$ in \eqref{ThetaPsiRW}
also by $\bn \prec \bi$. Note that we have just casted $\Theta^{N,M;\psi}_{\bi}$ in
the same form as $\Theta^{N,M}_{\bi}$ in \eqref{thetai}.
\medskip

\noindent
{\bf Variance Calculations.}
We first show that when $i_1<M$, $\Theta^{N,M;\psi}_{\bi}\to 0$ because its variance tends to $0$. Note that
\begin{align}\label{varpsi}
\bbvar\big( \Theta^{N;M;\psi}_{\bi} \big)
= \frac{\big(\frac{M}{\Ro_{N}} \big)^{|\bi|-1} }{\phi(N)^dN^3}
 \sumtwo{\bn\prec \bi}{\bx \in (\Z^d)^{|\bi|}} (q^{N,\psi}_{i_1, n_1}(x_1))^2 \prod_{j=2}^{|\bi|} q_{n_j-n_{j-1}}(x_j-x_{j-1})^2.
\end{align}
Note that given $(n_1, x_1)$, the sum over $(n_2,x_2),...,(n_{|\bi|},x_{|\bi|})$ asymptotically equals $\big(\frac{\Ro_{N}}{M}\big)^{{|\bi|}-1}$ by the same calculations as in \eqref{eq:varcom}. Therefore
\begin{align}
&\bbvar\big( \Theta^{N;M;\psi}_{\bi} \big)
 = \frac{1+o(1)}{\phi(N)^dN^3} \sumtwo{x_1 \in \Z^d}{1\leq n_1\leq N} (q^{N,\psi}_{i_1, n_1}(x_1))^2 \notag \\
= \ \ &  \frac{1+o(1)}{\phi(N)^dN^3} \sumtwo{x_1 \in \Z^d}{1\leq n_1\leq N}
\sumtwo{0\leq n_0,n_0'\in n_1-I_{i_1}}{x_0, x_0'\in\Z^d}  \psi\big(\hat {\sfX}_N \big) \psi\big(\hat{\sfX}'_N \big) q_{n_1-n_0}(x_1-x_0) q_{n_1-n_0'}(x_1-x_0'). \label{varpsi2}
\end{align}
Since $\psi\in C_c(\R^d\times [0,1])$, we can choose $A>0$ large enough such that ${\rm supp}(\psi)\subset [-A,A]^d\times [0,1]$. Then in \eqref{varpsi2}, we can restrict the sums to $|x_0|_\infty, |x_0'|_\infty \leq A\phi(N)$. By the local limit theorem for $q_n(\cdot)$ in Hypothesis \ref{hypoth}, we observe that the dominant contribution in \eqref{varpsi2} comes from $x_1\in\Z^d$ with $|x_1|_\infty < \widetilde A\phi(N)$ if $\widetilde A$ is large enough. By first summing over $(x_0,n_0)$ and $(x_0', n_0')$, we then have
\begin{align}
\bbvar\big( \Theta^{N;M;\psi}_{\bi} \big)
&\leq \frac{C}{\phi(N)^dN^3}
 \sumtwo{|x_1|_\infty < \tilde A \phi(N)}{1\leq n_1\leq N}
|\psi|_\infty^2 \sumtwo{x_0, x_0'\in\Z^d}{1\leq n_0, n_0'\in n_1-I_{i_1}}  q_{n_1-n_0}(x_1-x_0) q_{n_1-n_0'}(x_1-x_0') \notag\\
&= \frac{C}{\phi(N)^dN^3}
 \sumtwo{|x_1|_\infty < \widetilde A\phi(N)}{1\leq n_1\leq N}
 \ \sum_{1\leq n_0, n_0'\in n_1-I_{i_1}} |\psi|_\infty^2 \notag \\
 &\leq \frac{C|\psi|_\infty^2}{\phi(N)^dN^3} \cdot N (2\widetilde A)^d \phi(N)^d |I_{i_1}|^2 \asto{N}{0}, \label{ThePsiVar}
\end{align}
where the convergence holds because $\Ro_{\sft_{i_1}}/\Ro_N\sim i_1/M<1$ implies that $|I_{i_1}|\leq \sft_{i_1}=o(N)$.

We next check that when $i_1=M$, $ \Theta^{N;M;\psi}_{\bi}$ has the correct limiting variance. Note that because $I_{i_1}=(\sft_{M-1}, \sft_M]=(\sft_{M-1}, N]$ with $\sft_{M-1}\ll N$, by the same bound as in \eqref{ThePsiVar}, we can enlarge the range of summation
of $n_0, n'_0$ in \eqref{varpsi2} to $0 \le n_0, n'_0 < n_1$
without changing the limiting variance. Moreover, by the local limit theorem for $q_n$ in Hypothesis \ref{hypoth}, we have
\begin{align}\label{varpsi3}
\bbvar\big( \Theta^{N;M;\psi}_{\bi} \big)
&= \frac{1+o(1)}{\phi(N)^dN^3}
\sumtwo{0\leq n_0,n_0'<n_1\leq N}{x_1, x_0, x_0'\in\Z^d} \!\!\!\!\psi\big(\hat {\sfX}_N \big)
 \psi\big(\hat{\sfX}'_N \big) q_{n_1-n_0}(x_1-x_0) q_{n_1-n_0'}(x_1-x_0') \notag\\
& = \frac{1+o(1)}{\phi(N)^dN^3}
\sumtwo{0\leq n_0,n_0'<n_1\leq N}{x_1, x_0, x_0'\in\Z^d} \!\!\!\!\psi\big(\hat {\sfX}_N \big)
 \psi\big(\hat{\sfX}'_N \big) \frac{g(\frac{x_1-x_0}{\phi(n_1-n_0)})}{\phi(n_1-n_0)^d} \frac{g(\frac{x_1-x_0'}{\phi(n_1-n_0')})}{\phi(n_1-n_0')^d} \\
&= \frac{1+o(1)}{\phi(N)^{3d} N^3}
\sumtwo{0\leq n_0, n_0'<n_1\leq N}{x_1, x_0, x_0'\in\Z^d} \!\!\!\!\!\!\!\!\!\psi\big(\hat {\sfX}_N \big)
 \psi\big(\hat{\sfX}'_N \big) \frac{g\big(\frac{x_1-x_0}{\phi(N)}/(\frac{n_1-n_0}{N})^{\frac{1}{d}}\big)}{\frac{n_1-n_0}{N}} \frac{g\big(\frac{x_1-x_0'}{\phi(N)}/(\frac{n_1-n_0'}{N})^{\frac{1}{d}}\big)}{\frac{n_1-n_0'}{N}},\notag
\end{align}
where in the last equality
we have replaced $\phi(n_1-n_0)$ and $\phi(n_1-n_0')$ respectively by $\phi(N)(\frac{n_1-n_0}{N})^{1/d}$ and $\phi(N)(\frac{n_1-n_0'}{N})^{1/d}$. This is justified when $n_1-n_0, n_1-n_0'>\eps N$ for any fixed $\eps>0$, because $\phi(n)=(L(n)^2n)^{1/d}$ and $L(\cdot)$ is slowly varying; while on the other hand, the contributions to \eqref{varpsi3} from $n_0,n_0', n_1$ with $n_1-n_0\leq \eps N$ or $n_1-n_0' \leq  \eps N$ can be made arbitrarily small by choosing $\eps$ small, thanks to the same estimates as in \eqref{ThePsiVar}. A Riemann sum approximation with $y:=x_0/\phi(N)$, $y':=x_0'/\phi(N)$, $z:=x_1/\phi(N)$, $s=n_0/N$, $s':=n_0'/N$ and $t:=n_1/N$ then gives
\begin{align}\label{limtoVpsi}
\bbvar\big( \Theta^{N;M;\psi}_{\bi} \big)  & \asto{N} \idotsint\limits_{y, y', z\in\R^d \atop 0\leq s, s' <t\leq1} \psi(y, s) \psi(y', s') \frac{g\left( \frac{z -  y}{(t- s)^{1/d}} \right)}{t-s}
	\frac{g\left( \frac{z - y'}{(t-s')^{1/d}} \right)}{t-s'} \dd z \dd t \dd y \dd y'  \dd s \dd s' \notag\\
& \quad = \idotsint\limits_{y, y', z\in\R^d \atop 0\leq s, s' \leq1} \psi(y, s) K((y,s), (y', s')) \psi(y', s') \dd y \dd y' \dd s \dd s' =V^\psi,
\end{align}
where
\begin{align}
 K((y,s), (y', s'))  & = \int_{s \vee s'}^1\int_{\R^d}\frac{g\left( \frac{z - y}{(t-s)^{1/d}} \right)}{t-s}
	\frac{g\left( \frac{z-y'}{(t-s')^{1/d}} \right)}{t-s'} \dd z \dd t \notag \\
& = \int_{s \vee s'}^1\int_{\R^d} g_{t-s}(z-y) g_{t-s'}(z-y') \dd z \dd t = \int_{s \vee s'}^1 g_{2t-s-s'} (y-y') \dd t \notag \\
& = \int_{s \vee s'}^1 \frac{g\big(\frac{y-y'}{(2t-s-s')^{1/d}}\big)}{2t-s-s'} \dd t = \frac{1}{2}\int_{|s-s'|}^{2-s-s'} \frac{g\big(\frac{y-y'}{u^{1/d}}\big)}{u} \dd u. \label{Kagain}
\end{align}
Here we used the fact that the transition density $g_t$ of a Brownian motion in $\R^2$ (or a Cauchy process in $\R$) is symmetric and scaling invariant, and $g=g_1$. Note that $K$ agrees with the kernel in \eqref{eq:kappakernel}, which completes the variance verification.
\medskip

\noindent
{\bf Fourth moment calculations.} We now apply the fourth moment theorem Theorem~\ref{T:4mom} to prove Proposition \ref{KeyStep2}.
By the variance calculations above, it suffices to restrict our attention to $(\Theta^{N,M; \psi}_{\bi})_{\bi \in D_M}$ with $i_1=M$. For distinct
$\bi, \bi'\in D_M$ with $i_1=i'_1=M$, it is easily seen that $\E[\Theta^{N,M; \psi}_{\bi} \Theta^{N,M; \psi}_{\bi'}]=0$. Therefore condition (i) in Theorem \ref{T:4mom} is satisfied. Clearly condition (iii) also holds. It only remains to verify the fourth moment condition:
\begin{align}\label{eq:tosho2}
\lim_{N\to\infty} \bbE \big[ (\Theta^{N,M;\psi}_{\bi})^4 \big] = 3 V^\psi
	\qquad \forall\, \bi \in D_M  \mbox{ with } i_1=M,
\end{align}
assuming that $(\eta_{(n,x)})_{n\in\N, x\in\Z^d}$ are i.i.d.\ standard normal.

Using the representation for $\Theta^{N,M;\psi}_{\bi}$ in \eqref{ThetaPsiRW}, we have a similar expansion of the fourth-moment as in \eqref{sumTheta} and \eqref{prodQ}:
\begin{align}
\bbE \big[ (\Theta^{N,M;\psi}_{\bi})^{4} \big]  = \frac{\big(\frac{M}{\Ro_N}\big)^{2|\bi|-2} }{\phi(N)^{2d}N^6}
\!\!\!\!\!\!\!\!\! \sumtwo{\ba, \bb, \bc, \bd \prec \bi}{\bx, \by, \bz, \bw\in (\Z^d)^{|\bi|}}
  \!\!\!\!\!\!\!\!\!\!\!\!  q^\psi_{\ba}(\bx) q^\psi_{\bb}(\by)  q^\psi_{\bc}(\bz) q^\psi_{\bd}(\bw)
\, \bbE \big[ \eta_{(\ba,\bx)}
\eta_{(\bb,\by)}  \eta_{(\bc,\bz)}  \eta_{(\bd,\bw)} \big], \label{4mom_psi}
\end{align}
where $q^\psi_{\bn}(\bx) := q^{N,\psi}_{M, n_1}(x_1) \prod_{j=2}^{|\bn|}q_{n_j-n_{j-1}}(x_j-x_{j-1})$ with $q^{N,\psi}_{M, n_1}(x_1)$ defined as in \eqref{qpsi1}.

Note that the only difference between the expansion in \eqref{4mom_psi} and the expansion for $\bbE \big[ (\Theta^{N,M}_{\bi})^{4} \big]$ in \eqref{sumTheta}--\eqref{prodQ} is that the factors $q_{n_1}(x_1)$ are replaced by $q^{N,\psi}_{M, n_1}(x_1)$, and the corresponding normalizing constant $\sum_{x_1\in\Z^d, n_1\in I_M} q_{n_1}(x_1)^2 \sim \Ro_N/M$ is replaced by
\begin{align}\label{eq:usefu}
\frac{1}{V^\psi} \sum_{x_1\in\Z^d \atop 1\leq n_1\leq N} (q^{N,\psi}_{M, n_1}(x_1))^2 \sim \phi(N)^d N^3,
\end{align}
where this asymptotic relation follows from the variance calculations in \eqref{varpsi2} and \eqref{limtoVpsi}.

The verification of the fourth moment condition \eqref{eq:tosho2} now follows the same argument as that for $\Theta^{N,M}_{\bi}$ in Section \ref{main_estimate}. Recall from \eqref{eq:sfp} that
\begin{align*}
	\sfp := \big| (\ba,\bx) \cup (\bb, \by) \cup (\bc, \bz) \cup (\bd, \bw) \big|
	= \sum_{(n,r) \in \N \times \Z^d} \ind_{\{(n,r) \in
	(\ba,\bx) \cup (\bb, \by) \cup (\bc, \bz) \cup (\bd, \bw)\}},
\end{align*}
and we relabel $(\ba,\bx) \cup (\bb, \by) \cup (\bc, \bz) \cup (\bd, \bw) = \{ (f_1, h_1), (f_2, h_2), \ldots, (f_p, h_p) \}$, with $f_1 \le f_2 \le \cdots  \le f_p$.

In the proof of Proposition \ref{thm:main1}, we considered 5 cases. {\bf Case 1} with $\sfp>2|\bi|$ can be treated exactly the same way here.

For {\bf Case 2} with $\sfp<2|\bi|$, we can follow the same arguments up to \eqref{Sn2-Cauchy}
(note that $0 \le q_n(x) \le 1$ under our assumptions).
If there are only two factors of $q$ and $q^{N, \psi}$ in the l.h.s.\ of \eqref{Sn2-Cauchy} that involve
$(f_p, h_p)$, then we apply Cauchy-Schwartz exactly as in \eqref{Sn2-Cauchy}, which gives the
desired factors of $(\Ro_N/M)^{1/2}$ or $(\phi(N)^d N^3)^{1/2}$. If there are four factors of $q$
and $q^{N, \psi}$, then we can pick any two factors and bound the factor of $q$ by 1, and bound
the factor of $q^{N, \psi}$ by $N \|\psi \|_\infty$,  since
$$
q^{N,\psi}_{M, n_1}(x_1) := \sumtwo{x_0\in\Z^d}{n_0\in [0, n_1)\cap (n_1-I_M)}
\psi\Big(\frac{x_0}{\phi(N)}, \frac{n_0}{N}\Big) q_{n_1-n_0}(x_1-x_0) \leq N \|\psi\|_\infty .
$$
Note that the pre-factor in \eqref{4mom_psi} will be cancelled out exactly when each $q$
contributes a factor of $(\Ro_N/M)^{1/2}$
to the sum, and each $q^{N, \psi}$ contributes a factor of $(\phi(N)^d N^3)^{1/2}$. Each replacement
of $q$ by $1$ in \eqref{Sn2-Cauchy}
leads to the loss of a factor $(\Ro_N/M)^{1/2}$ in the bound for $ \mathcal{S}_N^{(p)}$
in \eqref{Sn2est}, and similarly, each replacement of $q^{N,\psi}$ by $N \|\psi\|_\infty$ leads
to the loss of a factor $(\phi(N)^dN^3)^{1/2}/N \|\psi\|_\infty$. Summing successively over
$(f_{p-1},h_{p-1})$, $(f_{p-2},h_{p-2})$, \ldots, $(f_1,h_1)$ then gives a similar bound as in
\eqref{Sn2est},  so that the contributions in this case is negligible.

For {\bf Case 3} where $[\ba] \cup [\bb] \cup [\bc] \cup [\bd]$ consists of three or four connected components, it again reduces to {\bf Case 1}.

For {\bf Case 4} where $[\ba] \cup [\bb] \cup [\bc] \cup [\bd]$ consists of a single connected component, we follow the same calculations
up to \eqref{Sn2-Cauchybis}, where we note that because $f_u=c_1$ is the first index of the sequence $\bc$, the first factor in the r.h.s.\ of \eqref{Sn2-Cauchybis} should be replaced by
$$
\Big( \sum_{x\in \bbZ^d, f_1 < n \leq f_1 + \bar m_1} (q^{N, \psi}_n(x))^2 \Big)^{1/2} \leq  \big(C \bar m_1 (2\widetilde A)^d \phi(N)^d |I_M|^2\big)^{1/2}
= o\big((\phi(N)^dN^3)^{1/2}\big),
$$
where the inequality follows the same calculations as in \eqref{ThePsiVar}, and the last equality holds
since $|I_M| \le N$ and
$\bar m_1= |\bi| \sft_{M-2}=o(N)$, by its definition in \eqref{eq:constra}.
This implies a similar bound as in \eqref{case3b} and shows that this case is also negligible.

{\bf Case 5} where $[\ba] \cup [\bb] \cup [\bc] \cup [\bd]$ consists of two connected components gives the full contribution to the limiting fourth moment of $\Theta^{N, M;\psi}_{\bi}$ in \eqref{eq:tosho2}, and the argument is exactly the same as in the proof of Proposition \ref{thm:main1}.
\end{proof}
\smallskip

\begin{proof}[Proof of Theorem \ref{thm:field}]
Recall from \eqref{eq:J0} that
\begin{equation*}
J^\psi_N:=\frac{1}{\phi(N)^d N }\sum_{\sfX\in \bbZ^d\times \bbN_0}
\sqrt{\Ro_N L(N)^2} \,\Big(\,Z_{N,\gb_N}^{\go}(\sfX)-1\Big) \psi (\hat{\sfX}_N ).
\end{equation*}
To prove Theorem \ref{thm:field}, i.e., show that $J^\psi_N$ converges in distribution to a Gaussian random variable with mean zero and variance given in \eqref{eq:sigmapsi}--\eqref{eq:kappakernel}, we plug in the polynomial chaos expansion of $Z_{N,\gb_N}^{\go}(\sfX)$
from \eqref{eq:Zpoly0} (with $(x_0, n_0):=\sfX$) and rewrite $J^\psi_N$ as
\begin{align}\label{J000}
J^\psi_N & = \frac{L(N)}{\phi(N)^dN } \sum_{k=1}^\infty \frac{\hbeta^k}{\Ro_N^{(k-1)/2}}
\sumtwo{1\leq n_1<\cdots<n_k\leq N}{x_1, x_2, \ldots, x_k \in \bbZ^d} \!\!\!\!\!\!\!\!\!\! q^{N,\psi}_{n_1}(x_1)
	\, \prod_{j=2}^k q_{n_j-n_{j-1}}(x_j-x_{j-1})
	\, \prod_{i=1}^k \eta_{(n_i,x_i)},
\end{align}
where
\begin{equation}
q^{N,\psi}_{n_1}(x_1) := \sum_{x_0\in\Z^d, 0
\le n_0<n_1} \psi\Big(\frac{x_0}{\phi(N)}, \frac{n_0}{N}\Big) q_{n_1-n_0}(x_1-x_0).
\end{equation}

The approximation steps {\bf (A1)}--{\bf (A3)},
described for the partition function $Z_{N,\gb_N}^\go$ in Section~\ref{ss:proofsteps}
(and proved in Section~\ref{sec:CG}),
can also be performed for $J^\psi_N$ with only minor differences. Similar to \eqref{ZA3K}, we can therefore approximate $J^\psi_N$ by
\begin{align}\label{hatJ00}
\hat{J}^{\psi}_{N,M}&:=
\sum_{k=1}^M \frac{\hat\gb^k }{M^{(k-1)/2}}\,\sum_{\bi\in\{1,...,M\}_\sharp^k}
\Theta^{N;M;\psi}_{\bi^{(1)}} \,
 \Theta^{N;M}_{\bi^{(2)}}\cdots \Theta^{N;M}_{\bi^{(\mathfrak{m})}},
\end{align}
where $\bi^{(1)},...,\bi^{(\mathfrak{m})}$ is the decomposition of $\bi$ into dominated sequences and $\Theta^{N;M;\psi}_{\bi}$ was defined in \eqref{psi_theta}.  Taking the limit $N\to\infty$ in \eqref{hatJ00} and applying Proposition~\ref{KeyStep2} as done in \eqref{eq:cutoff}, we obtain 
\begin{align}
\hat{J}^{\psi}_{N,M}
\xrightarrow[N\to\infty]{d}
	\bs{J}_\hbeta^{(M)}
	:= \!\!\sum_{k=1}^\infty \frac{\hbeta^k }{M^{\frac{k-1}{2}}} \!\!\sum_{\bi \in \{1, \ldots, M\}^k_\sharp}
	 \prod_{l=1}^{\mathfrak{m}(\bi)} \zeta^\psi_{\bi^{(l)}} = \sum_{k=1}^\infty \frac{\hbeta^k }{M^{\frac{k-1}{2}}} \!\!\sum_{\bi \in \{1, \ldots, M\}^k_\sharp \atop i_1=M} \zeta^\psi_{\bi},
\end{align}
where we used the fact that
by Proposition \ref{KeyStep2} only $\zeta^\psi_{\bi}$ with $i_1=M$ are non-zero. Again, by
Proposition \ref{KeyStep2}, we note that $\bs{J}_\hbeta^{(M)}$ is a normal random
variable (as a sum of independent normal variables) with mean zero and variance
$$
\bbvar(\bs{J}_\hbeta^{(M)}) = \sum_{k=1}^\infty \frac{\hbeta^{2k} }{M^{k-1}} |\{\bi \in \{1, \ldots, M\}^k_\sharp : i_1=M \}| V^\psi \asto{M}
\frac{\hbeta^2}{1-\hbeta^2} V^\psi,
$$
which is exactly the variance $\sigma_\psi^2$ in \eqref{eq:sigmapsi}. Therefore $\bs{J}_\hbeta^{(M)}$ converges to a normal random variable with mean zero and variance $\sigma_\psi^2$. This is the analogue of step {\bf (A4)} in Section \ref{ss:proofsteps}, which concludes the proof of Theorem \ref{thm:field}.
\end{proof}

\section{Proof for the $2d$ Stochastic Heat Equation}\label{S:SHE}

In this section, we prove Theorems~\ref{T:SHE} and \ref{T:SHEfield} for the regularized 2d Stochastic Heat Equation \eqref{2dSHEeps}. The basic strategy is to compare the solution $u^\eps$ with the partition function of a directed polymer on $\Z^{2+1}$, so that we can apply Theorems \ref{thm:short-cor} and
\ref{thm:field}.

The starting point is the Feynman-Kac representation \eqref{Feynman1.1} for $\tilde u^\eps(t,x)$, which has the same distribution as $u^\eps(t,x)$, but differs by a time reversal in the Feynman-Kac formula \eqref{Feynman}. We can extend this representation jointly to all $(\tilde u^\eps(t,x))_{t\in [0,1], x\in \R^2}$. Namely, let
\begin{align}
\tilde u^\eps(t,x) & = \E_{(\eps^{-2}(1-t), \eps^{-1}x)}\!\Big[\! \exp\Big\{\beta_\eps  \! \int_{\eps^{-2}(1-t)}^{\eps^{-2}} \int_{\R^2} \! j(B_{s}- y) {\tilde W}(\dd s, \dd y)  - \frac{1}{2}\beta_\eps^2 \eps^{-2}t \Vert j\Vert_2^2\Big\}\!\Big] \notag\\
& = \E_{(\eps^{-2}(1-t), \eps^{-1}x)}\!\Big[:\exp\Big\{\beta_\eps   \int_{\eps^{-2}(1-t)}^{\eps^{-2}} \int_{\R^2} j(B_{s}- y) {\tilde W}(\dd s, \dd y)  \Big\}:\Big], \label{FKrec}
\end{align}
where $\E_{(s, y)}$ denotes expectation for a Brownian motion $B$ starting at $y$ at time $s$. It is then clear that $(\tilde u^\eps(t,x))_{t\in [0,1], x\in \R^2}$ has the same distribution as $(u^\eps(t, x))_{t\in [0,1], x\in\R^2}$.

Furthermore, by the definition of the Wick exponential $:\exp :$
\cite[\S 3.2]{J97}, we can write
\begin{align}
\tilde u^\eps(t,x) &=
1+ \sum_{k=1}^\infty\beta_\eps^k \!\!\!\!\! \idotsint\limits_{\eps^{-2}(1-t)<t_1<\cdots <t_k<\eps^{-2}}\!\!\!\!\!\!\!\!\!\!\!\!\!\!\! \int_{\R^{2k}} \!\!\!\!\! \E_{(\eps^{-2}(1-t), \eps^{-1}x)}\Big[\prod_{i=1}^k j(B_{t_i}-x_i) \Big] \prod_{i=1}^k \widetilde W(\dd t_i, \dd x_i) \label{Feynman2} \\
&= 1+ \sum_{k=1}^\infty\beta_\eps^k\!\!\!\!\! \idotsint\limits_{\eps^{-2}(1-t)<t_1<\cdots <t_k<\eps^{-2}} \!\!\!\!\!\!\!\!\!\!\!\!\!\!\!  \int_{\R^{2k}}\Big(\int_{\R^{2k}} \prod_{i=1}^k p_{t_i-t_{i-1}}(y_i-y_{i-1}) j(y_i-x_i) \dd \vec y\Big) \prod_{i=1}^k \widetilde W(\dd t_i, \dd x_i), \nonumber
\end{align}
where $p_t(\cdot)$ is the probability density for $B_t$, $(t_0, y_0)= (\eps^{-2}(1-t), \eps^{-1}x)$, and $\dd\vec y=\dd y_1 \cdots \dd y_k$.

We are now ready to give the proofs.
\medskip

\begin{proof}[Proof of Theorem \ref{T:SHE}] Let $\hbeta<1$. We will perform a series of approximations, eventually approximating $\tilde u^\eps(t^{(i)}_\eps,x^{(i)}_\eps)$ by $Z^\eta_{\eps^{-2}, \beta_\eps}(\eps^{-1} x^{(i)}_\eps, \eps^{-2}(1-t^{(i)}_\eps))$, $1\leq i\leq r$, the partition functions of a directed polymer model on $\Z^{2+1}$. Since our approximations will be carried out in $L^2$ on the same probability space, it suffices to consider a single term $\tilde u^\eps(t_\eps^{(i)}, x^{(i)}_\eps)$, or just $\tilde u^\eps(1, 0)$.

\smallskip
\noindent
{\bf Step 1.} First we show that $\tilde u^\eps(1,0)$ can be approximated in $L^2$ if for each $k\in\N$, the integral over $\vec t:=(t_1,\ldots, t_k)$ in \eqref{Feynman2} is restricted to the set
\begin{equation}\label{Tkeps1}
T_{k,\eps} := \{(t_1, \ldots, t_k) \in (0,\eps^{-2})^k : t_i-t_{i-1} \geq \sqrt{\log \eps^{-1}} \ \forall\, 1\leq i\leq k \}.
\end{equation}
This is necessary because in the $L^2$ approximations that follow,
$\int p^2_{t_i-t_{i-1}}(y_i-y_{i-1}) \dd y_i$ is not integrable in $t_i$ due to the singularity
when $t_i$ is near $t_{i-1}$.
We thus approximate $\tilde u^\eps(1,0)$ by
\begin{equation}
v^\eps(1, 0) := 1+ \sum_{k=1}^\infty\beta_\eps^k \int_{T_{k,\eps}} \int_{\R^{2k}}\Big(\int_{\R^{2k}} \prod_{i=1}^k p_{t_i-t_{i-1}}(y_i-y_{i-1}) j(y_i-x_i) \dd \vec y\Big) \prod_{i=1}^k \widetilde W(\dd t_i, \dd x_i).
\label{vepst}
\end{equation}
By It\^o isometry and the orthogonality of terms in different orders of the Wiener chaos,
$$
\bbE[(\tilde u^\eps(1,0)- v^\eps(1,0))^2] = \sum_{k=1}^\infty\beta_\eps^{2k} \int_{T_{k,\eps}^c} \int_{\R^{2k}}\Big(\int_{\R^{2k}} \prod_{i=1}^k p_{t_i-t_{i-1}}(y_i-y_{i-1}) j(y_i-x_i) \dd \vec y\Big)^2 \dd \vec x \dd \vec t.
$$
Given $\vec t \in T_{k, \eps}^c$, let
$I(\vec t):= \{1\leq i\leq k: t_i-t_{i-1}\leq \sqrt{\log\eps^{-1}}\} \subseteq \{1,\ldots, k\}$.
We can then rewrite the integrand
$\prod_{i=1}^k p_{t_i-t_{i-1}}(y_i-y_{i-1}) j(y_i-x_i)$ above as
$$
\Big(\prod_{i\in I(\vec t)} p_{t_i-t_{i-1}}(y_i-y_{i-1}) \prod_{i\notin I(\vec t)} j(y_i-x_i)\Big) \times \Big(\prod_{i\notin I(\vec t)} p_{t_i-t_{i-1}}(y_i-y_{i-1}) \prod_{i\in I(\vec t)} j(y_i-x_i)\Big).
$$
Taking the first factor as a probability density for $\vec y$ while taking the second factor as the integrand, we can then apply Jensen's inequality to obtain the bound
\begin{align}
&\bbE[(\tilde u^\eps(1,0)- v^\eps(1,0))^2] \\
&\leq \sum_{k=1}^\infty\beta_\eps^{2k} \!\!\!\!\!\!\!\!\int\limits_{T^c_{k,\eps}\times \R^{4k}}\prod_{i\notin I(\vec t)}\!\!\! p^2_{t_i-t_{i-1}}(y_i-y_{i-1}) \!\!\prod_{i\in I(\vec t)}\!\!\! j^2(y_i-x_i) \!\!\prod_{i\in I(\vec t)}\!\!\! p_{t_i-t_{i-1}}(y_i-y_{i-1}) \!\!\prod_{i\notin I(\vec t)}\!\!\! j(y_i-x_i) \dd \vec x\dd \vec y \dd \vec t, \nonumber \\
&= \sum_{k=1}^\infty\beta_\eps^{2k} \!\!\! \int\limits_{T^c_{k,\eps}} \int\limits_{\R^{2k}} \Vert j\Vert_2^{2|I(\vec t)|} \!\! \prod_{i\notin I(\vec t)}\!\!\! p^2_{t_i-t_{i-1}}(y_i-y_{i-1}) \!\!\prod_{i\in I(\vec t)}\!\!\! p_{t_i-t_{i-1}}(y_i-y_{i-1}) \dd \vec y \dd \vec t \nonumber \\
&= \sum_{k=1}^\infty\beta_\eps^{2k} \int\limits_{T^c_{k,\eps}} \int\limits_{\R^{2k}} \Vert j\Vert_2^{2|I(\vec t)|} \!\! \prod_{i\notin I(\vec t)} \frac{e^{-\frac{| y_i-y_{i-1}|^2}{t_i-t_{i-1}}} }{4\pi^2 (t_i-t_{i-1})^2}
\prod_{i\in I(\vec t)} \frac{e^{-\frac{| y_i-y_{i-1} |^2}{2(t_i-t_{i-1})}}}{2\pi (t_i-t_{i-1})} \dd \vec y \dd \vec t \nonumber \\
&= \sum_{k=1}^\infty\beta_\eps^{2k} \int\limits_{T^c_{k,\eps}} \Vert j\Vert_2^{2|I(\vec t)|} \!\! \prod_{i\notin I(\vec t)}
\frac{1}{4\pi (t_i-t_{i-1})} \dd \vec t \nonumber \\
&\leq \sum_{k=1}^\infty \beta_\eps^{2k} \!\!\!\!\!\! \sum_{I\neq\emptyset \atop I\subset \{1, \ldots, k\}} \!\!\!\!\!\!
\big(\Vert j\Vert_2^2\sqrt{\log \eps^{-1}}\big)^{|I|} \Big(\frac{\log (\eps^{-2})}{4\pi}\Big)^{k-|I|} \nonumber \\
&= \sum_{k=1}^\infty \hbeta^{2k} \Big(\frac{2\pi}{\log \eps^{-1}}\Big)^k \Big\{ \Big(\Vert j\Vert_2^2\sqrt{\log \eps^{-1}} + \frac{\log (\eps^{-2} )}{4\pi} \Big)^k - \Big(\frac{\log (\eps^{-2} )}{4\pi}  \Big)^k \Big\} \astoo{\eps}{0} 0, \label{0thapprox}
\end{align}
where in the last step we used a binomial expansion and the last convergence follows from the dominated convergence theorem, since $\hbeta<1$.
\smallskip

\noindent
{\bf Step 2.} We next show that $v^\eps(1,0)$ can be approximated in $L^2$ by $w^\eps(1,0)$, where we replace $p_{t_i-t_{i-1}}(y_i-y_{i-1})$ in \eqref{vepst} by $p_{t_i-t_{i-1}}(x_i-x_{i-1})$, i.e.,
\begin{eqnarray}
w^\eps(1,0) &:=& 1+ \sum_{k=1}^\infty\beta_\eps^k \int_{T_{k,\eps}} \int_{\R^{2k}}\Big(\int_{\R^{2k}} \prod_{i=1}^k p_{t_i-t_{i-1}}(x_i-x_{i-1}) j(y_i-x_i) \dd \vec y\Big) \prod_{i=1}^k \widetilde W(\dd t_i, \dd x_i) \nonumber \\
&=& 1+ \sum_{k=1}^\infty\beta_\eps^k \int_{T_{k,\eps}} \int_{\R^{2k}}\prod_{i=1}^k p_{t_i-t_{i-1}}(x_i-x_{i-1}) \prod_{i=1}^k \widetilde W(\dd t_i, \dd x_i). \label{wepst}
\end{eqnarray}

Indeed,  by Jensen
\begin{eqnarray}
&& \bbE[(v^\eps(1,0) - w^\eps(1,0))^2] \nonumber\\
&\!\!\!\!=\!\!\!& \sum_{k=1}^\infty \beta_\eps^{2k} \!\! \int_{T_{k,\eps}\times\R^{2k}}\!\!\Big\{\! \int_{\R^{2k}} \!\!\Big(\prod_{i=1}^k p_{t_i-t_{i-1}}(y_i-y_{i-1})- \prod_{i=1}^k p_{t_i-t_{i-1}}(x_i-x_{i-1})\Big) \!\! \prod_{i=1}^k j(y_i-x_i) \dd \vec y\Big\}^2 \dd \vec x \dd \vec t \nonumber\\
&\!\!\!\!\leq\!\!\!& \sum_{k=1}^\infty \beta_\eps^{2k} \!\! \int_{T_{k,\eps}\times\R^{2k}\times\R^{2k}} \!\!\Big(\prod_{i=1}^k p_{t_i-t_{i-1}}(y_i-y_{i-1})- \prod_{i=1}^k p_{t_i-t_{i-1}}(x_i-x_{i-1})\Big)^2 \! \prod_{i=1}^k j(y_i-x_i) \dd \vec y\dd \vec x \dd \vec t. \qquad \quad \ \label{1stapprox}
\end{eqnarray}
To show that this bound goes to $0$ as $\eps\downarrow 0$, we will divide the integral over $\vec y$ into two parts.

Using $(a+b)^2 \leq 2(a^2+b^2)$, we can bound \eqref{1stapprox} by
\begin{eqnarray}
&&\!\!\! 2\sum_{k=1}^\infty \beta_\eps^{2k} \!\!\!\!\!\!\!\! \int\limits_{T_{k,\eps}\times\R^{2k}\times\R^{2k}} \!\!\!\!\!\!\!\Big(\prod_{i=1}^k p^2_{t_i-t_{i-1}}(y_i-y_{i-1})+ \prod_{i=1}^k p^2_{t_i-t_{i-1}}(x_i-x_{i-1})\Big) \! \prod_{i=1}^k j(y_i-x_i) \dd \vec y\dd \vec x \dd \vec t \ \quad \qquad \label{1stapprox2}\\
&=&\!\!\! 4 \sum_{k=1}^\infty \beta_\eps^{2k} \!\!\!\!\!\! \int\limits_{T_{k,\eps}\times\R^{2k}} \!\! \prod_{i=1}^k p^2_{t_i-t_{i-1}}(y_i-y_{i-1})\dd \vec y \dd \vec t
= 4 \sum_{k=1}^\infty \beta_\eps^{2k} \int\limits_{T_{k,\eps}} \prod_{i=1}^k \frac{1}{4\pi (t_i-t_{i-1})} \dd \vec t \nonumber \\
&\leq&  4 \sum_{k=1}^\infty \hbeta^{2k} \Big(\frac{2\pi}{\log \eps^{-1}}\Big)^{k} \Big(\prod_{i=1}^k \frac{\log (\eps^{-2})}{4\pi}\Big)
\astoo{\eps}{0} 4\sum_{k=1}^\infty \hbeta^{2k},  \nonumber
\end{eqnarray}
which is finite since $\hbeta<1$. Firstly, this calculation implies that it suffices to show that each summand on the r.h.s.\ of \eqref{1stapprox}, for a fixed $k\in\N$, tends to $0$ as $\eps\to 0$. Secondly, given $\vec t \in T_{k, \eps}$, if we restrict the integral  over $\vec x$ and $\vec y$ in \eqref{1stapprox2} to the set
\begin{equation}\label{Ek}
E_k=E_k(\vec t):=\{(\vec x, \vec y)\in \R^{2k}\times \R^{2k}: |y_i-y_{i-1}| \leq |t_i-t_{i-1}|^{\frac{3}{4}} \ \forall \, 1\leq i\leq k\},
\end{equation}
then it is easily seen that the $k$-th term in \eqref{1stapprox2} still converges to the same limit as $\eps\downarrow 0$, because the dominant contribution to $\int_{\R^2} p^2_t(y) \dd y$ comes from $|y|$ of the order $\sqrt{t}$. Therefore if the $k$-th integral over $\vec x$ and $\vec y$ in \eqref{1stapprox2} is restricted to $E_k^c$, then the $k$-th term in \eqref{1stapprox2}  tends to $0$ as $\eps\downarrow 0$, and the same is true for the $k$-th term in \eqref{1stapprox}.

It then only remains to consider the $k$-th term in \eqref{1stapprox}, where the integral over $\vec x$ and $\vec y$ is restricted to $E_k$. This can be bounded by
\begin{align}
& \beta_\eps^{2k} \!\!\!\!\! \int\limits_{T_{k,\eps}\times E_k}\!\! \Big(
\prod_{i=1}^k p_{t_i-t_{i-1}}(x_i-x_{i-1})
-\prod_{i=1}^k p_{t_i-t_{i-1}}(y_i-y_{i-1})
\Big)^2  \prod_{i=1}^k j(y_i-x_i) \dd \vec y\dd \vec x \dd \vec t \nonumber \\
=\ & \beta_\eps^{2k} \!\!\!\!\!\int\limits_{T_{k,\eps}\times E_k}\!\! \prod_{i=1}^k
p^2_{t_i-t_{i-1}}(y_i-y_{i-1})
\Big(e^{\sum_{i=1}^k\frac{- | x_i - x_{i-1}|^2 + |y_i - y_{i-1}|^2}{2(t_i-t_{i-1})}}
- 1 \Big)^2 \prod_{i=1}^k j(y_i-x_i) \dd \vec y\dd \vec x \dd \vec t \nonumber\\
\leq\ & \beta_\eps^{2k}\!\!\!\!\! \int\limits_{T_{k,\eps}\times E_k}\!\! \prod_{i=1}^k p^2_{t_i-t_{i-1}}(y_i-y_{i-1})
\Big(
e^{\frac{Ck}{(\log \eps^{-1})^{1/8}}} - 1
\Big)^2
\prod_{i=1}^k j(y_i-x_i) \dd \vec y\dd \vec x \dd \vec t \astoo{\eps}{0} 0, \label{3rdapprox}
\end{align}
where the convergence follows easily from the domination by \eqref{1stapprox2},
and to obtain the inequality, we argued as follows.
If we fix $A>0$ such that the support of $j(\cdot)$
is contained in a ball of radius $A$, the factor $j(y_i - x_i)$ in
\eqref{3rdapprox} entails $|y_i-x_i|\leq A$, for all $1 \le i \le k$.
Then, writing $|a|^2 - |b|^2 = |a-b|^2 + 2|b||a-b|$, we obtain
$$
\frac{|\, |x_i-x_{i-1}|^2 - |y_i-y_{i-1}|^2|}{2(t_i-t_{i-1})}
\leq \frac{(2A)^2 + 4A |y_i-y_{i-1}|}{2(t_i-t_{i-1})}
\leq \frac{C}{(\log \eps^{-1})^{\frac{1}{8}}}
$$
for any $C>0$ when $\eps>0$ is sufficiently small. Finally, we
note that $(e^x-1)^2 \le (e^{|x|}-1)^2$.

This concludes the proof that the bound in \eqref{1stapprox} tends to $0$ as $\eps\downarrow 0$, which shows that $v^\eps(1,0)$ and $w^\eps(1,0)$ have the same limiting distribution.

\medskip
\noindent
{\bf Step 3.} We now show that $w^\eps(1,0)$ can be approximated in $L^2$ by $Z^\eta_{\eps^{-2}, \beta_\eps}=Z^\eta_{\eps^{-2}, \beta_\eps}(0,0)$, the partition function of a directed polymer on $\Z^{2+1}$ starting at $(0,0)$, defined on the same probability space as $w^\eps$.

Let $(S_n)_{n\geq 0}$ be an irreducible aperiodic random walk on $\Z^2$ with $n$-step increment distribution $\hat p_n(\cdot)$, such that
$S_0=0$ and $\E[S_1(i) S_1(j)]=1_{\{i=j\}}$ for $i,j=1,2$, where $S_1(i)$ denotes the $i$-th coordinate of $S_1$. Recall from \eqref{eq:Zpoly0} that the partition function of a directed polymer model constructed from $S$ and i.i.d.\ space-time disorder $\eta$, with parameter $\beta_\eps$ and polymer length $\eps^{-2}$, admits the following polynomial chaos expansion:
\begin{equation}\label{wepstapp}
Z^\eta_{\eps^{-2}, \beta_\eps} = 1+ \sum_{k=1}^{\eps^{-2}} \beta_\eps^k \sum_{1\leq n_1<\cdots <n_k\leq \eps^{-2} \atop x_1, \ldots, x_k\in\Z^2}
\prod_{i=1}^k \hat p_{n_i-n_{i-1}}(x_i-x_{i-1}) \prod_{i=1}^k \eta_{n_i, x_i},
\end{equation}
where $(\eta_{n,x})_{n\in\N, x\in \Z^2}$ are i.i.d.\ random variables with mean $0$. For our purposes, we will let $\eta$ be i.i.d.\ standard normal variables defined from the space-time white noise $\widetilde W$ in the chaos expansion for $w^\eps(1, 0)$ in \eqref{wepst}:
\begin{equation}\label{etainWtilde}
\eta_{n,x} := \int_{\Delta_{n,x}} \widetilde W(\dd s, \dd y), \qquad n\in \N,\, x=(x(1), x(2))\in \Z^2,
\end{equation}
where $\Delta_{n,x}:=[n-1, n]\times [x(1)-1, x(1)]\times [x(2)-1, x(2)]$.

We can then rewrite \eqref{wepstapp} as
\begin{equation}\label{peiepstapp2}
Z^\eta_{\eps^{-2}, \beta_\eps} = 1+ \sum_{k=1}^{\infty} \beta_\eps^k \!\!\!\!\idotsint\limits_{0<t_1<\cdots<t_k<\eps^{-2}}\!\!\! \int_{(\R^2)^k}
\prod_{i=1}^k \hat p_{\lceil t_i\rceil -\lceil t_{i-1}\rceil}(\lceil x_i\rceil -\lceil x_{i-1}\rceil) \prod_{i=1}^k \widetilde W(\dd t_i, \dd x_i),
\end{equation}
where we set $\hat p_0(\cdot)\equiv 0$.

By Gnedenko's
local limit theorem (see e.g.~\cite[Theorem 8.4.1]{cf:BinGolTeu})
\begin{equation}\label{lclt}
\hat p_n(x) = \frac{1}{2\pi n} \Big(e^{-\frac{| x|^2}{2n}}+o(1)\Big)=p_n(x)+o\Big(\frac{1}{n}\Big) \quad \mbox{ uniformly in } x\in\Z^2 \mbox{ as } n\to\infty,
\end{equation}
where we recall that $p_n(x)$ in the right hand side is the transition kernel of Brownian motion.
By similar calculations as those leading to \eqref{0thapprox}, we can restrict the integral over $(t_1, \ldots, t_k)$ to $T_{k,\eps}$ as in the definition of $v^\eps(1,0)$ and $w^\eps(1,0)$ in \eqref{vepst} and \eqref{wepst}, i.e., if
\begin{equation}\label{peiepstapp3}
\tilde Z^\eta_{\eps^{-2}, \beta_\eps} := 1+ \sum_{k=1}^{\infty} \beta_\eps^k \int_{T_{k,\eps}\times (\R^2)^k}
\prod_{i=1}^k \hat p_{\lceil t_i\rceil -\lceil t_{i-1}\rceil}(\lceil x_i\rceil -\lceil x_{i-1}\rceil) \prod_{i=1}^k \widetilde W(\dd t_i, \dd x_i),
\end{equation}
then $\bbE[(\tilde Z^\eta_{\eps^{-2}, \beta_\eps}- Z^\eta_{\eps^{-2}, \beta_\eps})^2]\to 0$ as $\eps\downarrow 0$.

We can now bound the $L^2$ distance between $w^\eps(1,0)$ and $\tilde Z^\eta_{\eps^{-2}, \beta_\eps}$:
\begin{eqnarray}
&& \bbE[(w^\eps(1,0) - \tilde Z^\eta_{\eps^{-2}, \beta_\eps})^2] \nonumber\\
&\!\! =&\!\!\! \sum_{k=1}^\infty \beta_\eps^{2k} \int_{T_{k,\eps}\times \R^{2k}}\!\!\Big(\prod_{i=1}^k p_{t_i-t_{i-1}}(x_i-x_{i-1})- \prod_{i=1}^k \hat p_{\lceil t_i\rceil -\lceil t_{i-1}\rceil}(\lceil x_i\rceil -\lceil x_{i-1}\rceil)\Big)^2 \dd \vec x \dd \vec t, \qquad \quad \label{2ndapprox}
\end{eqnarray}
and we will separate the integration over $\vec x$ into two sets for each $k\in\N$.

Given $\vec t \in T_{k, \eps}$ and $L>0$, let
\begin{equation}\label{EkL}
E_{k,L}=E_{k,L}(\vec t):=\{\vec x \in \R^{2k} : |x_i-x_{i-1}| \leq L\sqrt{t_i-t_{i-1}} \ \forall \, 1\leq i\leq k\}.
\end{equation}
By the same calculations as for \eqref{1stapprox2}, we note that when the integrals over $\vec x$ in \eqref{2ndapprox} are restricted to $E^c_{k,L}$ for each $k\in\N$, the resulting series converges to a limit (as $\eps\downarrow 0$) that can be made arbitrarily small by choosing $L$ large. On the other hand, for any fixed $L>0$,
\begin{align}\label{lcltunibd}
\mbox{uniformly in } \ \vec t\in T_{k,\eps} \mbox{ and } \vec x\in E_{k,L}, \qquad
1-\prod_{i=1}^k \frac{\hat p_{\lceil t_i\rceil -\lceil t_{i-1}\rceil}(\lceil x_i\rceil -\lceil x_{i-1}\rceil)}{p_{t_i-t_{i-1}}(x_i-x_{i-1})} \astoo{\eps}{0} 0
\qquad
\end{align}
by the local central limit theorem \eqref{lclt}. Therefore when the integrals over $\vec x$ in \eqref{2ndapprox} are restricted to $E_{k,L}$, the resulting series also tends to $0$ as $\eps\downarrow 0$, as in \eqref{3rdapprox}.
In conclusion, the series in \eqref{2ndapprox} tends to $0$ as $\eps\downarrow 0$ and we can approximate $w^\eps(1, 0)$ by $Z^\eta_{\eps^{-2}, \beta_\eps}$.

\medskip
\noindent
{\bf Step 4.} When $\hbeta<1$, we just showed that each $\tilde u^\eps(t^{(i)}_\eps,x^{(i)}_\eps)$ can be approximated in $L^2$ by $Z^\eta_{\eps^{-2}, \beta_\eps}(\eps^{-1} x^{(i)}_\eps, \eps^{-2}(1-t^{(i)}_\eps))$. Identifying $\eps^{-2}$ with $N$, the convergence in Theorem~\ref{T:SHE} then follows by applying Theorem \ref{thm:short-cor} to $Z^\eta_{\eps^{-2}, \beta_\eps}(\eps^{-1} x^{(i)}_\eps, \eps^{-2}(1-t^{(i)}_\eps))$ for $1\leq i\leq r$.
\vskip 4mm
For $\hbeta\geq 1$, the proof for $u^\eps(1,0)\Rightarrow 0$ as $\eps\downarrow 0$ is the same as that
for the pinning and directed polymer models in Theorem~\ref{thm:subcritical}. Proving that
for $0<\theta<1$ the quantity $\bbE\big[\big(\tilde u^\eps(t,0)\big)^\theta\big]$
is decreasing in $\hbeta$
is even simpler and proceeds as follows. Note that by \eqref{Feynman1.1},
\begin{equation*}
	\tilde u^\eps(t,0) = \E_0\Big[ e^{\beta\int\int j(B_s-y) W(\dd s, \dd y)- \frac{\beta^2}{2}\Vert
	j\Vert_{L^2(\bbR\times\bbR^d)}^2 }\Big] =:
	\E_0 \Big[ e^{\beta G^W(B) - \frac{\beta^2}{2} \bbvar[G^W(B)]} \Big] \,,
\end{equation*}
where, for a fix realization of $B = (B_s)_{s \ge 0}$, the random variable
$G^W(B)$ has a centered Gaussian distribution
with variance $\Vert j\Vert_{L^2(\bbR\times\bbR^d)}^2$. For
$\beta^2=\beta_1^2+\beta_2^2$ we can write
$$
\tilde u^\eps(t,0)
\overset{d}{=}
\E_0 \Big[ e^{\beta_1 G^{W_1}(B) - \frac{\beta_1^2}{2} \bbvar[G^{W_1}(B)]}
\, e^{\beta_2 G^{W_2}(B) - \frac{\beta_2^2}{2} \bbvar[G^{W_2}(B)]} \Big]
$$
where $W_1, W_2$ are two independent space-time white noise, and we used the fact that $\beta W\stackrel{\rm dist}{=}\beta_1 W_1 +\beta_2 W_2$. Using Jensen's inequality to pass the expectation w.r.t.\ $W_2$ inside the fractional
root in $\E[(\tilde u^\eps(t,0))^\theta]$ then gives the desired monotonicity in $\beta$ as well as $\hbeta$, since the two differ by a constant factor.
\end{proof}
\smallskip

\begin{proof}[Proof of Theorem \ref{T:SHEfield}] Since $(\tilde u^\eps(t,x))_{t\in [0,1], x\in \R^2}$ defined in \eqref{FKrec} has the same distribution as $(u^\eps(t, x))_{t\in [0,1], x\in\R^2}$,
we note that $J^\psi_\eps$ in \eqref{SHEfield} has the same distribution as
$$
\tilde J^\psi_\eps := \Big(\frac{\log \eps^{-1}}{2\pi}\Big)^{1/2} \int_{\R^2\times [0,1]} (\tilde u^\eps(t,x)-1)\, \psi(x,1-t)\, \dd t \dd x.
$$
Applying the chaos expansion \eqref{Feynman2} with $t_0=\eps^{-2}(1-t)$ and $x_0=y_0=\eps^{-1}x$ gives
\begin{align}
\tilde J^\psi_\eps
&=\sum_{k=1}^\infty\hbeta^k \eps^4 \Big(\frac{2\pi}{\log \eps^{-1}}\Big)^{\frac{k-1}{2}}   \label{hatJpsi}\\
& \qquad \times \!\!\!\!\!\! \idotsint\limits_{x_1, \ldots, x_k\in \R^2 \atop
0<t_1<\cdots <t_k<\eps^{-2}} \!\!\!\!\! \Big(\int\limits_{\R^{2k}}p^{\,\eps, \psi}_{t_1}(y_1) \prod_{i=2}^k p_{t_i-t_{i-1}}(y_i-y_{i-1}) j(y_i-x_i) \dd \vec y\Big)
\prod_{i=1}^k \widetilde W(\dd t_i, \dd x_i), \notag
\end{align}
where we have plugged in $\beta_\eps=\hbeta \big(\frac{2\pi}{\log \eps^{-1}}\big)^{1/2}$,
changed variables $(t,x) = (1-\epsilon^2 t_0, \epsilon y_0)$ producing
the pre-factor $\eps^4$, and
\begin{align}\label{pepspsidef}
p_{t_1}^{\,\eps, \psi}(y_1) & := \int_{\R^2\times [0,t_1]}
\psi(\eps y_0, \eps^2 t_0)\  p_{t_1-t_0}(y_1-y_0)\ \dd y_0 \dd t_0.
\end{align}
Note that \eqref{hatJpsi} has the same form as the expansion for $\tilde u^\eps(1,0)$ in \eqref{Feynman2}, except that the factor $p_{t_1}(y_1)$ therein is now replaced by $p_{t_1}^{\,\eps, \psi}(y_1)$, and a pre-factor of $\big(\frac{2\pi}{\log \eps^{-1}}\big)^{1/2}$ has been replaced by $\eps^4$. We can now carry out the same steps as in the proof of Theorem \ref{T:SHE} to approximate $\tilde J^\psi_\eps$ by $J^{\psi}_N$ for a directed polymer on $\Z^2$, so that Theorem \ref{thm:field} can be applied.
\smallskip

In {\bf Step 1} of the proof of Theorem \ref{T:SHE}
we restricted the range of integration of $t_1, \ldots, t_k$ in \eqref{Feynman2} so that $t_i-t_{i-1}\geq \sqrt{\log \eps^{-1}}$ for all $1\leq i \leq k$. This is necessary because in the $L^2$ approximations that follow, $\int p^2_{t_i-t_{i-1}}(y_i-y_{i-1}) \dd y_i$ is not integrable in $t_i$ due to the singularity when $t_i$ is near $t_{i-1}$.There is no such singularity for $p^{\,\eps, \psi}_{t_1}(y_1)$, and in fact,
\begin{align}\label{pepspsi0}
& \Vert p^{\,\eps,\psi}\Vert_2^2 := \int\limits_{\R^2\times [0, \eps^{-2}]} (p^{\,\eps, \psi}_{t_1}(y_1))^2 \dd y_1 \dd t_1 \\
=\ &  \int\limits_{(\R^2)^3} \!\!
\int\limits_{t_1\in (0, \eps^{-2}) \atop t_0, t'_0\in (0, t_1)} \!\!\!\!\!\!\!\!\!
\psi(\eps y_0, \eps^2 t_0)\, \psi(\eps y'_0, \eps^2 t'_0) \, p_{t_1-t_0}(y_1-y_0) \, p_{t_1-t'_0}(y_1-y'_0)\ \dd y_0 \dd t_0 \dd y'_0 \dd t'_0 \dd y_1 \dd t_1 \notag
\end{align}
can be bounded in the same way as its discrete counterpart $q^{N, \psi}_{M, n_1}(x_1)$ from \eqref{qpsi1} (see the variance calculations in \eqref{ThePsiVar}--\eqref{limtoVpsi}
and \eqref{eq:usefu},
with $N=\eps^{-2}$, $\phi(N)=\eps^{-1}$), which gives $\Vert p^{\,\eps, \psi}\Vert_2 \sim C\eps^{-4}$ as $\eps\to 0$. Therefore in our current setting, we need to restrict $t_1, \ldots, t_k$ to
\begin{equation}\label{Tkeps2}
\tilde T_{k,\eps} := \{(t_1, \ldots, t_k) \in (0,\eps^{-2})^k : t_1>t_0, \ t_i-t_{i-1} \geq \sqrt{\log \eps^{-1}} \ \, \forall\, 2\leq i\leq k \}. 
\end{equation}
The rest of the calculations in {\bf Step 1} carry through once we take into account that $\Vert p^{\,\eps, \psi}\Vert_2$ is of the order $\eps^{-4}$, which cancels the pre-factor $\eps^4$ in \eqref{hatJpsi}.

In {\bf Step 2} of the proof of Theorem \ref{T:SHE}, we replaced $p_{t_i-t_{i-1}}(y_i-y_{i-1})$ by $p_{t_i-t_{i-1}}(x_i-x_{i-1})$ for each $1\leq i\leq k$, using the fact that $y_i-x_i$ must lie in the support of $j(\cdot)$. The same applies here, except that we also need to replace $p^{\,\eps, \psi}_{t_1}(y_1)$ by $p^{\,\eps, \psi}_{t_1}(x_1)$.

More precisely, we
can first apply the same calculations as in \eqref{1stapprox}--\eqref{3rdapprox} to replace $p_{t_i-t_{i-1}}(y_i-y_{i-1})$ by $p_{t_i-t_{i-1}}(x_i-x_{i-1})$ for each $2\leq i\leq k$. The only change we need to make is to redefine the set $E_k$ in \eqref{Ek} by
\begin{equation}\label{Ek2}
\tilde E_k:=\{(\vec x, \vec y)\in \R^{2k}\times \R^{2k}: |y_i-y_{i-1}| \leq |t_i-t_{i-1}|^{\frac{3}{4}} \ \forall \, 2\leq i\leq k\}.
\end{equation}
After making these replacements, it  only remains to bound the following simpler analogue of \eqref{3rdapprox}:
\begin{align}
& \eps^8 \Big(\frac{2\pi}{\log \eps^{-1}}\Big)^{k-1}  \!\!\!\!\! \int\limits_{\tilde T_{k,\eps}\times \tilde E_k}\!\!
(p^{\,\eps, \psi}_{t_1}(y_1)- p^{\,\eps, \psi}_{t_1}(x_1))^2 \prod_{i=2}^k p^2_{t_i-t_{i-1}}(x_i-x_{i-1}) \prod_{i=1}^k j(y_i-x_i) \dd \vec y\dd \vec x \dd \vec t \nonumber \\
\leq\ & C \eps^8  \!\!\!\!\! \int\limits_{\R^4 \times [0, \eps^{-2}]}\!\!
(p^{\,\eps, \psi}_{t_1}(y_1)- p^{\,\eps, \psi}_{t_1}(x_1))^2  j(y_1-x_1) \dd y_1\dd x_1 \dd t_1 \astoo{\eps}{0} 0,
\end{align}
where in the inequality, we applied the same calculations as for \eqref{1stapprox2}, and the convergence to $0$ can be easily deduced from the definition of $p^{\,\eps, \psi}_t(y)$ in \eqref{pepspsidef}.

We can therefore approximate $\tilde J^\psi_\eps$ in $L^2$ by
\begin{align}\label{whatJpsi}
\widehat J^\psi_\eps :=
\sum_{k=1}^\infty\hbeta^k \eps^4 \Big(\frac{2\pi}{\log \eps^{-1}}\Big)^{\frac{k-1}{2}}
\!\! \idotsint\limits_{\tilde T_{k,\eps}\times \R^{2k}} p^{\,\eps, \psi}_{t_1}(x_1)
\prod_{i=2}^k p_{t_i-t_{i-1}}(x_i-x_{i-1})  \prod_{i=1}^k \widetilde W(\dd t_i, \dd x_i).
\end{align}

Following {\bf Step 3} in the proof of Theorem \ref{T:SHE}, we now introduce a directed polymer on $\Z^2$, where the $n$-step transitional kernel $\hat p_n(\cdot)$ of the underlying random walk $S$ satisfies the local limit theorem in \eqref{lclt}, so that Hypothesis \ref{hypoth} holds with $L(\cdot)\equiv 1$. As in \eqref{J000}, and with $N=\eps^{-2}$ and $\phi(N)=(L(N)^2N)^{1/2}=\eps^{-1}$, we can define
\begin{align}
J^{\psi}_N & = \frac{L(N)}{\phi(N)^2N} \sum_{k=1}^\infty \frac{\hbeta^k}{\Ro_N^{(k-1)/2}}
\sumtwo{1\leq n_1<\cdots<n_k\leq N}{x_1, x_2, \ldots, x_k \in \bbZ^d} \!\!\!\!\!\!\!\!\!\! \hat p^{N,\psi}_{n_1}(x_1)
	\, \prod_{j=2}^k \hat p_{n_j-n_{j-1}}(x_j-x_{j-1})
	\, \prod_{i=1}^k \eta_{(n_i,x_i)}  \notag \\
& =  \sum_{k=1}^\infty \hbeta^k \eps^4\Big(\frac{2\pi+o(1)}{\log \eps^{-1}}\Big)^{\frac{k-1}{2}} \!\!\!\!\!\!\!\!\!\!\!\!
\sumtwo{1\leq n_1<\cdots<n_k\leq \eps^{-2}}{x_1, x_2, \ldots, x_k \in \bbZ^d} \!\!\!\!\!\!\!\!\!\!\!\!\!\!\! \hat p^{N,\psi}_{n_1}(x_1)
	\, \prod_{j=2}^k \hat p_{n_j-n_{j-1}}(x_j-x_{j-1})
	\, \prod_{i=1}^k \eta_{(n_i,x_i)} \label{JtildepsiN}
\end{align}
where we used $\Ro_N=\Ro_{\eps^{-2}} \sim \log \eps^{-1}/2\pi$ by \eqref{eq:RN}, and
\begin{align}
\hat p^{N,\psi}_{n_1}(x_1) := \!\!\!\!\! \sum_{x_0\in\Z^d, 0<n_0<n_1} \!\!\!\!\! \psi(\eps x_0, \eps^2 n_0)\ \hat p_{n_1-n_0}(x_1-x_0).
\end{align}
Note that $\hat p^{N,\psi}_{n_1}(x_1)$ is a discrete sum approximation of $p_{t_1}^{\,\eps, \psi}(x_1)$ in \eqref{pepspsidef},
and if we let $(\eta_{(n,x)})_{n\in\N, x\in\Z^2}$ be Gaussian random variables defined from $\widetilde W$ as in \eqref{etainWtilde}, then $J^{\psi}_N$ in \eqref{JtildepsiN} is just a discrete sum approximation of $\widehat J^\psi_\eps$ in \eqref{whatJpsi}. The $L^2$ difference $\Vert J^{\psi}_N-\widehat J^\psi_\eps\Vert_2^2$ can be shown to vanish as $\eps\to 0$, similar to \eqref{peiepstapp3}--\eqref{lcltunibd}, and we will omit the details.

Similar to {\bf Step 4} in the proof of Theorem \ref{T:SHE}, we can finally
apply Theorem \ref{thm:field} (for the directed polymer on $\Z^2$) to $J^{\psi}_N$ and conclude
that $J^{\psi}_N$, and hence also $J^\psi_\eps$, converge in distribution to a Gaussian random
variable with mean zero and variance $\sigma_{\psi}^2$.
\end{proof}

\section*{Acknowledgements}

We thank Massimiliano Gubinelli for useful discussions.
We also thank Giulia Comi for spotting some mistakes in the draft.
F.C. acknowledges the support of GNAMPA-INdAM.
R.S. is supported by NUS grant R-146-000-185-112.
N.Z. is supported by EPRSC through grant EP/L012154/1.

%%%%%%%%%%%%%%%%%%%%%%%%%%%%%%%%%%%%%%%%%%%%%%%%%%%%%%%%%%%%%%%%%%%%%%%%%%%%%%
%%%%%%%%%%%%%%%%%%%%%%%%% The bibliography %%%%%%%%%%%%%%%%%%%%%%%%%%%%%%%%%%%
%%%%%%%%%%%%%%%%%%%%%%%%%%%%%%%%%%%%%%%%%%%%%%%%%%%%%%%%%%%%%%%%%%%%%%%%%%%%%%

\end{document}